\documentclass{amsart}
\usepackage{amsfonts, amssymb, amsmath, verbatim, amsthm, amscd, enumerate}
 
 \usepackage{enumitem}
 \usepackage{color,umoline}
\usepackage[dvipsnames]{xcolor}
\usepackage{graphicx}
\usepackage[hidelinks, %backref=page
]{hyperref}
\usepackage{tikz}
\usetikzlibrary{calc,trees,shadows,positioning,arrows,chains,shapes.geometric,
decorations.pathreplacing,decorations.pathmorphing,shapes,
matrix,shapes.symbols,patterns,intersections,fit}

\newtheorem{theorem}{Theorem}[section]
\newtheorem{lemma}[theorem]{Lemma}
\newtheorem{proposition}[theorem]{Proposition}

\theoremstyle{definition}
\newtheorem{definition}[theorem]{Definition}

\theoremstyle{remark}
\newtheorem{remark}{Remark}

\newtheorem*{organisation}{Organisation}
\newtheorem*{acknowledgements}{Acknowledgements}

\numberwithin{equation}{section}

\newcommand{\Be}{\begin{equation}}
\newcommand{\Ee}{\end{equation}}
\newcommand{\Bel}{\begin{align}}
\newcommand{\Eel}{\end{align}}

\newcommand{\tq}{\widetilde q}
\newcommand{\tr}{\widetilde r}
\newcommand{\tp}{\widetilde p}
\def\1{\textbf{\rm 1}}

\newcommand{\jb}[1]{\langle#1\rangle}

\begin{document}

\author[Bez]{Neal Bez}
\address[Neal Bez]{Department of Mathematics, Graduate School of Science and Engineering,
Saitama University, Saitama 338-8570, Japan}
\email{nealbez@mail.saitama-u.ac.jp}
\author[Lee]{Sanghyuk Lee}
\address[Sanghyuk Lee]{Department of Mathematical Sciences and RIM, Seoul National University, Seoul 151-747, Korea}
\email{shklee@snu.ac.kr}
\author[Nakamura]{Shohei Nakamura}
\address[Shohei Nakamura]{Department of Mathematics and Information Sciences, Tokyo Metropolitan University,
1-1 Minami-Ohsawa, Hachioji, Tokyo, 192-0397, Japan}
\email{nakamura-shouhei@ed.tmu.ac.jp}

\title[Strichartz estimates for orthonormal systems ]{Strichartz estimates for orthonormal families of initial data and weighted
oscillatory \\ integral estimates}

\begin{abstract} 
We establish new Strichartz estimates for orthonormal families of initial data in the case of the wave, Klein--Gordon and fractional Schr\"odinger equations.  Our estimates extend those of Frank--Sabin in the case of the wave and Klein--Gordon equations, and generalize work of Frank--Lewin--Lieb--Seiringer and Frank--Sabin for the Schr\"{o}dinger equation.   Due to a certain technical barrier,  except for the classical Schr\"odinger equation, the Strichartz estimates for orthonormal families of initial data  have  not  previously been established up to the sharp summability exponents  in the full range of admissible pairs. We obtain the optimal estimates in various notable cases and improve the previous results.  The main novelty of this paper is the use of  estimates for weighted oscillatory integrals which we combine with an approach due to Frank and Sabin.  This strategy also leads us to proving new estimates for weighted oscillatory integrals with optimal decay exponents which we believe to be of wider independent interest. Applications to the theory of infinite systems of Hartree type, weighted velocity averaging lemmas for kinetic transport equations, and refined Strichartz estimates for data in Besov spaces are also provided.
\end{abstract}

\maketitle

\section{Introduction}

For a given dispersion relation $\phi$, the function $U_\phi f$ denotes the solution to the initial value problem   
 \[
 \begin{cases}
 i\partial_t u+\phi(D)u=0,\quad \quad (t,x)\in\mathbb{R}^{1+d}, \ \  d\ge1,\\
 u(0,\cdot)=f.\\
 \end{cases}
 \]
This paper is concerned with extended versions of Strichartz estimates taking the form
\begin{equation} \label{e:ONSgeneral}
\bigg\| \sum_j  \nu_j |U_\phi f_j| ^2 \bigg\|_{L^{\frac q2}_tL^{\frac r2}_x} \lesssim \| \nu \|_{\ell^\beta}
\end{equation}
for families of orthonormal functions $(f_j)_j$ in a given Hilbert space $\mathcal{H}$, which we shall take to be homogeneous or inhomogeneous Sobolev spaces. This particular line of investigation originated in recent work of Frank \emph{et al.} \cite{FLLS} for the Schr\"odinger propagator ($\phi(\xi) = |\xi|^2$). The idea to generalize classical inequalities from a single-function input to an orthonormal family traces back further, with pioneering work of Lieb--Thirring \cite{LiebThirring} establishing extended versions of certain Gagliardo--Nirenberg--Sobolev inequalities and applications to the stability of matter. We also mention Lieb's extended version of Sobolev inequalities in \cite{Lieb_Sobolev} which will be of use to us in the current paper. 

Motivation to study estimates of the form \eqref{e:ONSgeneral} is plentiful. Applications to the Hartree equation modelling infinitely many fermions in a quantum system can be found in work of Chen--Hong--Pavlovic \cite{CHP-1, CHP-2}, Frank--Sabin \cite{FS_AJM}, Lewin--Sabin \cite{LewinSabin-1, LewinSabin-2}, and Sabin \cite{Sabin_survey}. Physically, the quantity $\sum_{j=1}^N  |e^{it\Delta} f_j| ^2$ gives a representation of the density of a quantum system of $N$ fermions (at time $t$), where $f_j$ represents the $j$th fermion at the initial time, and thus it is desirable to have optimal control on such a quantity in terms of the number of particles $N$; this corresponds to obtaining bounds of the form \eqref{e:ONSgeneral} with $\beta$ as large as possible. Other applications include consequences for the wave operator for time-dependent potentials, which may be found in \cite{FLLS}. In a somewhat different direction, one may obtain refined versions of the classical (single-function) Strichartz estimates for data in certain Besov spaces rather quickly from \eqref{e:ONSgeneral} via the Littlewood--Paley inequality; this observation may be found in \cite{FS_Survey}, and provides an approach to refined Strichartz inequalities which is distinct and rather simpler than the more well-known approach via bilinear Fourier (adjoint) restriction estimates. In addition to the papers cited already, we refer the reader forward to Section \ref{section:applications} where we present several applications of the nature described above; these applications will be consequences of the new estimates we obtain in the current paper and to which we now begin to focus.

The classical (single-function) Strichartz estimates enjoy a rather general theory which allows for a fairly unified approach to a wide class of dispersive equations such as the Schr\"odinger equation $\phi(\xi) = |\xi|^2$, or more generally the fractional Schr\"odinger equation $\phi(\xi) = |\xi|^{\alpha}$ ($\alpha \neq 0,1$), the wave equation $\phi(\xi)=|\xi|$,  and the Klein--Gordon equation $\phi(\xi) = \langle \xi \rangle$, where $\langle \xi \rangle := (1 + |\xi|^2)^{1/2}$. In the extended framework \eqref{e:ONSgeneral}, matters are significantly more complicated and a comparable general theory seems rather distant. In this paper we provide substantial progress in this direction, generalizing work in \cite{FLLS} to the fractional Schr\"odinger case, and significantly extending the results in \cite{FS_AJM} for the wave and Klein--Gordon equations. Our approach is, to a certain extent, based on an abstract framework and we expect that similar results can be obtained by our approach for broader classes of dispersion relations. To facilitate the presentation of prior results and our new results in the extended framework \eqref{e:ONSgeneral}, we first review the classical (single-function) case.

\subsection{Classical Strichartz estimates}  
Let $\mathcal{H}$ be a Hilbert space, such as $\dot{H}^s$ or $H^s$, the homogeneous or inhomogeneous Sobolev spaces of order $s$ built over $L^2(\mathbb{R}^d)$.  (For the definitions of $\dot{H}^s$,  $H^s$, and  $\phi(D)$ we refer the reader forward to Section \ref{section:Prelims}.) The classical Strichartz estimates usually take the form
\begin{equation} 
\label{e:classicalStrichartz}
\| U_\phi f\|_{L^q_tL^r_x} \lesssim \|f\|_\mathcal{H}.
\end{equation}
Since $\phi$ is typically smooth away from the origin, 
a commonly used technique to prove the estimate  \eqref{e:classicalStrichartz} is  first  to consider the case that the Fourier (frequency) support of the initial data $f$   is localized to an annulus and then apply Littlewood--Paley theory to extend to general data in the class $\mathcal{H}$. This strategy allows us to avoid complication which arises from singular behavior of the dispersive relation $\phi$ near the origin and at the infinity. For initial data with such localized frequency, one can prove estimates of the form \eqref{e:classicalStrichartz} if a dispersive estimate such as
\begin{equation} \label{e:dispersiveclassical}
\sup_{x \in \mathbb{R}^d} \bigg| \int_{\mathbb{R}^d} e^{i(x \cdot \xi + t\phi(\xi))} \chi_0(|\xi|) \, \mathrm{d}\xi \bigg| \lesssim  (1 + |t|)^{-\sigma}
\end{equation}
holds, where $\chi_0$ is a bump function with support away from zero and $\sigma > 0$. More precisely, for $\sigma > 0$, we say that $(q,r) \in [2,\infty) \times [2,\infty)$ is \emph{$\sigma$-admissible} if
\begin{equation*}
\frac{1}{q} \leq \sigma\bigg(\frac{1}{2} - \frac{1}{r}\bigg).
\end{equation*}
In the case of equality
\begin{equation*}
\frac{1}{q} = \sigma\bigg(\frac{1}{2} - \frac{1}{r}\bigg),
\end{equation*}
we say that $(q,r)$ is \emph{sharp $\sigma$-admissible}, and in the case of strict inequality
\begin{equation*}
\frac{1}{q} < \sigma\bigg(\frac{1}{2} - \frac{1}{r}\bigg)
\end{equation*}
we say that $(q,r)$ is \emph{non-sharp $\sigma$-admissible}. 
It follows from work of Keel--Tao \cite{KeelTao} that \eqref{e:classicalStrichartz} holds  for frequency localized $f$ whenever we have the dispersive estimate \eqref{e:dispersiveclassical} and $(q,r)$ is $\sigma$-admissible. 

We remark that pairs $(q,\infty)$ and $(\infty,r)$ are not included in our  definition of admissible pairs. Estimates of the form \eqref{e:classicalStrichartz} are available for certain pairs of this type; indeed, such an estimate obviously holds if $\mathcal{H} = L^2$ and $(q,r) = (\infty,2)$. The case $r = \infty$, in particular, requires special attention. For example, in the case of the wave equation $\phi(\xi) = |\xi|$, it was shown by Fang--Wang \cite{FangWang} that \eqref{e:classicalStrichartz} fails with $\mathcal{H} = \dot{H}^\frac34$ when $(q,r,d) = (4,\infty,2)$, and it was shown in \cite{MontSmith} and \cite{GLNY} that \eqref{e:classicalStrichartz} fails with $\mathcal{H} = \dot{H}^{\frac{d-1}{2}}$ when $(q,r) = (2,\infty)$, for $d=3$ and $d\ge4$, respectively. For the Schr\"odinger equation $\phi(\xi) = |\xi|^2$, the situation is slightly different; \eqref{e:classicalStrichartz} holds with $\mathcal{H} = L^2$ when $(q,r,d) = (4,\infty,1)$ and fails with $\mathcal{H} = \dot{H}^{\frac{d-2}{2}}$ when $(q,r) = (2,\infty)$ for $d\ge2$ (see \cite{MontSmith} and \cite{GLNY}). 

\subsection*{The fractional Schr\"odinger equation} 
First, we consider the case of the fractional Schr\"odinger equation $\phi(\xi) = |\xi|^\alpha$, where $\alpha \neq 0,1$:
\begin{equation}\label{e:fractiona}
i\partial_t u = (-\Delta)^{\alpha/2}u,\qquad u(0) = u_0. 
\end{equation} 
{If $\alpha=2$, this corresponds to the well-known Schr\"{o}dinger equation. The fractional Schr\"{o}dinger equation  \eqref{e:fractiona} with $\alpha\in(0,2)\setminus\{1\}$ appears in \cite{Laskin1,Laskin2}  as a consequence of generalizing Feynman's path integral  
arising in Brownian motion to the one generated by L\'{e}vy motion, as well as in \cite{IP} as a model of water waves. For further studies on the fractional Schr\"{o}dinger equation, we refer the reader to \cite{CS,CKS,GuoWang,HongSire,Ke,KenigPonceVega} and for the case $\alpha=4$, to which special attention has been paid, see  \cite{Karpman,KS,Pausader1,Pausader2}.

If $\alpha \in \mathbb{R} \setminus \{0,1\}$, for $(q,r)$ which are $\tfrac{d}{2}$-admissible, the classical Strichartz estimate
\begin{equation}\label{e:FracStri}
\| e^{it(-\Delta)^{\alpha/2}} f \|_{L^q_tL^r_x} \lesssim \|f\|_{\dot{H}^s},   \   \   s = \tfrac{d}{2} - \tfrac{d}{r} - \tfrac{\alpha}{q}
\end{equation}
holds
(see,  for example, \cite{COX}). For $s \leq -\tfrac{d}{2}$, the space $\dot{H}^s$ does not admit natural classes of dense functions such as the Schwartz class, so we restrict our attention to the case where $\tfrac{d}{r} + \tfrac{\alpha}{q} < d$. Note that for the classical Schr\"odinger equation with $\alpha = 2$, we have $s \in [0,\frac{d}{2})$ whenever $(q,r)$ is $\frac{d}{2}$-admissible, so the additional assumption $\frac{d}{r} + \frac{\alpha}{q} < d$ is automatically satisfied in this case.

\subsection*{The wave equation} 
It is well-known that the solution of the  wave equation
\[
\partial_{tt} u = \Delta u, \qquad (u(0),\partial_tu(0)) = (u_0,u_1)
\]
can be written as $u = u_+ + u_-$, where $u_+$ and $u_-$ are given by
\[
u_\pm(x,t) = e^{\pm i t \sqrt{-\Delta}}f_\pm(x), 
\]
and $f_+$ and $f_-$ satisfy
$
 (f_+ + f_-,i\sqrt{-\Delta}(f_+ - f_-))=(u_0,u_1).
$
As a result, the Strichartz estimates for the wave equation are  usually  given by those for the one-sided propagator $e^{\pm it\sqrt{-\Delta}}$.
For $d \geq 2$, it is known that the estimate
\begin{equation} \label{e:StrWave}
\| e^{it\sqrt{-\Delta}}f \|_{L^q_tL^r_x} \lesssim \|f\|_{\dot{H}^s},  \   \  s = \tfrac{d}{2} - \tfrac{d}{r} - \tfrac{1}{q}
\end{equation} 
holds if $(q,r)$ is $\frac{d-1}{2}$-admissible. By an elementary scaling argument, 
one may easily verify that \eqref{e:StrWave} fails for other values of $s$. We also remark that if $(q,r)$ is  sharp $\frac{d-1}{2}$-admissible, then we have 
$s = \tfrac{d+1}{2}\big(\tfrac{1}{2}-\tfrac{1}{r}\big).$  
In the case where $(q,r)$ is sharp $\frac{d-1}{2}$-admissible and $q=r$, the Strichartz estimate \eqref{e:StrWave} becomes
\begin{equation} \label{e:waveST}
\| e^{it\sqrt{-\Delta}}f \|_{L^{\frac{2(d+1)}{d-1}}_{x,t}} \lesssim \|f\|_{\dot{H}^{1/2}},
\end{equation}
which basically  corresponds to the Stein--Tomas adjoint restriction estimate for the cone.

\subsection*{The Klein--Gordon equation}
In a similar manner to the wave equation, the solution of the Klein--Gordon equation
\[
\partial_{tt} u + u = \Delta u, \qquad (u(0),\partial_tu(0)) = (u_0,u_1)
\]
decomposes into a sum of  waves which are given by the propagators $e^{\pm it\sqrt{1-\Delta}}$   and therefore the Strichartz estimates usually take the form
\begin{equation} \label{e:StrKG}
\| e^{it\sqrt{1-\Delta}}f \|_{L^q_tL^r_x} \lesssim \|f\|_{H^s}.
\end{equation} 
The propagator $e^{ it\sqrt{1-\Delta}}$ possesses properties of a mixed nature. With high frequency initial data, the behavior resembles the wave propagator $e^{i t \sqrt{-\Delta}}$, whereas 
the mapping property of $e^{it\sqrt{1-\Delta}}$ is similar with that of $e^{it\Delta}$ in the low frequency regime. This is related to the fact that the surface $(\xi, \jb \xi)$ has nonvanishing Gaussian curvature near the origin while the surface gets close to the cone  $(\xi, |\xi|)$ as $|\xi| \to \infty$.  In fact,  if $s = \tfrac{d}{2} - \tfrac{d}{r} - \tfrac{1}{q}$,
 it is possible to deduce the estimate \eqref{e:StrWave} from 
\eqref{e:StrKG} via Littlewood--Paley theory, scaling and a limiting argument. Similarly, it is not difficult to obtain the Strichartz estimate for  $e^{it\Delta}$ if \eqref{e:StrKG}
holds for  sharp $\frac{d}{2}$-admissible $(q,r)$.

 For $d \geq 2$, it is known that \eqref{e:StrKG} holds if $(q,r)$ is sharp $\frac{d-1}{2}$-admissible and
\begin{equation}
\label{e:KGd-1}
s \geq \frac{d+1}{2}\bigg(\frac{1}{2} - \frac{1}{r}\bigg).
\end{equation}
Testing \eqref{e:StrKG} on a Knapp-type example reveals that this range of $s$ cannot be enlarged. Also, for $d \geq 1$, it is known that \eqref{e:StrKG} holds if $(q,r)$ is sharp $\frac{d}{2}$-admissible and
\begin{equation}
\label{e:KGd}
s \geq \frac{d+2}{2}\bigg(\frac{1}{2} - \frac{1}{r}\bigg),
\end{equation}
which, again, may be shown to be the optimal range of $s$ for such $(q,r)$.

\subsection{Strichartz estimates for orthonormal functions -- known results}
In this work, we are concerned with Strichartz estimates for orthonormal families of initial data \eqref{e:ONSgeneral}. 
As mentioned before, the key point is to make the exponent $\beta$ as large as possible. Indeed, the case $\beta = 1$ is equivalent to the classical Strichartz estimates: Clearly \eqref{e:classicalStrichartz} implies \eqref{e:ONSgeneral} via the triangle inequality, and conversely, \eqref{e:ONSgeneral} trivially implies \eqref{e:classicalStrichartz} by taking the family $(f_j)_j$ to consist of a single function $f$ with unit norm in $\mathcal{H}$.  

For the Schr\"odinger equation, contributions in \cite{FLLS, FS_AJM, FS_Survey} mean that the sharp value of $\beta$ has been obtained whenever $(q,r)$ is sharp $\frac{d}{2}$-admissible. 
\begin{theorem} \cite{FLLS, FS_AJM, FS_Survey} \label{t:ONS_Schro} Let $d\geq 1$ and suppose $(q,r)$ is sharp $\frac{d}{2}$-admissible. 
\vspace{-7pt}
\begin{enumerate} 
[leftmargin=.8cm, labelsep=0.3 cm, topsep=0pt]
\item[$(i)\,\,$] If  $2\leq r<\frac{2(d+1)}{d-1}$, then
\begin{equation}
\label{orth-schro}
\bigg\|   \sum_j  \nu_j |e^{it\Delta} f_j| ^2 \bigg\|_{L^{\frac q2}_tL^{\frac r2}_x} \lesssim \| \nu \|_{\ell^\beta}
\end{equation}
holds for all families of orthonormal functions $(f_j)_j$ in $L^2$ and $\beta = \frac{2r}{r+2}$. This estimate is sharp in the sense that the estimate fails for $\beta > \frac{2r}{r+1}$.
\item[$(ii)$]
 If $d = 2$ and
$
6 \leq r <\infty,
$ 
or if $d \geq 3$ and
$
\frac{2(d+1)}{d-1} \leq r \leq \frac{2d}{d-2},
$
then \eqref{orth-schro}
holds for all families of orthonormal functions $(f_j)_j$ in $L^2$ and $\beta < \frac{q}{2}$. This estimate is sharp in the sense that the estimate fails for $\beta > \frac{q}{2}$.
\end{enumerate}
\end{theorem}
A remarkable  phenomena here is that, for $d\ge3$,  the sharp value of $\beta$ for the estimate  \eqref{orth-schro} coincides with one (i.e. the trivial case in light of the remarks prior to the above theorem) at the endpoints of the sharp $\frac{d}{2}$-admissible line, that is, $(q,r) = (\infty,2)$ and $(q,r) = (2,\frac{2d}{d-2})$ which respectively correspond to the trivial energy conservation estimate and   
the Keel--Tao endpoint estimate \cite{KeelTao}. Also, the sharp value of $\beta$ reaches its maximum at the point $(q,r) = (\frac{2(d+1)}{d},\frac{2(d+1)}{d-1})$. In fact, as pointed out in \cite{FS_Survey}, the estimates in $(ii)$ follow from those in $(i)$ by interpolation between $(q,r) = (2,\frac{2d}{d-2})$ and points arbitrarily close to $(q,r) = (\frac{2(d+1)}{d},\frac{2(d+1)}{d-1})$; in this sense, $(q,r) = (\frac{2(d+1)}{d},\frac{2(d+1)}{d-1})$ may be considered as an endpoint case of the estimate \eqref{orth-schro}. It remains open whether one can establish a suitable estimate  in weaker form with $(q,r) = (\frac{2(d+1)}{d},\frac{2(d+1)}{d-1})$\footnote{When $d=1$, $(q,r) = (4,\infty)$.} so that other sharp estimates can be recovered from it by interpolation.  See \cite{BHLNS} for discussion on failure of such endpoint estimates in Lorentz spaces.

We also note that the sharp value of $\beta$ has been obtained for the Schr\"odinger equation whenever $(q,r)$ is non-sharp $\frac{d}{2}$-admissible; see \cite{BHLNS}. 
In this case  $f_j$ are assumed to be contained in $\dot H^s$ with $s = \tfrac{d}{2} - \tfrac{d}{r} - \tfrac{2}{q}$. See, for example,  Theorem \ref{t:fracSchrosharp}. As we shall see below, in general, the sharp admissible case will play a crucial role in establishing the estimates in the non-sharp admissible case and therefore, in this introductory section, we only present our new results for the sharp admissible cases; the statements for the non-sharp cases will appear in Section \ref{section:nonsharp}. 

The work of Frank--Sabin \cite{FS_AJM} also contains results for estimates of the form \eqref{e:ONSgeneral} for the wave equation and the Klein--Gordon equation, however, as we shall see below, these are not as advanced as the results contained in Theorem \ref{t:ONS_Schro} for the Schr\"odinger equation.

For the wave equation, Frank--Sabin \cite{FS_AJM} obtained a substantial generalization of \eqref{e:waveST} for orthonormal functions of initial data in $\dot{H}^{1/2}$. By interpolation with a trivial estimate in the case $(q,r) = (\infty,2)$, we state their result as follows.

\begin{theorem} \cite{FS_AJM} \label{t:waveFS} 
Let $d\geq 2$. Suppose $(q,r)$ is sharp $\frac{d-1}{2}$-admissible and $ 2 \leq r \leq \frac{2(d+1)}{d-1}.$
Then,  for all families of orthonormal functions $(f_j)_j$ in $\dot{H}^{s}$, with $s = \frac{d+1}{2}(\frac{1}{2}-\frac{1}{r})$ and $\beta = \frac{2r}{r+2}$,  the following estimate holds:
\begin{equation} \label{e:ONSwaveFS}
\bigg\| \sum_j  \nu_j |e^{it\sqrt{-\Delta}} f_j| ^2 \bigg\|_{L^{\frac q2}_tL^{\frac r2}_x} \lesssim \| \nu \|_{\ell^\beta}.
\end{equation}
\end{theorem}

For the Klein--Gordon equation, Frank and Sabin \cite{FS_AJM} established the following.

\begin{theorem} \cite{FS_AJM} \label{t:KGFS} Let $\sigma > 0$ and suppose $(q,r)$ is sharp $\sigma$-admissible. 
\vspace{-7pt}
\begin{enumerate}
[leftmargin=0.7cm, labelsep=0.3 cm, topsep=0pt]
\item[$(i)\,$] If $\sigma = \frac{d-1}{2}$ and
$
2 \leq r \leq \frac{2(d+1)}{d-1},
$
 for all families of orthonormal functions $(f_j)_j$ in $H^{s}$, $s = \frac{d+1}{2}(\frac{1}{2}-\frac{1}{r})$,    the following estimate holds: 
\begin{equation}
\label{KG-ortho}
\bigg\|   \sum_j  \nu_j |e^{it\sqrt{1-\Delta}} f_j| ^2 \bigg\|_{L^{\frac q2}_tL^{\frac r2}_x} \lesssim \| \nu \|_{\ell^\beta}, \quad \beta = \frac{2r}{r+2}.
\end{equation}
\item[$(ii)$] If $\sigma = \frac{d}{2}$ and
$ 2 \leq r \leq \frac{2(d+2)}{d}, $
then \eqref{KG-ortho} 
holds for all families of orthonormal functions $(f_j)_j$ in $H^{s}$, with $s = \frac{d+2}{2}(\frac{1}{2}-\frac{1}{r})$ and $\beta = \frac{2r}{r+2}$.
\end{enumerate}
\end{theorem}

In both Theorems \ref{t:waveFS} and \ref{t:KGFS}, the range of exponents only  goes up to the diagonal case $q=r$, and the other estimates with $q<r$ are not sharp with respect to the summability exponent $\beta$. This is due to the fact that their argument relies on a special property which is only available when $q=r$ and it is clear that 
their argument does not extend beyond the diagonal case.
 More precisely, in \cite{FS_AJM} the authors followed the original idea of Strichartz \cite{Strichartz}  
 and they regarded  the evolution operators  as adjoint restriction operators given by 
a measure supported on the associated surface (the cone and the hyperboloid, respectively, for the wave and Klein--Gordon equations). For instance, let $S$  denote the upward cone 
$$
S= \{ (\xi,\tau) \in \mathbb{R}^d\times\mathbb{R}: \tau = |\xi| \}.
$$ 
Then, $e^{it\sqrt{-\Delta}}f(x)$ can be written in the form $\widehat{g\mathrm{d}\mu}(x,t)$, where $\mathrm{d}\mu(\xi,\tau) = \frac{\delta(\tau - |\xi|)}{|\xi|}$ and $g(\xi,\tau) = |\xi|\widehat{f}(\xi)$.} With this notation, Strichartz \cite{Strichartz} proved 
\begin{equation}\label{e:Stri}
\|\widehat{g\mathrm{d}\mu}\|_{L_{x,t}^{\frac{2(d+1)}{d-1}}} \lesssim  \|g\|_{L^2(S,\mathrm{d}\mu)} 
\end{equation}
and from this  \eqref{e:waveST} trivially follows. The proof of \eqref{e:Stri} in \cite{Strichartz} used an argument based on interpolation involving an analytic family of operators of the form $T_z f = {}^\vee G_z \ast f$, where $G_z$ is chosen suitably depending on $\mathrm{d}\mu$. The key oscillatory integral estimates on the kernel of these operators relies on a rather delicate identity for ${}^\vee G_z$ whose validity seems tightly connected to the choice of measure $\mathrm{d}\mu$ above. Frank and Sabin's clever observation was that the same basic ingredients could be used to derive \eqref{e:ONSwaveFS} in Theorem \ref{t:waveFS}. For $(q,r)$ beyond the diagonal case $q=r$ (where $s=\frac12$), however, it is necessary to handle data with higher regularity and thus, roughly speaking, one would like to replace $\mathrm{d}\mu$ above with a more singular version of the form $\frac{\delta(\tau - |\xi|)}{|\xi|^a}$ ($a>1$). This causes significant technical difficulty in getting the appropriate kernel estimates; for example, no explicit identity seems available away from the case $a=1$.

Our main contribution to overcome the aforementioned difficulty is to find an appropriate measure and establish corresponding weighted oscillatory integral estimates of so-called ``damped--type" with optimal decay rates (see Section \ref{section:oscillatoryintegrals} for further details). These oscillatory integral estimates then yield kernel estimates  for  a suitable analytic family of operators. Our new idea is sufficiently robust to allow us to significantly improve upon the above results for the wave and Klein--Gordon equations, as well as the fractional Schr\"{o}dinger equation. Below, we describe our main new results for the sharp admissible case; we also obtain new results in the non-sharp admissible case but, 
for reasons we already mentioned before, we postpone our presentation of these results to Section \ref{section:nonsharp}.    

\begin{remark} 
\label{remark-LP}
As an alternative approach to obtain the Strichartz estimate \eqref{e:ONSgeneral} for orthonormal data, one might attempt to adapt the typical strategy for the single-function case based on Littlewood--Paley theory. That is to say, first prove  the estimate for data with localized frequency on dyadic annuli and then put together the estimates for each dyadic piece. However, such an approach does not seem to be effective in the case of orthonormal data. Indeed, in contrast with the classical Strichartz estimate, for \eqref{e:ONSgeneral} there are cases where the frequency localized estimate can not be upgraded to the frequency global estimate; see \cite{BHLNS} for more detail. In particular, our new approach to the estimate \eqref{e:ONSgeneral} for the fractional Schr\"odinger ($\alpha \neq 2$), wave and Klein--Gordon  equations provides a novel method of establishing the classical Strichartz estimates \eqref{e:fractiona}, \eqref{e:StrWave} and \eqref{e:StrKG} which completely avoids use of Littlewood--Paley theory.
\end{remark}

\subsection{Main new results}
As already mentioned above, our primary aim is to extend Theorems \ref{t:waveFS} and \ref{t:KGFS} to all cases where $(q,r)$ is sharp $\frac{d-1}{2}$-admissible or sharp $\frac d2$-admissible, as well as generalizing Theorem \ref{t:ONS_Schro} to the fractional Schr\"{o}dinger equation.

We first present the result for the fractional Schr\"{o}dinger equation which essentially gives the sharp value of $\beta$ in all cases. As far as we are aware, except for the case $\alpha = 2$ in Theorem \ref{t:ONS_Schro}, there are no estimates of the type \eqref{e:ONSgeneral} in the existing literature for the fractional Schr\"{o}dinger equation.  From \eqref{e:FracStri} we note that   $(f_j)_j$ should be in $ \dot{H}^s$ with  $s = \frac{d(2-\alpha)}{2}(\frac{1}{2} - \frac{1}{r})$ if $(q,r)$ is sharp $\frac{d}{2}$-admissible. 

\begin{theorem}
[The fractional Schr\"odinger equation and the sharp $\frac{d}{2}$-admissible]
 \label{t:fracSchrosharp}
Suppose $\alpha \in \mathbb{R} \setminus \{0,1\}$. Let $d\geq 1$ and suppose $(q,r)$ is sharp $\frac{d}{2}$-admissible with $\frac{d}{r} + \frac{\alpha}{q} < d$. 
\vspace{-7pt}
\begin{enumerate}
[leftmargin=.8cm, 
labelsep= 0.3 cm, topsep=0pt]
\item[$(i)\,$] If
$2\leq r<\frac{2(d+1)}{d-1}, $   for all families of orthonormal functions $(f_j)_j$ in $\dot{H}^s$, 
\begin{equation}
\label{FS-ortho}
\bigg\|   \sum_j  \nu_j |e^{it(-\Delta)^{\alpha/2}} f_j| ^2 \bigg\|_{L^{\frac q2}_tL^{\frac r2}_x} \lesssim \| \nu \|_{\ell^\beta}
\end{equation}
holds with $s = \frac{d(2-\alpha)}{2}(\frac{1}{2} - \frac{1}{r})$, and $\beta = \frac{2r}{r+2}$. This estimate is sharp in the sense that the estimate fails for $\beta > \frac{2r}{r+2}$.
\item[$(ii)$] If 
$d = 2$ and
$
6 \leq r < \infty,
$
or if
$d \geq 3$ and
$
\frac{2(d+1)}{d-1} \leq r \le \frac{2d}{d-2},
$
 then
\eqref{FS-ortho} 
holds for all families of orthonormal functions $(f_j)_j$ in $\dot{H}^s$, with $s = \frac{d(2-\alpha)}{2}(\frac{1}{2} - \frac{1}{r})$, and $\beta < \frac{q}{2}$. This estimate is sharp in the sense that the estimate fails for $\beta > \frac{q}{2}$.
\end{enumerate}
\end{theorem}

We next give our result for the wave equation, which significantly improves Theorem \ref{t:waveFS}  when  $q\le r$.  In this case, we note from  \eqref{e:StrWave}
that  it is necessary to consider orthonormal families  $(f_j)_j$ in $\dot{H}^s$, $s = \frac{d+1}{2}(\frac{1}{2}-\frac{1}{r})$ whenever $(q,r)$ is sharp $\frac{d-1}{2}$-admissible. 

\begin{theorem}[The wave equation and the sharp $\frac{d-1}{2}$- admissible] 
\label{t:wave} Let $d\geq 2$ and suppose $(q,r)$ is sharp $\frac{d-1}{2}$-admissible.  
\vspace{-7pt}
\begin{enumerate}
[leftmargin=0.8cm, labelsep= 0.4 cm, topsep=0pt]
\item[$(i)\,$] If  $ 2\leq r<\frac{2d}{d-2}$,
then  \eqref{e:ONSwaveFS}
holds for all families of orthonormal functions $(f_j)_j$ in $\dot{H}^s$, with $s = \frac{d+1}{2}(\frac{1}{2}-\frac{1}{r})$, and $\beta = \frac{2r}{r+2}$. 
\item[$(ii)$] If $d = 3$ and $6 \leq r < \infty,$ or if
$d \geq 4$ and $\frac{2d}{d-2} \leq r \le \frac{2(d-1)}{d-3},$
then   \eqref{e:ONSwaveFS}
holds for all families of orthonormal functions $(f_j)_j$ in $\dot{H}^s$, with $s = \frac{d+1}{2}(\frac{1}{2}-\frac{1}{r})$, and $\beta < \frac{q}{2}$. This estimate is sharp in the sense that the estimate fails for $\beta > \frac{q}{2}$.
\end{enumerate}
\end{theorem}

Our main result regarding the Klein--Gordon equation in the cases $\sigma = \frac{d}{2}$ and $\sigma = \frac{d-1}{2}$ are described in the following two theorems.

\begin{theorem}[The Klein--Gordon equation and the sharp $\frac{d}{2}$-admissible] \label{t:KGSchro} 
Let $d\geq 1$ and suppose $(q,r)$ is sharp $\frac{d}{2}$-admissible. 
\vspace{-7pt}
\begin{enumerate}
[leftmargin=0.8cm, labelsep= 0.3 cm, topsep=0pt]
\item[$(i)\,$] 
If
$2\leq r<\frac{2(d+1)}{d-1}, $
  then  \eqref{KG-ortho} 
holds for all families of orthonormal functions $(f_j)_j$ in $H^s$, with $s\ge \frac{d+2}{2}(\frac{1}{2}-\frac{1}{r})$, and $\beta = \frac{2r}{r+2}$. This is sharp in the sense that the estimate fails if $\beta > \frac{2r}{r+2}$. 
\item[$(ii)$] If 
$d = 2$ and
$ 
6 \leq r < \infty,
$
or if 
$d \geq 3$ and
$ 
\frac{2(d+1)}{d-1} \leq r \leq \frac{2d}{d-2},
$
then
\eqref{KG-ortho} 
holds for all families of orthonormal functions $(f_j)_j$ in $H^s$, with $s \ge \frac{d+2}{2}(\frac{1}{2}-\frac{1}{r})$, and $\beta < \frac{q}{2}$. This estimate is sharp in the sense that the estimate fails for $\beta > \frac{q}{2}$.
\end{enumerate}
\end{theorem}

\begin{theorem}[The Klein--Gordon equation and the sharp $\frac{d-1}{2}$-admissible] 
\label{t:KGwave} 
Let $d\geq 2$ and suppose $(q,r)$ is sharp $\frac{d-1}{2}$-admissible. 
\vspace{-7pt}
\begin{enumerate}
[leftmargin=0.8cm, labelsep= 0.3 cm, topsep=0pt]
\item[$(i)\,$] If
$2\leq r<\frac{2d}{d-2},$
then
\eqref{KG-ortho}
holds for all families of orthonormal functions $(f_j)_j$ in $H^s$, with $s \ge \frac{d+1}{2}(\frac{1}{2}-\frac{1}{r})$, and $\beta = \frac{2r}{r+2}$.
\item[$(ii)$] If 
$d = 3$ and $6 \leq r < \infty, $
or if
$d \geq 4$ and $\frac{2d}{d-2} \leq r \leq \frac{2(d-1)}{d-3}, $
 then
\eqref{KG-ortho}
 holds for all families of orthonormal functions $(f_j)_j$ in $H^s$, with $s\ge  \frac{d+1}{2}(\frac{1}{2}-\frac{1}{r})$, and $\beta < \frac{q}{2}$. This estimate is sharp in the sense that the estimate fails for $\beta > \frac{q}{2}$.
\end{enumerate}
\end{theorem}

It is also possible to unify the cases $\sigma = \frac{d-1}{2}, \frac{d}{2}$ into the case $\sigma = \frac{d-1}{2} + \rho$, where $\rho \in [0,\frac{1}{2}]$. Indeed, Frank and Sabin \cite{FS_AJM} obtained a more general result than Theorem \ref{t:KGFS} corresponding to $\rho \in [0,\frac{1}{2}]$ ($d \geq 2$) and $\rho \in (0,\frac{1}{2}]$ ($d=1$). For simplicity of the exposition, we only consider the cases  $\rho = 0,\frac{1}{2}$ in Theorem \ref{t:KGFS}, Theorem \ref{t:KGSchro},  and Theorem \ref{t:KGwave}; we refer the reader forward to Section \ref{section:ONSsharp} for discussion of the more general case, including a result for the sharp $(\frac{d-1}{2} + \rho)$-admissible case which completely includes the result in \cite{FS_AJM}.

We shall prove two necessary conditions in Section \ref{section:necesssary} which justify the claims in the above theorems regarding the sharpness of the range of $\beta$. Left open is to show that the exponent $\beta = \frac{2r}{r+2}$ in $(i)$ of Theorems \ref{t:wave} and \ref{t:KGwave} is sharp. Our examples show that this exponent is sharp in the extreme cases $r=2$ and $r = \frac{2d}{d-2}$ and thus we expect that this is the sharp value for intermediate values of $r$ too. 

\subsection{Estimates for oscillatory integrals with weights}
As we have already mentioned, the new idea in the proof of our main theorems is to make use of estimates for oscillatory integrals with weights, or damping factors, which will be combined with an interpolation argument based on a suitable analytic family of operators.  See Proposition \ref{p:abstractdual} where we formalize our approach under  the assumption that we already have the desired oscillatory integral estimate.  In fact, oscillatory integrals with weights naturally arise by absorbing to new operators the multiplier operators $|D|^s$ and $\langle D\rangle^s$ defining the Sobolev spaces. In order to obtain the appropriate dispersive estimates for these operators we need to show  the optimal decay estimates for the associated oscillatory integrals with weights. Interestingly, in some cases, such oscillatory integrals turn out to be oscillatory integrals with damping factors (\cite{CDMM, KenigPonceVega, CarberyZiesler}).  

In Section \ref{section:oscillatoryintegrals}, we prove various estimates for oscillatory integrals of the form  
\begin{equation}
\label{osi}
I^\phi(w)=  \int_{\mathbb{R}^d} e^{i(x\cdot \xi + t \phi(\xi))}  a(\xi) \,w (\xi)\,   \mathrm{d}\xi, 
\end{equation}
for a suitable choice of the cutoff function $a$ and the weight function $w$ which we need to choose properly according to our particular purposes. Indeed, it will be crucial for proving Theorems \ref{t:fracSchrosharp}--\ref{t:KGwave} that our choice of $a$ and $w$ allows us to recover the optimal decay rate in $t$.  Oscillatory integrals of the form $I^\phi(w)$ are ubiquitous in analysis and often the weight is chosen to be a power of the determinant of the Hessian matrix of $\phi$, denoted by ${\rm det}\, H\phi$, which mitigates bad behavior near degeneracies. For the purpose of our applications to Strichartz estimates for orthonormal families of data, we will  choose $w(\xi) = |\xi|^{z}$ or $\langle \xi \rangle^z$, for certain $z\in \mathbb{C}$, which may not necessarily be the form of $|{\rm det}\, H\phi|^z$. Since the Hessian vanishes in the case of the wave equation, this case provides an example of the case where it will be \textit{necessary} to consider weights not belonging to the typical class of weights which are powers of the Hessian. Our new results concerning such estimates are Propositions \ref{p:osc-wave}, \ref{p:KGoscillatory_wave}, \ref{p:KGoscillatory_Schro} and \ref{p:OsciAlmostRad}. We believe that these weighted oscillatory estimates are of independent interest and we provide further contextual remarks in the beginning of Section \ref{section:oscillatoryintegrals}.

\begin{organisation}
In Section \ref{section:Prelims}, we introduce some notation and key facts that will be used throughout the paper. This includes the duality principle from \cite{FS_AJM} which rephrases Strichartz estimates for orthonormal families of initial data in terms of certain Schatten space estimates. In Section \ref{section:oscillatoryintegrals}, we state and prove the weighted oscillatory integral estimates which are key to our proof of those Schatten space estimates corresponding to the main theorems in this paper. In Section \ref{section:ONSsharp}, we prove the sufficiency claims in Theorems \ref{t:fracSchrosharp}, \ref{t:wave}, \ref{t:KGSchro} and \ref{t:KGwave}  (sharp admissible cases), and in Section \ref{section:nonsharp} we prove analogous results in the non-sharp admissible case. The necessity claims in both the sharp and non-sharp admissible cases all follow from the necessary conditions which we will establish in Section \ref{section:necesssary}. Finally, in Section \ref{section:applications} we present several applications of our new estimates to the theory of infinite systems of Hartree type, weighted velocity averaging lemmas for kinetic transport equations, and refined Strichartz estimates.
\end{organisation}

\section{Notation and preliminaries} \label{section:Prelims}

For $A, B>0$, and $\rho\in \mathbb R$, by 
$A \lesssim_\rho  B$  we mean that $A \le  C(1+|\rho|)^{c} B$ 
for some constants $C, c>0$.
We often use the notation $\langle x\rangle = (1+|x|^2)^{\frac12}$.  Also, the Hessian matrix of $\phi:\mathbb{R}^d \to \mathbb{C}$ at $\xi$ is given by $H\phi(\xi) = (\partial_i\partial_j \phi(\xi))_{1\le i,j \le d}$. 

\subsection{Function spaces} 
The spaces  $\dot{H}^s(\mathbb{R}^d)$,  ${H}^s(\mathbb{R}^d)$, respectively  denote  the  homogeneous and inhomogeneous Sobolev spaces based on $L^2(\mathbb{R}^d)$ which are equipped  with the norms
\[
\|f\|_{\dot{H}^s} = \||D|^s f \|_{L^2}, \  \    \|f\|_{{H}^s} = \|\jb D^s f \|_{L^2}.
\]
Here, $\varphi(D)$ will denote the Fourier multiplier operator given by
\[
\widehat{\varphi(D)}f(\xi) = \varphi(\xi)\widehat{f}(\xi),
\]
where the Fourier transform of a sufficiently nice function $f : \mathbb{R}^d \to \mathbb{C}$ is given by
\[
\mathcal{F}f(\xi) = \widehat{f}(\xi) = \int_{\mathbb{R}^d} f(x) e^{-ix \cdot \xi} \, \mathrm{d}x.
\]
Also, the inverse Fourier transform is defined by 
$$
\mathcal{F}^{-1}f(x) = {f}^{\vee}(x) = \frac{1}{(2\pi )^d} \int_{\mathbb{R}^d} f(\xi) e^{ix \cdot \xi} \, \mathrm{d}\xi. 
$$

For $p \in [1,\infty)$ and $r \in [1,\infty]$, we write $L^{p,r}(\mathbb{R}^d)$ for the Lorentz space with norm
\[
\|f\|_{L^{p,r}} = \bigg(\int_0^\infty (t^{1/p} f^*(t))^r \, \frac{\mathrm{d}t}{t} \bigg)^{1/r}
\]
for $r < \infty$, and
\[
\|f\|_{L^{p,\infty}} = \sup_{t > 0} t^{1/p}f^*(t).
\]
For $p \in (1,\infty)$ and $r \in [1,\infty]$, the space $L^{p,r}$ is normable. Strictly speaking, $\|\cdot\|_{L^{p,r}}$ is a true norm when $r \leq p$ and a quasi-norm otherwise; a true norm which is equivalent to $\|\cdot\|_{L^{p,r}}$ is obtained by replacing $f^*(t)$ with $\frac{1}{t} \int_0^t f^*$. For details on fundamental properties of Lorentz spaces, we refer the reader to \cite{SteinWeiss}. 

Since $L^{p,p} = L^p$ and $L^{p,r_1} \subseteq L^{p,r_2}$ if $r_1 \leq r_2$, Lorentz spaces provide a natural setting
to seek refinements of  certain classical inequalities for $L^p$ functions. 
An example is the refined version of the Hardy--Littlewood--Sobolev inequality for functions in Lorentz spaces:
\begin{equation} \label{e:HLS}
\bigg|\int_{\mathbb{R}} \int_{\mathbb{R}} \frac{g_1(t_1)g_2(t_2)}{|t_1-t_2|^\lambda} \, \mathrm{d}t_1\mathrm{d}t_2 \bigg| \lesssim \|g_1\|_{L^{p_1,r_1}}\|g_2\|_{L^{p_2,r_2}}
\end{equation}
where $\lambda \in (0,1)$, $p_1,p_2 \in (1,\infty)$ satisfy $\frac{1}{p_1} + \frac{1}{p_2} + \lambda = 2$, and $\frac{1}{r_1} + \frac{1}{r_2} \geq 1$. This can be found in \cite{ONeil}.

Finally, we introduce a convenient notation when handling mixed-norm estimates. For (space-time) functions $F : \mathbb{R}^d \times \mathbb{R} \to \mathbb{C}$ belonging to 
$L^{q}_t L^r_x$, we write
\[
\|F\|_{q,r} = \|F\|_{L^{q}_t L^r_x}.
\]
In the case of Lorentz spaces,  if $F\in L^{q,p}_t L^r_x$ or $F\in L^{q}_t L^{r,p}_x$, we write
\[
\|F\|_{(q,p),r} = \|F\|_{L^{q,p}_t L^r_x},    \  \  \|F\|_{q,(r,p)} = \|F\|_{L^{q}_t L^{r,p}_x},
\]
respectively.

\subsection{A duality principle}
We shall make use of the duality principle originating in the work of Frank and Sabin \cite[Lemma 3]{FS_AJM}\footnote{Strictly speaking, Lemma 3 in \cite{FS_AJM} is stated in terms of pure Lebesgue space norms; as noted in \cite{BHLNS}, the extension to the mixed-norm setting including Lorentz spaces follows with minimal modifications.}. In the following statement (and throughout the paper), given an exponent $q \geq 2$, we write $\tq$ for the exponent given by
\[
\frac{\tq}{2} = \bigg(\frac{q}{2}\bigg)',
\]
or equivalently
\[
\frac{1}{\tq} = \frac{1}{2} - \frac{1}{q}.
\]
For $\beta\in[1,\infty)$, $\mathcal{C}^\beta=\mathcal{C}^\beta(L^2(\mathbb{R}^d))$ denotes the Schatten space based on $L^2(\mathbb{R}^d)$ which is the space of all compact operators $A$ on $L^2(\mathbb{R}^d)$ such that ${\rm Tr}|A|^\beta<\infty$, where $|A|=\sqrt{A^*A}$, and its norm is defined by $\|A\|_{\mathcal{C}^\beta}=({\rm Tr}|A|^\beta)^\frac1\beta$. If $\beta=\infty$, we define $\|A\|_{\mathcal{C}^\infty} = \|A\|_{L^2\to L^2}$. Also, the case $\beta=2$ is special in the sense that $\mathcal{C}^2$ is the Hilbert--Schmidt class and the $\mathcal{C}^2$ norm is given by $\|A\|_{\mathcal{C}^2} = \|K_A\|_{L^2(\mathbb{R}^d\times \mathbb{R}^d)}$ where $K_A$ is the integral kernel of $A$. More details on the Schatten classes can be found in  \cite{simon}.

\begin{proposition}[Duality principle] \label{p:duality}
Suppose $T$ is a bounded operator from $L^2(\mathbb{R}^d)$ to $L^{q,p}_tL^r_x$ for some $q > 2$, $p,r \geq 2$ and $\beta \geq 1$. Then
\begin{equation*}
\bigg\|   \sum_j  \nu_j |Tf_j| ^2 \bigg\|_{(\frac q2, \frac p2), \frac r2 }\lesssim \| \nu \|_{\ell^\beta}
\end{equation*}
holds for all families of orthonormal functions $(f_j)_j$ in $L^2(\mathbb{R}^d)$ and all sequences $\nu \in \ell^\beta$ if and only if
\begin{equation*}
\|W TT^* \overline{W} \|_{\mathcal{C}^{\beta'}} \lesssim \|W\|^2_{(\tq, \tp), \tr}
\end{equation*}
holds for all $W \in L^{\tq,\tp}_tL^{\tr}_x$.
\end{proposition}
We shall apply the above duality principle in the case $\tp = 2\tr$. That is to say,  we will obtain estimates
\[
\bigg\|   \sum_j  \nu_j |Tf_j| ^2 \bigg\|_{(\frac q2, \beta), \frac r2 }\lesssim \| \nu \|_{\ell^\beta}, \qquad \beta = \frac{2r}{r+2}
\]
from estimates of the form
\[
\|W TT^* \overline{W} \|_{\mathcal{C}^{\tr}} \lesssim \|W\|^2_{(\tq, 2\tr), \tr}\,.
\]

\subsection{Dyadic decomposition} Throughout the paper, $\chi$ denotes a fixed function $C_c^\infty(-1,1)$ which satisfies $\chi=1$ on $(-\frac12, \frac12)$, and we define 
\[
\chi_0 = \chi(2^{-1}\cdot)-\chi. 
\]
For $j \in \mathbb Z$, we write $\chi_j=\chi_0(2^{-j}\cdot)$ so that $\chi_j$ is a smooth cutoff function supported in $(2^{j-1}, 2^{j+1})$. By construction
we have 
\begin{equation} \label{e:dyadicdecomp}
\chi =  \sum_{j \le - 1} \chi_j,   \ \ \  \sum_{j=-\infty}^\infty \chi_j=\chi+\sum_{j=0}^\infty  \chi_j=1.
\end{equation}

\subsection{Van der Corput's lemma}
The following proposition, often referred to as van der Corput's lemma, will be useful throughout Section \ref{section:oscillatoryintegrals} (see, for example, \cite[p. 334]{Stein}).
\begin{proposition} \label{p:vandercorput}
Let $k \geq 1$ and $\rho_1,\rho_2 \in \mathbb{R}$ with $\rho_1 < \rho_2$. Suppose $\theta : [\rho_1,\rho_2] \to \mathbb{R}$  satisfies 
\[
|\theta^{(k)}(\rho)| \geq 1
\]
for all $\rho \in [\rho_1,\rho_2]$. Suppose also that $a : [\rho_1,\rho_2] \to \mathbb{R}$ is differentiable and $a'$ is integrable on $[\rho_1,\rho_2]$. Then there exists a constant $C_k$, depending only on $k$, such that
\[
\bigg| \int_{\rho_1}^{\rho_2} e^{i \mu \theta(\rho)} a(\rho) \, \mathrm{d}\rho\bigg| \leq C_k\bigg(|a(\rho_2)| + \int_{\rho_1}^{\rho_2} |a'(\rho)| \, \mathrm{d}\rho \bigg) |\mu|^{-1/k}
\]
holds, when $k \geq 2$, or $k = 1$ and $\theta'$ is monotonic.
\end{proposition}

\section{Weighted oscillatory integral estimates} \label{section:oscillatoryintegrals}

In this section, we establish various new weighted oscillatory integral estimates, by which we mean decay estimates for oscillatory integrals of the form  $I^\phi(w)$
defined  by \eqref{osi}.  These estimates are related to the dispersive estimates for the associated propagators. 
If $a$ has compact support, $\phi$ is smooth on the support $a$,  and  $\phi$ is not degenerate, that is to say,  $\det H\phi \neq 0$, then  the stationary phase method gives 
\[ 
|I^\phi(1)|\lesssim |t|^{-\frac d2}.
\]
 This decay estimate is in general best possible under such a non-degeneracy assumption. However, if $\det H\phi$ vanishes, only weaker decay estimates are possible. 
 There have been attempts to recover optimal decay $O(|t|^{-\frac d2})$ by introducing a suitable damping weight.  A typical weight involves powers of the determinant of the Hessian matrix of $\phi$. In view of  the asymptotic expansion of the oscillatory integral  with non-degenerate phase (eg. see \cite[(7.7.12), p. 220]{Hormanderbook} and \cite[p. 360--361]{Stein}),   it is natural to use the damping factor $w(\xi)=|{\rm det} H\phi(\xi)|^{1/2}$ to recover the best possible decay  $O(|t|^{-\frac d2})$.   
 
 This type of  estimate for ${I}^{\phi}(|{\rm det} H\phi|^\gamma)$  has been studied by various authors.   Early results of this kind go back as far as work of Sogge and Stein \cite{SoggeStein}.   In work of Cowling \emph{et al.}\,\cite{CDMM}   the damped oscillatory integral estimates were studied with finite-type  convex  $\phi$ but  the weight was assumed to have sufficient smoothness and weights with  complex exponent $\gamma$ were not considered.  Establishing estimates of the form 
$$ 
| I^\phi(|{\rm det} H\phi|^{1/2 + i\kappa})|\lesssim_\kappa |t|^{-\frac d2}
$$ 
turned out to be a delicate problem, even without $\kappa$. 
 Until now only results for special classes of $\phi$  are
 known. Kenig, Ponce and Vega \cite{KenigPonceVega} obtained such an estimate with  {elliptic polynomial} $\phi$. In the radial case $\phi(\xi)=\phi_0(|\xi|)$, estimates were obtained by Carbery--Ziesler \cite{CarberyZiesler} under the assumption  that $\phi_0$ and $\phi_0'$ are convex.
  
Below, we present our weighted oscillatory integral estimates corresponding to the wave equation, the Klein--Gordon equation and the fractional Schr\"{o}dinger equation,  which we need for proof of the orthonormal Strichartz estimates. In the case of the wave and Klein--Gordon equations, our results are completely new. For the case of the fractional Schr\"{o}dinger equation, our estimates are new in the case $\alpha<2$ and overlap with work of Kenig, Ponce and Vega \cite{KenigPonceVega} when $\alpha\ge2$. See also the remark at the end of this section for further discussion about this point.   
Compared with the previous work,  our argument here is significantly simpler. This is mainly due to our efficient use of the bounds which are obtained by making use of the first order derivatives of the phase functions.

\subsection{The wave equation}

For each $\kappa \in \mathbb{R}$ and $(x,t) \in \mathbb{R}^d \times \mathbb{R}$, we define the oscillatory integral
\[ 
\mathcal W_{\kappa}(x,t)=\int e^{i(x\cdot \xi+ t|\xi|)} |\xi|^{-\frac{d+1}{2}+i\kappa} \, \chi(|\xi|) \,\mathrm{d}\xi.
\] 

\begin{proposition} \label{p:osc-wave} 
There exist constants $C < \infty$ and $N > 0$ such that 
\begin{equation} \label{e:osc-wave}
|\kappa  \mathcal W_{\kappa}(x,t)| \leq C(1+|\kappa|)^{N}|t|^{-\frac{d-1}2}
\end{equation} 
for all $\kappa \in \mathbb{R}$ and $(x,t) \in \mathbb{R}^d \times \mathbb{R}$.
\end{proposition} 
To place this estimate in some context, recall the standard dispersive estimate 
\begin{equation*}
\bigg| \int e^{i(x\cdot \xi+ t|\xi|)} \chi_0(|\xi|) \,\mathrm{d}\xi \bigg| \lesssim \min\big\{1,|t|^{-\frac{d-1}{2}}\big\}
\end{equation*}
related to the wave equation
and the rate of decay $\frac{d-1}{2}$ is best possible. By scaling it easily follows that
$| \int e^{i(x\cdot \xi+ t|\xi|)} \chi_0(2^{-j}|\xi|) \,\mathrm{d}\xi | \lesssim \min\{ 2^{dj}, 2^{\frac{d+1}{2}j }|t|^{-\frac{d-1}{2} } \} 
$ 
for  each $j \in \mathbb{Z}$. Thus, for any $\lambda < \frac{d-1}{2}$, a standard dyadic decomposition yields
\begin{equation*}
\bigg| \int e^{i(x\cdot \xi+ t|\xi|)}\frac{1}{|\xi|^{d - \lambda}}  \,\mathrm{d}\xi \bigg| \lesssim \sum_{2^j \leq |t|^{-1}} 2^{\lambda j} + |t|^{-\frac{d-1}{2}} \sum_{2^j > |t|^{-1}} 2^{(\lambda - \frac{d-1}{2})j} \lesssim |t|^{-\lambda}.
\end{equation*}
So, the estimate \eqref{e:osc-wave}  may be regarded as the endpoint case of the above family of estimates. However, it appears to be difficult to establish Theorem \ref{t:wave} using these weighted estimates and they seem to be useful only for the sub-optimal decay rate.  The additional oscillation introduced by the factor $|\xi|^{i\kappa}$ allows us to recover the optimal decay rate $\frac{d-1}{2}$  and this is crucial for the approach taken in this paper. We remark that the extra $\kappa$ factor in the left-hand side of  \eqref{e:osc-wave}  is also important  and it is not difficult to show that  the estimate is not true without this factor (this becomes clear in the proofs of Lemma \ref{elementary} and Proposition \ref{p:osc-wave} below).

As preparation for the proof of Proposition \ref{p:osc-wave}, we establish the following elementary estimate.
\begin{lemma}\label{elementary} 
There exists a constant $C < \infty$, independent of $\kappa, \varepsilon,$ and $R$,  such that 
we have 
\[ 
\bigg |  \kappa \int_{\varepsilon}^R e^{\pm  i A\rho }    \rho^{-1+i\kappa}  \, \mathrm{d}\rho\,\bigg|\le    C  (1+|\kappa|)^{2}.
\]   
\end{lemma} 
In fact, what we shall need in the proof of Proposition \ref{p:osc-wave} is the following estimate: 
\Be\label{mod} 
\bigg|  \kappa \int_0^R \psi(\rho) e^{\pm  i A\rho }    \rho^{-1+i\kappa}  \, \mathrm{d}\rho \bigg| \leq C(1+|\kappa|)^{2}(\|\psi'\|_1+\|\psi\|_\infty),
\Ee
where the constant $C < \infty$ is the same as in Lemma \ref{elementary} (and hence independent of $R > 0$, $\kappa \in \mathbb{R}$, $A \geq 0$ and $\psi \in C_c^\infty(0,\infty)$). We note that \eqref{mod} follows immediately from Lemma \ref{elementary}. Indeed, if we set 
\[
F(r)=\kappa\int^r_{\varepsilon}  e^{\pm  i A\rho }    \rho^{-1+i\kappa}  \, \mathrm{d}\rho,
\] 
where $\varepsilon>0$ is a sufficiently small number so that $\psi(\varepsilon)=0$, then, by integration by parts, we get
\begin{equation*}
   \bigg|  \kappa \int_0^R \psi(\rho) e^{\pm  i A\rho }    
        \rho^{-1+i\kappa}  \, \mathrm{d}\rho \bigg| 
          \leq |\psi(R)F(R)| + \bigg|\int_0^R \psi'(\rho) F(\rho)  \, \mathrm{d}\rho \bigg|  
\end{equation*}
and thus \eqref{mod} clearly follows from Lemma \ref{elementary}.

\begin{proof}[Proof of Lemma \ref{elementary}]  
Since $\frac{\mathrm{d}}{\mathrm{d}\rho}   \rho^{i\kappa} =i\kappa  \rho^{-1+i\kappa}$, we obviously have
\[
\bigg |  \kappa \int_{\varepsilon}^R  \rho^{-1+i\kappa}  \, \mathrm{d}\rho\,\bigg| \leq 2
\]
which gives the desired estimate when $A = 0$.

Next, note that the case $A \neq 0$ follows from the case $A = 1$ by rescaling $\rho\to \rho/A$, and thus it is sufficient to show that 
\[ 
\bigg |  \kappa \int_{\varepsilon}^R e^{\pm  i \rho }    \rho^{-1+i\kappa}  \, \mathrm{d}\rho\,\bigg| \le    C  (1+|\kappa|)^{2}
\] 
holds with $C$ independent of $R, \varepsilon>0$. We initially assume $\varepsilon\le 1\le R$ and split the integral
\[ 
\kappa \int_{\varepsilon}^R e^{\pm  i \rho }    \rho^{-1+i\kappa}  \, \mathrm{d}\rho  =  \kappa \Big(  \int_{\varepsilon}^1 +  \int_1^R \Big) e^{\pm  i \rho }    \rho^{-1+i\kappa}  \, \mathrm{d}\rho:= I+ I\!I. 
\]
By integration by parts it is easy to see $|I\!I|\le (1+|\kappa|)^{2}$. For $I$, if we write
\[
I=   \kappa   \int_{\varepsilon}^1 (e^{\pm  i \rho }  -1)  \rho^{-1+i\kappa}  \, \mathrm{d}\rho +\kappa   \int_{\varepsilon}^1  \rho^{-1+i\kappa}  \, \mathrm{d}\rho,
\] 
then it is clear that $|I|\le   C  (1+|\kappa|)$, thus completing the argument when $\varepsilon\le 1\le R$. If $\varepsilon \leq R \leq 1$ we only need the argument for $I$, and if $1 \leq \varepsilon \le R$ the argument for $I\!I$ is enough.
\end{proof}

\begin{proof}[Proof of Proposition \ref{p:osc-wave}]
The trivial estimate
\[
|\kappa  \mathcal W_{\kappa}(x,t)| \lesssim |\kappa| \int  |\xi|^{-\frac{d+1}{2}} \, \chi(|\xi|) \,\mathrm{d}\xi \lesssim |\kappa|
\]
means that we may reduce to the case where $|t| \geq 1$. Furthermore, in the case $|t| \geq 1$, we argue that it suffices to consider the case where
\begin{equation} \label{e:xliket}
2^{-1}|t|\le |x|\le 2|t|. 
\end{equation}
To see this, suppose either $|x|\ge 2 |t|$ or $2|x|\le |t|$. Using the dyadic decomposition \eqref{e:dyadicdecomp} and rescaling we get
\[
\mathcal W_{\kappa}(x,t) = \sum_{j \geq 1} 2^{-j(\frac{d-1}{2} + i\kappa)} \int e^{i\Phi_j(\xi)}  |\xi|^{-\frac{d+1}{2}+i\kappa} \,\chi_0(|\xi|) \,\mathrm{d}\xi,
\]
where
\[
\Phi_j(\xi) = 2^{-j}x \cdot \xi + 2^{-j}t|\xi|.
\]
In the case $|x|\ge 2 |t|$, we have $|\nabla \Phi_j(\xi)| \gtrsim 2^{-j}|x| $ on the support of $\chi_0(|\cdot|)$  and it follows by repeated integration by parts that
\[
\bigg| \int e^{i\Phi_j(\xi)}  |\xi|^{-\frac{d+1}{2}+i\kappa} \, \chi_0(|\xi|) \,\mathrm{d}\xi \bigg| \lesssim_\kappa \min(1, (2^{-j}|x|)^{-M})
\]
for any $M \geq 0$. From this, it follows that
\begin{align*}
|\mathcal W_{\kappa}(x,t)| & \lesssim \sum_{\substack{j \geq 1 \\ 2^j \leq |x|}} 2^{-j\frac{d-1}{2}} (2^{-j}|x|)^{-M} + \sum_{\substack{j \geq 1 \\ 2^j \geq |x|}} 2^{-j\frac{d-1}{2}} 
 \lesssim |x|^{-\frac{d-1}{2}},  
\end{align*}
if we choose $M$ sufficiently large, and hence
\[
|\mathcal W_{\kappa}(x,t)| \lesssim |t|^{-\frac{d-1}{2}}
\]
holds in the case $|x|\ge 2 |t|$.  It is easy to see the same estimate holds in the case $2|x|\le |t|$ by a similar argument. 
Since 
$$
\bigg| \int e^{i\Phi_j(\xi)}  |\xi|^{-\frac{d+1}{2}+i\kappa} \, \chi_0(|\xi|) \,\mathrm{d}\xi \bigg| \lesssim_\kappa \min(1, (2^{-j}|t|)^{-M})
$$
for any $M \geq 0$ if $2|x|\le |t|$, repeating the argument for the case  $|x|\ge 2 |t|$   yields the desired estimate.

For the remainder of the proof, we assume that \eqref{e:xliket} holds. Since we have
\[
\bigg|\int e^{i(x\cdot \xi+ t|\xi|)} |\xi|^{-\frac{d+1}{2}+i\kappa} \, \chi( |t \xi| ) \chi(|\xi|) \,\mathrm{d}\xi \bigg| \lesssim \int |\xi|^{-\frac{d+1}{2}} \chi(|t\xi|)  \,\mathrm{d}\xi \lesssim |t|^{-\frac{d-1}{2}},
\]
it suffices to consider
\begin{align*} 
\widetilde{\mathcal W}_{\kappa}(x,t) :=  \int e^{i(x\cdot \xi+ t|\xi|)} |\xi|^{-\frac{d+1}{2}+i\kappa} \,  (1- \chi(|t \xi|))\chi(|\xi|) \,\mathrm{d}\xi.
\end{align*}
To estimate this oscillatory integral, we use spherical  coordinates to write 
\[ 
\widetilde{\mathcal W}_{\kappa}(x,t)  =  \int  \widehat{\mathrm{d}\sigma}(\rho x)  \ e^{i t\rho} \rho^{\frac{d-3}{2}+i\kappa} \,  (1- \chi({ |t|} \rho))\chi(\rho) \,\mathrm{d}\rho,
 \] 
  where  $\mathrm{d}\sigma$ is the surface measure  on the sphere  $\{ x: |x|=1\}.$
 Since $\widehat{\mathrm{d}\sigma}(x) =C_d|x|^{-\frac{d-2}2} J_{\frac{d-2}2}(|x|) $, using an asymptotic expansion for the Bessel function we get 
\begin{equation} \label{fourier-sph}  
\widehat{\mathrm{d}\sigma}(x) = C_\pm |x|^{-\frac{d-1}{2}}e^{\pm i |x|}+ \sum_{j=1}^{N} C_{j, \pm} |x|^{-\frac{d-1}{2}-j}e^{\pm i |x|}+ \mathcal E(|x|),   \quad |x| \geq 1,
\end{equation}
where 
$ \big( \frac{d}{dr}\big)^k \mathcal E(r)= O(r^{-k-\frac{ N+d-1}2}).$
For more details regarding this, see \cite[{p. 338}]{Stein} and  also \cite[Proposition 3, {p. 334}]{Stein}.  Since we are assuming $|x|\sim |t|\gtrsim 1$, for the contribution from the leading term in this expansion,  it is enough to show that 
\[ 
\bigg| \kappa\int   \ e^{i (t\pm |x|)\rho} \rho^{-1+i\kappa} \,  (1- \chi({ |t|} \rho))\chi(\rho) \,\mathrm{d}\rho \bigg| \leq C(1+|\kappa|)^{2}.
\] 
However, this immediately follows from \eqref{mod} since $\|((1- \chi({ |t|} \cdot))\chi  )'\|_1\lesssim 1$ uniformly in $t$.
The other terms can be handled similarly but in an easier manner, so we omit the details. 
\end{proof}

\subsection{The Klein--Gordon equation} 
In this subsection we prove  the weighted oscillatory integral estimates which are needed for the proof of Theorems \ref{t:KGSchro} and \ref{t:KGwave}. The associated surface  $(\xi, \langle \xi\rangle)$ to the Klein--Gordon equation  has nonvanishing Gaussian curvature everywhere  but 
 its Gaussian curvature asymptotically vanishes at infinity.  
 This gives rise to significant complication in the argument for  obtaining  the sharp decay rate in the oscillatory integral estimates. 

For each $\kappa \in \mathbb{R}$, $\varepsilon > 0$ and $(x,t) \in \mathbb{R}^d \times \mathbb{R}$, we define the oscillatory integrals $\mathcal{J}_{\kappa,R}(x,t)$ and $\mathcal{K}_{\kappa,\varepsilon}(x,t)$ by
\[ 
\mathcal{J}_{\kappa,\varepsilon} (x,t)= \int e^{i(x\cdot \xi+t\langle \xi \rangle)}  \psi(\varepsilon \xi) ( 1+|\xi|^2)^{-\frac {d+1}4-i\kappa } \, \mathrm{d}\xi
\] 
and
\[ 
\mathcal K_{\kappa,\varepsilon}(x,t) = \int e^{i(x\cdot \xi+t\langle \xi \rangle)} \psi(\varepsilon \xi) ( 1+|\xi|^2)^{-\frac {d+2}4-i\kappa} \, \mathrm{d}\xi.
\] 
Here,  $\psi\in C_0^\infty (-1, 1)$. We insert $\psi(\varepsilon{|\xi|})$  to guarantee existence of the integral. However presence of the additional factor $\psi(\varepsilon |\xi|)$ does not have any significance and the overall argument below works as if there were no such factor. It is easy to see by integration by parts that, for $t\neq 0$,  
the limits $\lim_{\varepsilon \to 0}   \mathcal{J}_{\kappa,\varepsilon} (x,t)$ and   $\lim_{\varepsilon \to 0}   \mathcal{K}_{\kappa,\varepsilon} (x,t)$
 exist.  

The following  respectively give the optimal  wave-like and Schr\"odinger-like  dispersive estimates for the Klein-Gordon equation.

\begin{proposition} \label{p:KGoscillatory_wave} 
There exist constants  $C < \infty$ and $N > 0$ such that 
\begin{equation}
\label{KGwavedecay}
|\kappa \mathcal{J}_{\kappa,\varepsilon} (x,t)| \leq C(1+|\kappa|)^{N}|t|^{-\frac {d-1}2} 
\end{equation}
for all $\kappa \in \mathbb{R}$, $\varepsilon > 0$ and $(x,t) \in \mathbb{R}^d \times \mathbb{R}$.
\end{proposition}
\begin{proposition} \label{p:KGoscillatory_Schro}
There exist constants $C < \infty$ and $N > 0$ such that 
 \begin{equation} \label{e:KGdispersivemain}
 |\mathcal{K}_{\kappa, \varepsilon} (x,t)| \leq C(1+|\kappa|)^{N} |t|^{-\frac d2}
 \end{equation}
for all $\kappa \in \mathbb{R}$, $\varepsilon>0$ and $(x,t) \in \mathbb{R}^d \times \mathbb{R}$.
\end{proposition}

Before starting the proof,  we recall the following (see  \cite[Appendix]{GV}) which is an easy consequence of the stationary phase method.    

\begin{lemma}\label{KGKG}   Let $\widetilde \chi$ be a smooth function supported in $\{\xi: 2^{-1}\le |\xi|\le 2\}$. For $0<\rho \le  1$, let us set 
\[ I_\rho=\int e^{i(x\cdot \xi+ t\sqrt{\rho^2+|\xi|^2})}   \widetilde \chi (\xi) \, \mathrm{d}\xi .\] 
Then we have the estimate $| I_\rho|\lesssim \min (|t|^{-\frac{d-1}2}, \rho^{-1} |t|^{-\frac{d}2})$.
\end{lemma}

For convenience of the reader we briefly explain how to show this.  Via a finite decomposition and rotation  we may assume the cutoff function $\widetilde \chi
$ is  supported in   a small conic neighborhood of $e_d$ such that $\xi_d\in [2^{-2}, 2^2]$.  Thus  we may write 
\[   \sqrt{\rho^2+|\xi|^2}=\xi_d\Big(1+\frac{\rho^2+|\bar \xi|^2}{\xi_d^2}\Big)^\frac12=\xi_d+ \frac12\Big( \frac{\rho^2}{\xi_d}+\frac{|\bar \xi|^2}{\xi_d}\Big)  + \text{ error },  \] 
where $\xi=(\bar\xi, \xi_d)$. 
To get the  estimate $| I_\rho|\lesssim |t|^{-\frac{d-1}2}$, we may simply  ignore  $\xi_d$ and  apply the stationary phase method in $\bar \xi$. 
For the estimate  $| I_\rho|\lesssim  \rho^{-1} |t|^{-\frac{d}2}$, note that  one of the eigenvalues of the Hessian matrix of $\sqrt{\rho^2+|\xi|^2}$ is $\sim \rho^2$ while the other  $d-1$ eigenvalues are $\sim 1$. 

\begin{proof}[Proof of Proposition \ref{p:KGoscillatory_wave}]
As before, we break the integral dyadically.  
Using \eqref{e:dyadicdecomp},  we write 
\begin{align*}
\mathcal J_{\kappa, \varepsilon}=J_{\circ}+\sum_{j=0}^\infty J_j &:=  \int e^{i(x\cdot \xi+t\langle \xi \rangle)}  \frac{\psi(\varepsilon |\xi|)\chi(|\xi|)}{( 1+|\xi|^2)^{\frac {d+1}4+i\kappa}} \,   \mathrm{d}\xi  \\ \quad &+ \sum_{j=0}^\infty \int e^{i(x\cdot \xi+t\langle \xi \rangle)}   \frac{\psi(\varepsilon |\xi|) \chi_j(|\xi|)}{( 1+|\xi|^2)^{\frac {d+1}4+i\kappa}} \, \mathrm{d}\xi. 
\end{align*}
In what follows we assume $2^j \leq 2/\varepsilon$ even if it is not explicitly mentioned. Clearly $|J_\circ|\lesssim_\kappa  \min(1, |t|^{-d/2})$ since $\det H\langle\cdot\rangle \neq 0 $ on the support of $\chi(|\cdot|)$.  Thus, it is sufficient to show that  
\Be\label{decomposed} 
|\sum_{j\ge 0}  \kappa J_j|\lesssim_\kappa |t|^{-\frac{d-1}2}.
\Ee
After rescaling, we see that 
\Be 
J_j=  2^{(d-\frac {d+1}2-2\kappa i)j}  \int e^{i(2^j x\cdot \xi+2^j t\sqrt{2^{-2j}+|\xi|^2})}   \widetilde \chi_j(|\xi|) \, \mathrm{d}\xi,
\Ee
where $\widetilde\chi_j$ is  a smooth function  supported in $[1/2, 2]$ which satisfies
$  \|\widetilde\chi_j\|_{C^N}\lesssim (1+|\kappa|)^N $  
for any $N$.   As before we may also assume that  
\[   
2^{-1} |t|\le  |x|\le 2|t|.
 \]
 Otherwise,   $|\nabla_\xi(2^j x\cdot \xi+2^j t\sqrt{2^{-2j}+|\xi|^2)}| \gtrsim  2^j\max( |x|, |t|)$.  By integration by parts we see that, 
 for any $M$, 
 \[   
 |J_j|\lesssim_\kappa 2^{\frac{d-1}{2}j}  
                           \begin{cases} (2^j|x|)^{-M}  & \text{ if } |x|\ge 2|t| , \\  
                           (2^j|t|)^{-M}  & \text{ if } |t|\ge 2|x| .
                           \end{cases}    
                              \] 
This gives   $  |  \sum_{j\ge 0}  J_j|\lesssim_\kappa  \sum_{j\ge 0} 2^{\frac{d-1}{2}j} (2^j|t|)^{-M} \lesssim_{\kappa} |t|^{-\frac {d-1}2} $.

In order to show \eqref{decomposed}  we now consider the following cases, separately
\[  \mathbf A:   |t|\ge 1,   \  \   \mathbf B:   |t|< 1. \] 
From Lemma \ref{KGKG}  we have 
\Be \label{disp-wave}|J_j|\lesssim_\kappa 2^{\frac{d-1}{2}j}\min( 2^{ -\frac{d-2}2 j} |t|^{-\frac{d}2},\, (2^j |t|)^{-\frac{d-1}2},\, 1). \Ee

\subsection*{Case $\mathbf A$} Since  $|t|\ge 1$, from \eqref{disp-wave} we observe that
\[   
\sum_{ 2^0 \le 2^j \le t } |J_j|\lesssim_\kappa \sum_{2^0\le 2^j\le t}  2^{\frac12j} |t|^{-\frac{d}2}\lesssim  |t|^{-\frac {d-1}2} . 
\] 
Let us set 
\begin{equation}\label{chit} 
 \chi_{>t}:=\sum_{|t|<2^j}  \chi_j( \cdot) 
 \end{equation}
and set 
\[
 J_{>t}=\kappa \int e^{i(x\cdot \xi+t\langle \xi \rangle)} ( 1+|\xi|^2)^{-\frac{d+1}{4}-i\kappa} \chi_{>t} (|\xi|)\psi(\varepsilon |\xi|) \, \mathrm{d}\xi.
 \] 
 Then  to show \eqref{decomposed} we are reduced to showing that 
\[ |J_{>t}|\lesssim |t|^{-\frac{d-1}2}.\] 
As before, since $|x|\sim |t|\ge 1$,  using polar coordinates,  we may apply  the asymptotic expansion \eqref{fourier-sph}. 
Taking into account the main contribution from the leading terms in \eqref{fourier-sph}  it is sufficient to consider 
\[  
|x|^\frac{1-d}{2}\int e^{i(\pm |x|\rho+t\langle \rho \rangle)}  
( 1+\rho^2)^{-\frac{d+1}{4}-i\kappa} \rho^{\frac{d-1}2} \chi_{>t}(\rho) \psi(\varepsilon\rho)\, \mathrm{d}\rho. 
\] 
The other terms can be handled similarly but they are easier.   
Since we are assuming $|x|\sim |t|$, the matter reduces to showing  
\Be\label{aha}
  \left| \int e^{i(\pm |x|\rho+t\langle \rho \rangle)}  ( 1+\rho^2)^{-\frac{d+1}{4}-i\kappa} \rho^{\frac{d-1}2}  \chi_{>t} (\rho)\psi(\varepsilon \rho ) \, \mathrm{d}\rho \right|\lesssim_\kappa  1.
  \Ee 
Since $|t|\ge 1$, from the mean value theorem we note that 
\[ \int_0^\infty | ( 1+\rho^2)^{-\frac{d+1}{4}-i\kappa} -  (\rho^2)^{-\frac{d+1}{4}-i\kappa}|  \chi_{>t} (\rho) \rho^{\frac{d-1}2}  \, \mathrm{d}\rho\lesssim_\kappa 1.\] 
Let us set  
\begin{align*}
\mathcal J&= \kappa \int e^{i(\pm |x|\rho+t\langle \rho \rangle)}  \rho^{-1-i\kappa}  \chi_{>t}(\rho) \psi(\varepsilon\rho) \, \mathrm{d}\rho.
\end{align*}
Then \eqref{aha} follows  if we show that 
\[ |\mathcal J |\lesssim_\kappa 1.\]
Changing variables, we have 
\begin{align*}
\mathcal J=\kappa |t|^{-i\kappa} \int e^{i(\pm  {|x|}|t|\rho+t\sqrt{1+t^2\rho^2})}  \rho^{-1-i\kappa}  \chi_{>1}(\rho) \psi(\varepsilon|t|\rho) \, \mathrm{d}\rho.
\end{align*} 
Now let us set 
\begin{align}
\label{ff} F_{A,t}(r) &=\kappa  |t|^{-i\kappa} \int_0^r  e^{iA\rho}  \rho^{-1-i\kappa} \chi_{>1}(\rho) \psi(\varepsilon|t|\rho) \, \mathrm{d}\rho, \\ 
 \nonumber  G(r)&=e^{it(\sqrt{1+t^2r^2}- |t| r)} \chi_{>2^{-2}}(r). 
   \end{align}
Since $\chi_{>1}\chi_{>2^{-2}}= \chi_{>1}$, we note that 
 \[   \mathcal J=\int \frac{\mathrm d}{\mathrm d\rho} F_{|t|(t\pm |x|), |t|}(\rho) G(\rho) \,  \mathrm{d}\rho= -\int F_{|t|(t\pm |x|),| t|}(\rho) G'(\rho) \, \mathrm{d}\rho.\] 
Since $\|  (\chi_{>1} \psi(|t|\varepsilon \cdot) )'\|_1\lesssim 1$\footnote{Recall that $\chi_{>1}'$ is supported in  $[-1,1]$ because 
$1-\chi=\chi_{>1}$.}, by \eqref{mod} we have $|F_{A, t}(r)|\lesssim_\kappa 1$. Thus, it is sufficient to check $\|G'\|_1\lesssim 1$. This follows from 
\[  G'(r)= e^{it(\sqrt{1+t^2r^2}-|t|r)}\left(\chi_{>2^{-2}}(r) \frac{- i t|t|}{\sqrt{1+t^2 r^2}(|t|r+\sqrt{1+t^2r^2})}  +   \chi'_{>2^{-2}}(r)\right)\] 
and the fact that  $\chi'_{>2^{-2}}$ is supported in  $[-1,1]$.

\subsection*{Case $\mathbf B$} In this case we have $|t|<1$. To begin with,  from \eqref{disp-wave} we note that 
\[ \sum_{ 2^0 \le 2^j \le \frac1{|t|}} |J_j|\lesssim_\kappa \sum_{2^0\le 2^j \le \frac1{|t|}} 2^{\frac{d-1}{2}j}\lesssim  |t|^{-\frac {d-1}2}.  \]  
We set 
\[ J_{> t^{-1}}=\int e^{i(x\cdot \xi+t\langle \xi \rangle)} ( 1+|\xi|^2)^{-\frac{d+1}{4}-i\kappa} \chi_{>t^{-1}}(|\xi|) \psi(\varepsilon |\xi|) \, \mathrm{d}\xi,
 \]
where $\chi_{>t^{-1}}$ is given by \eqref{chit}.
From the above observation, for \eqref{decomposed} it is sufficient to  show that 
\[ |J_{> t^{-1}}|\lesssim_\kappa |t|^{-\frac{d-1}2}.\] 
 Similarly as before, since $|x|\sim |t|$, $|t||\xi|\sim |x||\xi|\gtrsim 1$ on the support of  $\chi_{>t^{-1}}(\cdot) $. Thus, 
 we may apply \eqref{fourier-sph} after using polar coordinates. 
 Considering the main contribution from the leading terms
in  \eqref{fourier-sph},  we need only  to handle 
\[  |x|^\frac{1-d}{2}\int e^{i(\pm |x|\rho+t\langle \rho \rangle)}  ( 1+\rho^2)^{-\frac{d+1}{4}-i\kappa} \rho^{\frac{d-1}2} \chi_{>t^{-1}}(\rho) \psi(\varepsilon\rho) \, \mathrm{d}\rho. \] 
Thus the desired estimate \eqref{decomposed}  follows from 
\[ \Big|  \int e^{i(\pm |x|\rho+t\langle \rho \rangle)}  ( 1+\rho^2)^{-\frac{d+1}{4}-i\kappa} \rho^{\frac{d-1}2} \chi_{>t^{-1}}(\rho) \psi(\varepsilon\rho) \, \mathrm{d}\rho\Big|\lesssim_\kappa 1. \]
Let us set 
\begin{align*}
\widetilde{\mathcal J} &=\tau   \int e^{i(\pm|x|\rho+t\langle \rho \rangle)}  \rho^{-1-i\kappa}\chi_{>t^{-1}}(\rho) \psi(\varepsilon\rho)\,  \mathrm{d}\rho.
\end{align*}
Proceedings along the lines of the case $\mathbf A$, we are further  reduced to showing that 
\Be\label{fin} 
 |\widetilde{\mathcal J}| \lesssim_\kappa 1. 
\Ee 
Changing variables $\rho\to \rho/|t|$, we see that 
\begin{align*}
\widetilde{\mathcal J}  &= \kappa |t|^{i\kappa}   \int e^{i(\pm\frac {|x|} {|t|}\rho+ \text{sgn}\,t \sqrt{t^2+\rho^2})}  \rho^{-1-i\kappa}\chi_{>1}(\rho) \psi(\varepsilon \rho/|t| )\, \mathrm{d}\rho. 
\end{align*} 
 Let us set 
\[  H(r)=e^{i \text{sgn}\,t(\sqrt{t^2+r^2}-r)}  \chi_{> 2^{-2}}(r). \]
Then, 
 \[   \widetilde{\mathcal J}=\int F_{\text{sgn}\,t\pm \frac {|x|} {|t|}, |t|^{-1}}'(\rho) H(\rho)  \, \mathrm{d}\rho=
  -\int F_{\text{sgn}\,t\pm\frac {|x|} {|t|}, |t|^{-1}}(\rho) H'(\rho) \, \mathrm{d}\rho,\] 
 where $F_{A, t}$ is given by \eqref{ff}.   Since $\|  ( \chi_{>1} \psi(\varepsilon/|t| \cdot))'\|_1\lesssim 1$, 
 by \eqref{mod} $|F_{A, |t|^{-1}}(r)|\lesssim_\kappa 1$.
Finally, we note that 
\[  H'(r)= e^{i\text{sgn}\,t(\sqrt{t^2+r^2}- r)}\left(  \chi_{> 2^{-2}}(r) \frac{- i\,\text{sgn}\,t\, t^2}{\sqrt{t^2+ r^2}(r+\sqrt{t^2+r^2})}  +  \chi_{> 2^{-2}}'(r)\right).\] 
Since we are assuming $|t|< 1$, it follows  that $\|H'\|_1\lesssim 1$.  Therefore we get \eqref{fin}. 
This completes the proof of Proposition \ref{p:KGoscillatory_wave}. 
\end{proof}

\begin{proof}[Proof of Proposition \ref{p:KGoscillatory_Schro}] Compared with the proof of Proposition \ref{p:KGoscillatory_wave},  that of Proposition \ref{p:KGoscillatory_Schro} is much more involved since we  show damped oscillatory integral estimate to recover the best possible decay. We provide  the proof  by dividing it  into several steps. 

\subsection*{Step 1: Reduction to the case $|t| \gtrsim 1$ and $|x| \sim |t|$} As before, we begin with a dyadic decomposition of the integral. Using \eqref{e:dyadicdecomp}, we have 
\[
\mathcal{K}{\kappa,\varepsilon} = \mathcal{K}_{\circ} + \sum_{j \geq 0}\mathcal{K}_{j},
\]
where
\begin{align*}
\mathcal{K}_{\circ}(x,t) & := \int e^{i(x\cdot \xi+t\langle \xi \rangle)} \frac{\chi(|\xi|)\psi(\varepsilon |\xi|)}{(1+|\xi|^2)^{\frac {d+2}4+i\kappa}} \, \mathrm{d}\xi  , \\
\mathcal{K}_{j}(x,t) & := \int e^{i(x\cdot \xi+t\langle \xi \rangle)} \frac{ \chi_j(|\xi|) \psi(\varepsilon |\xi|)}{(1+|\xi|^2)^{\frac {d+2}4+i\kappa}}  \, \mathrm{d}\xi , \qquad j \geq 0.
\end{align*}
From the stationary phase method  it follows that  $
|\mathcal{K}_{\circ}(x,t)|\lesssim_\kappa \min( |t|^{-\frac{d}2},\, 1).
$
Thus, to show  \eqref{e:KGdispersivemain} it is enough  to show 
\Be
\label{step1}
\Big| \sum_{j \geq 0}\mathcal{K}_{j}(x,t)\Big| \lesssim_\kappa |t|^{-\frac d2}.
\Ee
After rescaling, we see that 
\begin{equation*} \label{ker} 
\mathcal{K}_{j}(x,t)  =  2^{(\frac{d-2}2-2i\kappa )j}  \int e^{i(2^j x\cdot \xi+2^j t\sqrt{2^{-2j}+|\xi|^2})}   \widetilde \chi_j(|\xi|) \, \mathrm{d}\xi,
\end{equation*}
where $\widetilde\chi_j$ is  a smooth function  supported in $[1/2, 2]$ which satisfies
$  \|\widetilde\chi_j \|_{C^N}\lesssim (1+|\kappa|)^N $  
for any $N$. Thus,  from Lemma \ref{KGKG} we have the following estimate, for $j \geq 1$, 
\Be 
\label{e:dispersiveKG}
|\mathcal{K}_{j}(x,t)|\lesssim_\kappa 2^{\frac{d-2}{2}j}\min( 2^{ -\frac{d-2}2 j} |t|^{-\frac{d}2},\, (2^j |t|)^{-\frac{d-1}2},\, 1).
\Ee

With the estimate \eqref{e:dispersiveKG}, 
we justify that it suffices to prove \eqref{step1} in the case where 
\begin{equation} \label{e:KGxtconditions}
|t| > 1 \qquad \text{and} \qquad 2^{-1}|t| \leq |x| \leq 2|t|.
\end{equation}
Indeed, in the case $|t| \le 1$, from \eqref{e:dispersiveKG} we have
\[
\Big| \sum_{j \geq 1}\mathcal{K}_{j}(x,t)\Big|  \lesssim_\kappa 
\sum_{j\ge 1} 2^{-\frac12j} |t|^{-\frac{d-1}2}\lesssim |t|^{-\frac{d-1}2} \le  |t|^{-\frac d2}.
\]
Now we assume $|t| > 1$ and $|x| \geq 2|t|$, in which case, we have
$
|\nabla_\xi(2^j x\cdot \xi+2^j t\sqrt{2^{-2l}+|\xi|^2}| \gtrsim  2^j |x|
$
on the support of $ \widetilde \chi_j(|\cdot|)$ and thus 
\[
|\mathcal{K}_{j}(x,t)|\lesssim_\kappa 2^{\frac{d-2}{2}j} (2^j |x|)^{-M}  \le 2^{\frac{d-2}{2}j} (2^j |t|)^{-M},
\]
which follows by repeated integration by parts. Applying this with $M > \frac{d}{2}$, we obtain
\[
|\sum_{j \ge 1} \mathcal{K}_{j} (x,t)| \lesssim_\kappa   \sum_{j \geq 1} 2^{\frac{d-2}{2}j} (2^j|t|)^{-M}  \lesssim |t|^{-M} \lesssim |t|^{-\frac d2}. 
\]
By similar arguments, we may handle the case $|t| > 1$ and $|t| \geq 2|x|$ where  $
|\nabla_\xi(2^j x\cdot \xi+2^j t\sqrt{2^{-2l}+|\xi|^2}| \gtrsim  2^j |t|$. We omit  the details. 

{\bf Step 2: Reduction to a one-dimensional oscillatory integral.}
For the remainder of the proof of Proposition \ref{p:osc-wave}, we assume that $x$ and $t$ satisfy \eqref{e:KGxtconditions}. 
Since $\sum_{j\ge 0}  \chi_j(\cdot)=1-\chi$,  to show \eqref{step1} it is sufficient to show that 
\Be  \label{jj}
|\widetilde{\mathcal{K}} (x,t)| \leq C(1 + |\kappa|)^N |t|^{-\frac d2},
\Ee
where
\[
\widetilde{\mathcal{K}} (x,t) =\int e^{i(x\cdot \xi+t\langle \xi \rangle)}  (1-\chi(|\xi|)) \frac{\psi(\varepsilon |\xi|)\mathrm{d}\xi}{( 1+|\xi|^2)^{\frac {d+2}4+i\kappa}}.
\]
By changing to spherical coordinates, we write this expression as
\[
\widetilde{\mathcal{K}} (x,t) = \int \widehat{\mathrm{d}\sigma}(\rho x) e^{it\langle \rho \rangle}  (1-\chi(\rho)) \rho^{d-1} \frac{\psi(\varepsilon \rho)\mathrm{d}\rho}{( 1+\rho^2)^{^{\frac {d+2}4+i\kappa}} }.
\]
Since $\rho \gtrsim 1$ on the support of the integrand and $|x| \gtrsim 1$, we may use the asymptotic expansion 
\eqref{fourier-sph}.  In fact, as before it is sufficient to consider the contributions from the leading terms, which take the form 
\begin{equation*} \label{e:KGafterpolars}
|x|^\frac{1-d}{2}\int e^{i(\pm |x|\rho+t\langle \rho \rangle)}  a(\rho) \, \mathrm{d}\rho,
\end{equation*}
where 
\[  
a(\rho)= C_\pm (1-\chi(\rho))\rho^{\frac{d-1}{2}} \psi(\varepsilon \rho)( 1+\rho ^2)^{-{\frac {d+2}4-i\kappa}} =O(\rho^{-\frac 32})
\]
for some constant $C_\pm$. 
We note that $a$ satisfies, for  $\ell \geq 0$, 
\Be \label{e:a'decay} 
\bigg|\bigg(\frac{\mathrm{d}}{\mathrm{d}\rho}\bigg)^\ell a(\rho)\bigg|\leq C_\ell(1 + |\kappa|)^{N_\ell} \rho^{-\frac 32-\ell}.
\Ee
Since we are considering the case $|x| \sim |t| \gtrsim 1$, in order to show the desired estimate \eqref{jj}  we are reduced to showing 
that 
\Be 
\label{e:KGreduced}
\bigg|\int e^{i(r\rho+t\langle \rho \rangle)}  a(\rho) \, \mathrm{d}\rho\bigg|  \leq C(1 + |\kappa|)^{N} |t|^{-\frac12}, \quad |r| \sim |t| \gtrsim 1
 \Ee
 for some $N$. If $r$ and $t$ have the same sign, then
$
| \frac{\mathrm{d}}{\mathrm{d}\rho}(r\rho+t\langle \rho \rangle)| \gtrsim |t|
$
holds and therefore \eqref{e:KGreduced} follows easily by integration by parts. 
Thus it is enough to consider  \eqref{e:KGreduced} for the case $r < 0$ and $t > 0$, and the case $r > 0$ and $t < 0$. 
We provide the details of \eqref{e:KGreduced} when $r < 0$ and $t > 0$ since the case $r > 0$ and $t < 0$ will follow by essentially the same argument.  

For $r, t\gtrsim 1$ we set $\widetilde{\chi}_{ <t}(\rho)=\sum_{1\le 2^j<t}  \chi_j( \cdot)$ and set 
\begin{align*}
 \mathfrak{K}(r,t)&= \int e^{i(-r\rho+t\langle \rho \rangle)}   \widetilde{\chi}_{ <t}(\rho) a(\rho)
 \mathrm{d}\rho, \\ \mathfrak{K}_j(r,t) &= \int e^{i(-r\rho+t\langle \rho \rangle)}   \chi_j(\rho)a(\rho) \, \mathrm{d}\rho.
 \end{align*}
From \eqref{e:dyadicdecomp}  we have
\[
\int e^{i(-r\rho+t\langle \rho \rangle)}  a(\rho) \, \mathrm{d}\rho  = \mathfrak{K}(r,t) +  \sum_{2^j \ge  t} \mathfrak{K}_j(r,t).
\]
From \eqref{e:a'decay}, we have trivial estimates $
|\mathfrak{K}_j(r,t)| \lesssim \int_{\rho \sim  2^j} |a(\rho)| \, \mathrm{d}\rho \lesssim 2^{-{j}/{2}}
$. Hence
\[
\sum_{2^j \ge t} |\mathfrak{K}_j(r,t)| \lesssim t^{-\frac{1}{2}}.
\]
Therefore, to show \eqref{e:KGreduced} for the case  $r < 0$ and $t > 0$ we need only to show that 
\Be
\label{e:step3goal} | \mathfrak{K}(r,t)| \leq C(1 + |\kappa|)^{N} t^{-\frac12}, \quad r \sim t \gtrsim 1.
\Ee
{ This will be taken up in Steps 3 and 4 below, corresponding to the cases $t \leq r$ and $t \geq r$.
The first case  is easier since the phase function $-r\rho+t\langle \rho \rangle$ dose not have any stationary point. However, in the latter case the phase may have stationary point, so we make use of additional  dyadic decomposition away from the stationary point. }

\subsection*{Step 3: Proof of \eqref{e:step3goal} when $t \leq r$} 
From \eqref{e:dyadicdecomp} we note that 
\begin{align}
\label{ksum}
\mathfrak{K}(r,t)=\sum_{1\le 2^j<t} \mathfrak{K}_j(r,t).
\end{align}
Note that, for $\rho \sim 2^j$, we have
$
| \frac{\mathrm{d}}{\mathrm{d}\rho}(  - r\rho+t\langle \rho \rangle)| \geq t(1 - \frac{\rho}{\langle \rho \rangle}) \gtrsim \frac{t}{2^{2j}}
$
and from \eqref{e:a'decay}   it follows that  
\Be\label{amp}
\Big\|  \frac{\mathrm d}{\mathrm d\rho} ( \chi_j(\cdot)a)\Big\|_1 \lesssim_\kappa 2^{-\frac 32j},  \quad
\Big\|  \chi_j(\cdot)a\Big\|_\infty \lesssim_\kappa 2^{-\frac 32j}.
\Ee Thus, by Proposition \ref{p:vandercorput}, we have
\Be
\label{ahha}
|\mathfrak{K}_j(r,t)| \lesssim_\kappa  \frac{2^{{j}/{2}}}{t}.
\Ee
From this  \eqref{e:step3goal} immediately follows.

\subsection*{Step 4: Proof of \eqref{e:step3goal} when $t> r$}
We set
\[  
\rho_\ast =\frac{r}{\sqrt{t^2-r^2}} \,
\] 
which is the stationary point of the phase function $-r\rho+t\langle \rho \rangle$,
and note that $\rho_\ast \ge  1/\sqrt 3$ since $2r \geq t$. Also, we may write 
\Be \label{deriv} 
\frac {\mathrm{d}}{\mathrm{d}\rho}(  - r\rho+t\langle \rho \rangle)= -r +\frac{t\rho}{\langle \rho \rangle}=  t\int^\rho_{\rho_\ast} \frac{\mathrm{d}s}{(1+s^2)^{\frac 32}}.
\Ee
Now we distinguish  the cases: 
\[
\text{($i$) $\rho_*\ge  2^2t$ \quad and \quad ($ii$) $\rho_* \in ( 2^{-1},2^2t)$}. 
\]

{\bf Case $(i)$.}  This case is easier to handle since  the stationary point $\rho_\ast$ does not appear in the region of integration. 
 We estimate each of $\mathfrak K_j$ in \eqref{ksum}. For each $\mathfrak K_j$ we may clearly  assume $\rho \leq 2t$. 
If $\rho \sim 2^j$ and $\rho \leq 2t$, then
\begin{align*}
\bigg|\frac {\mathrm{d}}{\mathrm{d}\rho}(- r\rho+t\langle \rho \rangle)\bigg| \sim t  \int^{\rho_\ast}_\rho \frac{\mathrm{d}s}{s^3} \gtrsim \frac{t}{\rho^2} \sim 2^{-2j} t.
\end{align*}
Therefore, by Proposition \ref{p:vandercorput} and \eqref{amp}, we again obtain
\eqref{ahha}, from which  \eqref{e:step3goal} follows.

{\bf Case $(ii)$.}  The remainder of Step 4 is devoted to proof of   \eqref{e:step3goal} for Case $(ii)$ and here we split
\[
\mathfrak{K}(r,t) = \int_{\rho_\ast}^\infty  e^{i(-r\rho+t\langle \rho \rangle)}   \widetilde{\chi}_{<t}(\rho ) a(\rho)\, \mathrm{d}\rho + \int_{0}^{\rho_\ast}  e^{i(-r\rho+t\langle \rho \rangle)}   \widetilde{\chi}_{<t}(\rho ) a(\rho)\, \mathrm{d}\rho,
\]
and make decomposition away from the stationary point $\rho_\ast$. 
For the integral over $(\rho_\ast,\infty)$, we make use of a dyadic decomposition to write
\begin{align*}
 \int_{\rho_\ast}^\infty  e^{i(-r\rho+t\langle \rho \rangle)}  \widetilde{\chi}_{<t}(\rho ) a(\rho)\, \mathrm{d}\rho  = \sum_{ 2^j \geq \rho_\ast }  \mathfrak K_j^{\rho_\ast}
 +  \sum_{2^j < \rho_\ast} \mathfrak K_j^{\rho_\ast}, 
\end{align*}
where 
\[ \mathfrak K_j^{\rho_\ast}(r,t)=  \int_{\rho_\ast}^\infty  e^{i(-r\rho+t\langle \rho \rangle)}   \widetilde{\chi}_{<t}(\rho ) a(\rho) \chi_j(\rho - \rho_\ast)\, \mathrm{d}\rho.\] 
We first consider the former sum $\sum_{ 2^j \geq \rho_\ast } \mathfrak K_j^{\rho_\ast}$, which is easier to handle. 
If $2^j \geq \rho_\ast$ and $\rho - \rho_\ast \sim 2^j$, from \eqref{deriv} it follows that
\[
\bigg|\frac {\mathrm{d}}{\mathrm{d}\rho}(- r\rho+t\langle \rho \rangle)\bigg| \gtrsim  t \int_\rho^{\rho^\ast}  \frac{{\mathrm d}s}{s^3} \gtrsim  \frac{t2^j}{\rho \rho_\ast^2} \gtrsim t 2^{-2j}.
\]
Since $\rho\sim 2^{j}$ on the support of integrand, from \eqref{e:a'decay} it is easy to see that 
$\|   \widetilde{\chi}_{<t}\, a\,\chi_j(\cdot - \rho_\ast)\|_\infty\lesssim 2^{-3j/2}$ and $\|\frac{\mathrm d}{\mathrm d\rho}( \widetilde{\chi}_{<t}\, a\, \chi_j(\cdot - \rho_\ast))  \|_1\lesssim_\kappa 2^{-3j/2}$. Thus, by Proposition \ref{p:vandercorput} we have 
$ $
\[|\mathfrak K_j^{\rho_\ast}(r,t)| \lesssim 2^{{j}/{2}}t^{-1}.
\]
Therefore
\Be\label{bigj}
\sum_{ 2^j \geq \rho_\ast} |\mathfrak K_j^{\rho_\ast}(r,t)|\lesssim t^{-\frac{1}{2}}.
\Ee
Here we also use $2^j\lesssim  t$ (otherwise, the integral is zero).

Now we consider the second sum $\sum_{ 2^j < \rho_\ast } \mathfrak K_j^{\rho_\ast}$. 
Since $2^j < \rho_\ast$ and $\rho - \rho_\ast \sim 2^j$, using \eqref{deriv}  we have
\[
\bigg|\frac {\mathrm{d}}{\mathrm{d}\rho}(- r\rho+t\langle \rho \rangle)\bigg| \sim \frac{t2^j}{\rho \rho_\ast^2} \gtrsim \frac{t2^j}{\rho_\ast^3}.
\]
We also note that $\rho\sim \rho_\ast$ on the support of integrand. From \eqref{e:a'decay} it follows that 
\Be 
\label{rho*}
\|   \widetilde{\chi}_{<t}\, a\,\chi_j(\cdot - \rho_\ast)\|_\infty\lesssim \rho_\ast^{-3/2},   \   \  \Big\|\frac{\mathrm d}{\mathrm d\rho}( \widetilde{\chi}_{<t}\, a\, \chi_j(\cdot - \rho_\ast)) \Big \|_1\lesssim_\kappa \rho_\ast^{-3/2} .
\Ee Hence, Proposition \ref{p:vandercorput} and the trivial estimate $|\mathfrak K_j^{\rho_\ast}(r,t)| \lesssim \rho_\ast^{-\frac{3}{2}} 2^{j}$ yield 
\[
|\mathfrak K_j^{\rho_\ast}(r,t)| \lesssim_\kappa \min(\rho_\ast^{\frac{3}{2}} 2^{-j}t^{-1}, \rho_\ast^{-\frac{3}{2}} 2^{j}).
\]
Therefore, by splitting the summation further into the cases $2^j \geq \rho_\ast^{\frac{3}{2}}t^{-\frac{1}{2}}$ and $2^j < \rho_\ast^{\frac{3}{2}}t^{-\frac{1}{2}}$, we obtain \begin{align*} \sum_{ 2^j <\rho_\ast} |\mathfrak K_j^{\rho_\ast}(r,t)|\lesssim_\kappa t^{-\frac{1}{2}}.\end{align*}
Thus, combining this with \eqref{bigj}  we get 
\[
\bigg|\int_{\rho_\ast}^\infty  e^{i(-r\rho+t\langle \rho \rangle)}   \widetilde{\chi}_{<t}(\rho ) a(\rho)\, \mathrm{d}\rho \bigg| \lesssim_\kappa t^{-\frac{1}{2}}.
\]

To complete the proof, it remains to show
\begin{equation}
\label{e:desired}
\bigg|\int_0^{\rho_\ast}  e^{i(-r\rho+t\langle \rho \rangle)}   \widetilde{\chi}_{<t}(\rho ) a(\rho)\, \mathrm{d}\rho \bigg| \lesssim_\kappa t^{-\frac{1}{2}}.
\end{equation}
Let set  
$
\widetilde{\chi}_{\rho_\ast}(\rho) = 1 - \sum_{2^j \le 2^{-2} \rho_\ast}  \chi_j(\rho_\ast - \rho).
$
As before we break the integral 
\[\int^{\rho_\ast}_0  e^{i(-r\rho+t\langle \rho \rangle)}   \widetilde{\chi}_{<t}(\rho ) a(\rho)\, \mathrm{d}\rho
=\sum_{2^j\le 2^{-2}\rho_\ast}  \mathfrak K_{\rho_\ast, j}(r,t) + \mathfrak K_{<\rho_\ast}(r,t), \]
where 
\begin{align*}
\mathfrak K_{\rho_\ast, j}(r,t)&=  \int_0^{\rho_\ast} e^{i(-r\rho+t\langle \rho \rangle)}   \widetilde{\chi}_{<t}(\rho ) a(\rho) \chi_j(\rho - \rho_\ast)\, \mathrm{d}\rho,\\
\mathfrak K_{< \rho_\ast}(r,t)&=  \int_0^{\rho_\ast} e^{i(-r\rho+t\langle \rho \rangle)}   \widetilde{\chi}_{<t}(\rho ) a(\rho) \widetilde  \chi_{\rho_\ast}(\rho)\, \mathrm{d}\rho.
\end{align*}

For $2^j\le 2^{-2}\rho_\ast$, $\rho\in [ 2^{-1}\rho_\ast, \rho_\ast]$ whenever $\rho$ is contained in the support of the integrand of the integral $\mathfrak K_{\rho_\ast, j}(r,t)$.  Thus the sum $\sum_{2^j\le 2^{-2}\rho_\ast}  \mathfrak K_{\rho_\ast, j}(r,t)$ can be handled as it was done for  $\sum_{ 2^j < \rho_\ast } \mathfrak K_j^{\rho_\ast}$.  Indeed,  note that $|\frac {\mathrm{d}}{\mathrm{d}\rho}(- r\rho+t\langle \rho \rangle)| \sim \frac{t2^j}{\rho^2 \rho_\ast} \gtrsim \frac{t2^j}{\rho_\ast^3}$ and we also have \eqref{rho*}. Thus, similarly as before,  by Proposition \ref{p:vandercorput} and the trivial estimate   it follows that $
|\mathfrak K_{\rho_\ast,j}(r,t)| \lesssim_\kappa\min(\rho_\ast^{\frac{3}{2}} 2^{-j}t^{-1}, \rho_\ast^{-\frac{3}{2}} 2^{j}).
$
This gives 
\[\sum_{2^j\le 2^{-2}\rho_\ast} | \mathfrak K_{\rho_\ast, j}(r,t)|\lesssim_\kappa |t|^{-\frac12}.\] 
Thus,  to show \eqref{e:desired}  it remains to show that 
\begin{equation} \label{e:finalcaseKGwave}
|\mathfrak K_{<\rho_\ast}(r,t)| \lesssim_\kappa t^{-\frac{1}{2}}.
\end{equation}
 We make a dyadic decomposition away from the origin. Let us set 
\[ \mathcal A_\ell =\widetilde{\chi}_{<t}(\rho ) a(\rho)\widetilde{\chi}_{\rho_\ast}(\rho) \chi_\ell({\rho}).\]
We may write 
\Be
\label{ahaha} 
\mathfrak K_{<\rho_\ast}(r,t)= \sum_{1\lesssim 2^\ell \lesssim \rho_\ast} \int_{0}^{\rho_\ast}  e^{i(-r\rho+t\langle \rho \rangle)} \mathcal A_\ell (\rho)\,   \mathrm{d}\rho.
\Ee
We note that $\widetilde{\chi}_{<t}\widetilde{\chi}_{\rho_\ast}\chi_{(0,\rho_\ast)}$ is supported in the interval  $[2^{-1},(1-2^{-4})\rho_\ast]$.
Since $\rho\sim 2^{l} $ and $\rho_\ast-\rho\sim \rho_\ast$ on the support of $\mathcal A_\ell$,     by \eqref{deriv}, we have 
\[
\bigg|\frac {\mathrm{d}}{\mathrm{d}\rho}(- r\rho+t\langle \rho \rangle)\bigg| \gtrsim \int_{\rho}^{\rho_\ast} \frac{{\mathrm{d}} s}{s^3}\sim \frac{t}{\rho^2} \sim \frac{t}{2^{2\ell}}
\]
provided $\rho$ is contained in the support of $\mathcal A_\ell$. Also  from  \eqref{e:a'decay} it follows that 
$\| \mathcal A_\ell\|_\infty \lesssim 2^{-3\ell/2}$ and $\| \frac {\mathrm{d}}{\mathrm{d}\rho} \mathcal A_\ell\|_1\lesssim 2^{-3\ell/2}$. By making use of Proposition \ref{p:vandercorput} we get 
\[
\bigg|\int_{0}^{\rho_\ast}  e^{i(-r\rho+t\langle \rho \rangle)} \mathcal A_\ell (\rho)\,   \mathrm{d}\rho\bigg| \lesssim 2^{\frac{1}{2}\ell}t^{-1}.
\]
Since $2^\ell \le \rho_\ast \lesssim t$, by this and \eqref{ahaha} we now obtain \eqref{e:finalcaseKGwave}. This completes our proof of Step 4 and thus Proposition \ref{p:KGoscillatory_Schro}.
\end{proof}

\begin{remark}
\label{more-damping} The estimates \eqref{KGwavedecay} and  \eqref{e:KGdispersivemain} continue to be valid with higher order of damping.
Such estimates can be shown  without difficulty by following the argument 
in the proofs of Proposition \ref{p:KGoscillatory_wave} and Proposition \ref{p:KGoscillatory_Schro}, so we  state the estimates without proof.  For $\gamma> 0$, 
there exist constants  $C < \infty$ and $N > 0$ such that 
\begin{align}
\label{e:KGwave-more-damping}
\Big| \int e^{i(x\cdot \xi+t\langle \xi \rangle)}  \psi(\varepsilon \xi) ( 1+|\xi|^2)^{-\frac {d+1}4-\gamma -i\kappa } \, \mathrm{d}\xi\Big| & \leq C(1+|\kappa|)^{N}|t|^{-\frac {d-1}2}, 
\\
\label{e:KGschro-more-damping}
\Big| \int e^{i(x\cdot \xi+t\langle \xi \rangle)}  \psi(\varepsilon \xi) ( 1+|\xi|^2)^{-\frac {d+2}4-\gamma -i\kappa } \, \mathrm{d}\xi\Big| & \leq C(1+|\kappa|)^{N}|t|^{-\frac {d}2}
\end{align}
for all $\kappa \in \mathbb{R}$, $\varepsilon>0$ and $(x,t) \in \mathbb{R}^d \times \mathbb{R}$. 
\end{remark}

\subsection{The fractional Schr\"odinger equation} 
We now provide the damped oscillatory integral estimates which are key to the proof of Theorem \ref{t:fracSchrosharp}. 
For our purpose we only need to consider the oscillatory integrals with phases of the form $x\cdot\xi+t|\xi|^\alpha$ but our method here admits extension  of $|\xi|^\alpha$ to more general symbols $\phi$ which are basically perturbations of a homogeneous function. To this end, we first introduce the notion of an \emph{almost homogeneous phase}   which is inspired by the condition introduced by Kenig--Ponce--Vega \cite[Lemma 3.4] {KenigPonceVega}.

\begin{definition}
Let $\alpha\in\mathbb{R}\setminus\{0,1\}$. We say $\phi\in \mathfrak R(\alpha, B, \lambda,  d_1, d_2)$ if   $\psi\in C^{d+2}(\mathbb R^d\setminus\{0\}) $ and  there
are positive  constants $d_1, d_2, \lambda$, and $B$, such that, for all $\xi\in \mathbb R^d\setminus\{0\}$,
\begin{align}
d_1|\xi|^{\alpha-1} &\le |\nabla\phi(\xi)| \le d_2|\xi|^{\alpha-1}, \label{e:first} 
\\
|\partial^\gamma \phi(\xi)| &\le B |\xi|^{\alpha-|\gamma|},\quad   | \gamma|\le d+2, \label{e:second} 
\\
\lambda |\xi|^{\alpha-2} &\le |v^{\rm t} H\phi (\xi) v|,    
\quad   \forall  v\in \mathbb{S}^{d-1} \label{e:third}.
\end{align}
We also say  $\phi$ is \emph{almost homogeneous of order $\alpha$} if $\phi\in \mathfrak R(\alpha, B, \lambda,  d_1, d_2)$ for some positive constants $d_1, d_2, \lambda$, and $B$.
\end{definition} 

Kenig--Ponce--Vega \cite{KenigPonceVega} considered phase functions $\phi$ satisfying \eqref{e:first}, \eqref{e:second} and 
\begin{equation}\label{e:KPVcondi}
 \lambda_1 |\xi|^{d(\alpha-2)} \le  | {\rm det} H\phi (\xi)| \le \lambda_2 |\xi|^{d(\alpha-2)}
\end{equation}
for some $\lambda_1, \lambda_2$ instead of \eqref{e:third}. It is easy to see that  \eqref{e:second} and  \eqref{e:third} imply  \eqref{e:KPVcondi}. 
Indeed,    \eqref{e:second}    implies that all the absolute values of the eigenvalues of ${\rm det} H\phi (\xi)$ are at most $d^2 B|\xi |^{\alpha-2}$ and  from  \eqref{e:third}  we 
see the absolute values of all eigenvalues of ${\rm det} H\phi (\xi)$ are at least $\lambda |\xi|^{\alpha-2}.$ Thus we get \eqref{e:KPVcondi} 
with $\lambda_2=(d^2B)^d$ and $\lambda_1=\lambda^d$.  The converse is also not difficult to see.  For this we use  the following  
simple observation which is also useful in what follows. This can be shown by direct computation, so we omit the proof.

\begin{lemma}\label{l:phaseScale}  Let $\alpha \in \mathbb{R}\setminus\{0,1\}$ and $\rho>0$. Set 
\[
\phi_\rho^{\alpha}(\xi) = \rho^{-\alpha} \phi(\rho \xi).
\]
 Then, if $\phi$ satisfies \eqref{e:first}, then so does $\phi_\rho^{\alpha}$. The same is true with the conditions 
  \eqref{e:second},
 \eqref{e:third}, 
\eqref{e:KPVcondi}. That is to say, for  any $\rho>0$, $\phi\in \mathfrak R(\alpha, B, \lambda,  d_1, d_2)$ if and only if $\phi_\rho^{\alpha} \in \mathfrak R(\alpha, B, \lambda,  d_1, d_2)$.
\end{lemma}

\subsubsection*{Proof of implication from \eqref{e:KPVcondi} to \eqref{e:third}}
Now we show \eqref{e:KPVcondi} implies \eqref{e:third} under the condition \eqref{e:second}. Let $\xi\in \mathbb R^d\setminus \{0\}$ and
suppose $|\xi|\in (\lambda, 2\lambda]$ for some $\lambda>0$, then the conditions \eqref{e:second}, \eqref{e:third} and  \eqref{e:KPVcondi} are invariant under $\phi\to \phi_\lambda^\alpha$  by Lemma \ref{l:phaseScale}. Thus it is enough to show that, 
for $|\xi|\in (1,2]$,  the conditions \eqref{e:second} and \eqref{e:KPVcondi} imply 
\eqref{e:third}.  
By the condition, \eqref{e:second} we see the eigenvalues  $\mu_1, \mu_2,\dots, \mu_d$ of the matrix $H\phi(\xi)$ satisfy $|\mu_i|\le b:=d^2 B\max(2^{\alpha-2},1)$ for $i=1,\dots,d$.  Since $\mu_1\mu_2\cdots\mu_d=\det H\phi(\xi)$ and $\lambda_1\min(2^{\alpha-2},1)\le |\det H\phi(\xi)|$, it follows that
\[ b^{1-d} \lambda_1\min(2^{\alpha-2},1) \le  |\mu_i|.\]
 Since $\min \mu_i\le v^{\rm t} H\phi (\xi) v\le \max \mu_i$ for $v\in \mathbb S^{d-1}$, we see that 
 \eqref{e:third} holds with $\lambda=b^{1-d} \lambda_1\min(2^{\alpha-2},1)(\max(2^{\alpha-2},1))^{-1}$. This completes the proof of the implication. \qed

For a given $\phi$ which is almost homogeneous of order $\alpha \in \mathbb{R} \setminus \{0,1\}$, the definition of the relevant oscillatory integral depends on the sign of $\alpha$. If $\alpha > 0$,  for each $\kappa \in \mathbb{R}$ and $(x,t) \in \mathbb{R}^d \times \mathbb{R}$, we define the oscillatory integral 
\begin{equation*}
\mathcal{I}^{\phi,+}_{\varepsilon,\kappa}(x,t) = \int_{\mathbb{R}^{d}} e^{i(x\cdot\xi + t\phi(\xi))} \chi(\varepsilon \xi) |{\rm det} H\phi(\xi)|^{1/2 + i\kappa}\, \mathrm{d}\xi.
\end{equation*}
If $\alpha < 0$, for each $\kappa \in \mathbb{R}$ and $(x,t) \in \mathbb{R}^d \times \mathbb{R}$, 
we define the oscillatory integral
\begin{equation*}
\mathcal{I}^{\phi,-}_{\varepsilon,\kappa}(x,t) =  
    \int_{\mathbb{R}^{d}} e^{i(x\cdot\xi + t\phi(\xi))} \chi_\infty (\varepsilon^{-1}\xi) |{\rm det} H\phi(\xi)|^{1/2 + i\kappa}\, \mathrm{d}\xi,
\end{equation*}
where $\chi_\infty := 1-\chi$.  
Note that the cutoff function $\chi$ has compact support  and $\chi_\infty$ is supported away from the origin. 
Observe also that 
\Be\label{hf}
|{\rm det} H\phi(\xi)|^{1/2} 
\sim |\xi|^{\frac{\alpha-2}{2}d}
\Ee
for $\alpha \neq 0,1$. Thus,  $\mathcal{I}^{\phi,+}_{\varepsilon,\kappa}$,  $\mathcal{I}^{\phi,-}_{\varepsilon,\kappa}$  are well defined for the cases $\alpha>0$, $\alpha<0$, respectively.

 Our main oscillatory integral estimate is as follows.

\begin{proposition}\label{p:OsciAlmostRad}
Let $\alpha\in\mathbb{R}\setminus\{0,1\}$ and $\phi$ be almost homogeneous of order $\alpha$. Then there exist $C < \infty$ and $N > 0$ such that 
\begin{align}
\label{near-away}
|\mathcal{I}^{\phi,\pm }_{\varepsilon,\kappa}(x,t)| &\le C(1+|\kappa|)^N |t|^{-d/2},   \   \pm \alpha>0,
\end{align}
 for all $\varepsilon > 0$, $\kappa\in\mathbb{R}$ and $(x,t) \in \mathbb{R}^d \times \mathbb{R}$.
\end{proposition}
 
\begin{remark} As is to be clearly seen from the proof,  for the estimate \eqref{near-away} with 
 $+$ it is enough to have the conditions  
 \eqref{e:first}, \eqref{e:second}, and \eqref{e:third} for $|\xi|\le 2\varepsilon^{-1}$. Similarly, for 
 \eqref{near-away} with $-$,  we only need consider $|\xi|\ge 2\varepsilon$. Also, we remark that Proposition \ref{p:OsciAlmostRad} has been already  proved by Kenig, Ponce and Vega \cite [Lemma 3.4]{KenigPonceVega} when $\alpha\ge2$.
 \end{remark}

 The proof of  Proposition \ref{p:OsciAlmostRad}  heavily relies on 
 stability of the oscillatory integral estimates which are obtained by the stationary phase method.  For example, 
see  \cite[p. 220, Theorem 7.7.5]{Hormanderbook} and \cite[Theorem 1]{ABZ}.

\begin{lemma}\label{l:station}  Let $\alpha \in \mathbb{R}\setminus\{0,1\}$. Suppose $\phi\in\mathfrak R(\alpha,  B, \lambda, d_1, d_2)$ for some positive constants $  \lambda, d_1, d_2$ and  $\psi$ is a smooth function supported in the set $\{\xi:  2^{-1}\le |\xi|\le 2\}$. Then there exists a constant $C < \infty$ depending only on $\lambda,  B, d_1, d_2$, such that 
\Be
\label{stationary}
\bigg| \int_{\mathbb{R}^d} e^{i(x\cdot\xi + t \phi(\xi))} \psi(\xi)\, \mathrm{d}\xi \bigg| \le C\|\psi\|_{C^{d+1}} |t|^{-d/2}.
\Ee
\end{lemma}

This may be seen as a simple consequence  of \cite[Theorem 1]{ABZ}. Since \eqref{e:third} gives uniform lower bounds for the eigenvalues of the Hessian matrix, the lemma  also can be proved by making use of the standard argument which relies on the Morse Lemma. 

\begin{proof}  Since $\psi$ is supported in the set $\{\xi:  2^{-1}\le |\xi|\le 2\}$, from \eqref{e:first} it is easy to see that 
there are  $c_1, c_2$, $b_1,$ and $b_2>0$, depending only on $d_1, d_2$,  such that  $|\nabla_\xi (x\cdot\xi + t \phi(\xi))|\ge  c_1|t|$ 
if $|t|>b_1|x|$ and $|\nabla_\xi (x\cdot\xi + t \phi(\xi))|\ge  c_2|x|$ if $|x|>b_2|t|$. Thus, in either case, by integration by parts we get 
\eqref{stationary} with $C$ depending only on  $c_1, c_2$, and $B$. Thus we may assume $ b_2^{-1} |x|\le |t|\le  b_1|x|$.
We may rewrite the estimate as
$$
\bigg| \int_{\mathbb{R}^d} e^{it\Phi_{x,t}(\xi)}\psi(\xi)\, \mathrm{d}\xi \bigg| \le C |t|^{-d/2}
$$
with  $\Phi_{x,t}(\xi) =t^{-1}x\cdot \xi+\phi(\xi)$. We note that $\Phi_{x,t}$ is uniformly bounded in  $C^{d+2}$ and $|\!\det(H\Phi_{x,t})|$ has a nonzero uniform lower bound. 
Once we have ensured these conditions, then we may employ the stationary phase method for the oscillatory integral (see \cite[Theorem 1]{ABZ}\footnote{Here we need the condition
$\phi\in C^{d+2}$.}, or alternatively Theorem 7.7.5 in \cite{Hormanderbook}, and Theorem 3 and Example 4 in  \cite{ABZ}, for more explicit statements), to have the desired uniform bound.  
\end{proof} 
 
 \begin{proof}[Proof of Proposition \ref{p:OsciAlmostRad}]
First of all, we observe that it suffices to show the estimate when $\varepsilon =1$.
Let $\phi\in  \mathfrak R(\alpha, B_,\lambda, d_1, d_2)$ for some positive constants $\alpha, B_,\lambda, d_1, d_2$. 
Since $\det H\phi_\rho^\alpha(\xi)= \rho^{d(2-\alpha)}\det H\phi( \rho \xi)$ for any $\rho>0$,
after rescaling $\xi\to \varepsilon^{\mp 1}\xi$,  we have
\[\mathcal{I}^{\phi,\pm}_{\varepsilon,\kappa}(x,t)=  \varepsilon^{\pm [-\frac{\alpha d}{2}+ i{d(2-\alpha)\kappa}]}
\mathcal{I}^{\phi_{\varepsilon^{\mp 1}}^\alpha,\pm}_{1,\kappa}(\varepsilon^{\mp 1} x,\varepsilon^{\mp \alpha} t), \,\, \pm\alpha >0.\]
By Lemma  \ref{l:phaseScale}, $\phi_{\varepsilon^{\mp 1}}^\alpha \in \mathfrak R(\alpha, B,\lambda, d_1, d_2)$.  
Thus, the desired estimate follows if we show 
\Be 
\label{normalized}  |\mathcal{I}^{\phi,\pm}_{1,\kappa}(x,t)|\lesssim_\kappa |t|^{-d/2},
\Ee whenever  $\phi\in  \mathfrak R(\alpha, B_,\lambda, d_1, d_2)$.    
Since $|\mathcal{I}^{\phi,\pm}_{1,\kappa}|\lesssim 1$ by \eqref{hf},  we may assume $|t|>1$, otherwise the estimate is trivial.  
To show  \eqref{normalized}  we need to consider two cases, $\alpha>0$ and $\alpha<0$, separately. 

\subsection*{The case $\alpha > 0$: estimate for $\mathcal{I}^{\phi,+}_{1,\kappa}$}  We use the dyadic decomposition $\chi = \sum_{j=1}^\infty \chi_0(2^{j}\cdot)$ (recall \eqref{e:dyadicdecomp}). By changing variables,  we write  
\[
\mathcal{I}^{\phi,+}_{1,\kappa}=\sum_{j=1}^\infty I_{j}:=\sum_{j=1}^\infty  C_{\kappa,j} 2^{-\frac{d\alpha}{2}j}  \int e^{i(2^{-j}x\cdot \xi+ 2^{-\alpha j} t  \phi_j(\xi) )
}  |{\rm det} H\phi_j(\xi)|^{1/2 + i\kappa} \chi_0(|\xi|) \,\mathrm{d}\xi, 
\]
where $\phi_j := \phi^\alpha_{2^{-j}}$ (recall the notation in Lemma \ref{l:phaseScale}) and $C_{\kappa,j} $ is  a complex number with $|C_{\kappa,j} |=1$. Let us set 
\[ 
a_1=2^{-1} d_1 \min(2^{1-\alpha},2^{\alpha-1}),\quad   a_2=2 d_2 \max(2^{1-\alpha},2^{\alpha-1}).   
\]                                           
We consider the following three cases:
\begin{align*}
\mathbf {A}&: \quad 2^{(\alpha-1)j}|x|\ge a_2 |t |, 
\\
 \mathbf {B}&: \quad  a_1 |t|<  2^{(\alpha-1)j}|x| <a_2 |t |, 
 \\
 \mathbf {C}&:  \quad 2^{(\alpha-1)j}|x|\le a_1 |t |.
  \end{align*}
In case $\mathbf B$, there are only  finitely many  $j$.  By Lemma \ref{l:phaseScale},  $\phi_j\in  \mathfrak R(\alpha, B_,\lambda, d_1, d_2)$. Thus,  applying 
Lemma \ref{l:station}   to each $I_j$,  we see 
$$
\sum_{j\in \mathbf B} |I_j| \le C \sum_{j\in \mathbf B} 2^{-\frac{d\alpha}{2}j} |2^{-\alpha j}t|^{-d/2} \sim |t|^{-d/2}. 
$$
For case $\mathbf C$, 
since  $ \phi_j$ satisfies \eqref{e:first}, we have 
\begin{align*}
|\nabla_\xi  (2^{-j}x\cdot \xi+ 2^{-\alpha j} t  \phi_j(\xi))| 
&\ge    2^{-\alpha j} d_1 \min(2^{1-\alpha},2^{\alpha-1}) | t|  -
2^{-j}|x|
\ge  {a_1}
2^{-\alpha j} |t|
\end{align*}
on the support of $\chi_0$. So,  integration by parts gives  
$|I_j| \lesssim_\kappa 2^{-\frac{d\alpha}2 j}{\rm min}\,(|2^{-\alpha j}t|^{-N},1)$ for $N\le d+2$. Since $\alpha>0$, taking $N>d/2$,  we get  
$$
\sum_{j\in \mathbf C} |I_j| \lesssim_\kappa \sum_{j\in \mathbf C} 2^{-\frac{d\alpha}{2} j}{\rm min}\,(|2^{-\alpha j}t|^{-N},1) \lesssim |t|^{-d/2}.
$$

We now consider case $\mathbf A$. 
As before, we know from \eqref{e:first}, Lemma \ref{l:phaseScale} and assumption $j\in \mathbf A$ that 
$$
|\nabla_\xi  (2^{-j}x\cdot \xi+ 2^{-\alpha j} t  \phi_j(\xi))| 
\ge 
2^{-j}|x|-2^{-\alpha j} d_2\max (2^{1-\alpha},2^{\alpha-1})  |t|
\ge 
2^{-1}2^{-j} |x|.
$$  
Integration by parts yields $ |I_j| \lesssim_\kappa 2^{-\frac {\alpha d}2 j} {\rm min}\, (|2^{-j}x|^{-M},1)$ for any $M\le d+2$. 
In particular, taking $M=d/2$ we get
\begin{equation}\label{e:NonStational} 
|I_j| \lesssim_\kappa  {\rm min}\, (2^{\frac d2 (1-\alpha)j} |x|^{-\frac d2}, 2^{-\frac {\alpha d}2 j}).
\end{equation}
To sum each estimate, we need to separate two cases $0<\alpha<1$ and $\alpha>1$.
First, consider $0<\alpha<1$. In this case, $ 2^j \lesssim  (|x|/|t|)^\frac{1}{1-\alpha}$ and
using  \eqref{e:NonStational} we have
\begin{align*}
\sum_{j\in \mathbf A} |I_j| \lesssim_\kappa\, &
|x|^{-\frac d2} \sum_{ 1\le 2^j\lesssim (|x|/|t|)^\frac{1}{1-\alpha}}2^{\frac d2 j(1-\alpha)} 
\sim |t|^{-\frac d2}.
\end{align*}
For the case $\alpha>1$, $2^j\gtrsim   (|t|/|x|)^{1/(\alpha-1)}$. Thus, by \eqref{e:NonStational},   we have 
\begin{align*}
\sum_{j\in \mathbf A}|I_j| \lesssim_\kappa \sum_{2^j\gtrsim   
(|t|/|x|)^{1/(\alpha-1)}}^\infty  2^{\frac d2 j(1-\alpha)} |x|^{-\frac d2}\lesssim |t|^{-\frac d2}.
\end{align*}
This completes the proof of \eqref{normalized} with $+$ for the case $\alpha>0$.

\subsection*{The case $\alpha < 0$: estimate for  $\mathcal{I}^{\phi,-}_{1,\kappa}$ }  This can be handled in a similar manner so we shall be brief.  Using the dyadic decomposition $\chi_\infty = \sum_{j=0}^\infty  \chi_j(\cdot)$ and changing variables,  we write
\[
\mathcal{I}^{\phi,-}_{1,\kappa}=\sum_{j=0}^\infty \widetilde{I_{j}}:=\sum_{j=0}^\infty \widetilde C_{\kappa,j} 2^{\frac{d\alpha}{2}j}   \int e^{i(2^{j}x\cdot \xi+ 2^{\alpha j} t  {\widetilde \phi_j}(\xi) )  
}  |{\rm det} H{\widetilde \phi_j}(\xi)|^{1/2 + i\kappa} \chi_0(|\xi|) \,\mathrm{d}\xi, 
\]
where ${\widetilde \phi_j} = \phi^\alpha_{2^j}$ and $\widetilde C_{\kappa,j} $ is  a complex number with $|\widetilde C_{\kappa,j} |=1$. As before, we consider the following three cases:
\begin{align*}
\mathbf {\widetilde A}&: \quad 2^{(1-\alpha)j}|x|\ge a_2 |t |, 
\\
 \mathbf {\widetilde B}&: \quad  a_1 |t|<  2^{(1-\alpha)j}|x| <a_2 |t |, 
 \\
 \mathbf {\widetilde C}&:  \quad 2^{(1-\alpha)j}|x|\le a_1 |t |.
  \end{align*}

For case $\mathbf {\widetilde B}$ there are only finitely many $j$.  By  Lemma \ref{l:phaseScale} and Lemma \ref{l:station} it follows that 
$$
\sum_{j \in \mathbf{\widetilde B}} |\widetilde{I_j}| \le C \sum_{j\in \mathbf{\widetilde B}} 2^{\frac{d\alpha}{2}j} |2^{\alpha j}t|^{-d/2} \sim |t|^{-d/2}.
$$
In case $\mathbf{\widetilde C}$, similarly as before   we see
$
|\nabla_\xi  (2^{j}x\cdot \xi+ 2^{\alpha j} t  \widetilde{\phi}_j(\xi))| 
\gtrsim 
2^{\alpha j} |t|.
$
Hence, integration by parts gives  $ |\widetilde{I_j}| \lesssim  2^{\frac{d\alpha}{2}j} {\rm min}\, (|2^{\alpha  j}t|^{-N},1)$ for $N\le d+2$.  Since $\alpha<0$,  again taking $N>\tfrac d2$ and splitting the sum into the cases  $2^j\le |t|^{-1/\alpha}$ and $2^j> |t|^{-1/\alpha}$, we see 
$$
\sum_{j\in \mathbf { \widetilde C}} |\widetilde{I_j}| \lesssim_\kappa \sum_{j=0}^\infty 2^{\frac{d\alpha}{2}j} {\rm min}\, (|2^{\alpha  j}t|^{-N},1) \lesssim |t|^{-d/2}.
$$
Finally, for case $\mathbf { \widetilde A}$,  we have
$|\nabla_\xi  (2^{j}x\cdot \xi+ 2^{\alpha j} t  \widetilde\phi_j(\xi))| \gtrsim 2^{ j} |x|.$
Integration by parts yields 
$ |\widetilde{I_j}| \lesssim  2^{\frac{d\alpha}{2}j} {\rm min}\, ((2^{  j}|x|)^{-N},1)$ for $N\le d+2$. In particular, taking $N=d/2$,  
we see  
$$
\sum_{j\in \mathbf { \widetilde A}} |\widetilde{I_j}| \le C |x|^{-d/2} 
\sum_{2^j\gtrsim (|t|/|x|)^{1/(1-\alpha)}}2^{\frac{d}{2}(\alpha-1)j}  \lesssim  |t|^{-d/2} .
$$
Combining these estimates for the cases $\mathbf{\widetilde A}$, $ \mathbf{\widetilde B}$, and $\mathbf{\widetilde C}$ gives \eqref{near-away} with $-$.  This completes the proof. 
\end{proof}

We conclude this section with some remarks on our results by comparing them with previously known results.
\begin{remark}
First, the estimate for the wave case (Proposition \ref{p:osc-wave})  is of a different nature from the standard  context of damped oscillatory integral estimates since the determinant of the Hessian matrix is identically zero. As far as the authors  are aware, no such result has previously appeared in the literature. The same also applies to 
Proposition \ref{p:KGoscillatory_wave}.  Secondly, 
 concerning Proposition \ref{p:KGoscillatory_Schro}, note that  if  $|\xi|$ is large, then $|\nabla \langle  \xi\rangle |\sim 1$  and 
\[  \det H\langle  \xi\rangle = \langle  \xi\rangle^{-3d}|\xi|^{2d-2}\sim |\xi|^{-d-2}, \] 
which means $\jb \xi$ is not almost homogeneous. Thus, the estimate \eqref{e:KGdispersivemain} can not be covered by Proposition \ref{p:OsciAlmostRad} and we extend the result in Kenig--Ponce--Vega \cite{KenigPonceVega}.   Also,  \eqref{e:KGdispersivemain} can not be deduced from Carbery and Ziesler's work \cite{CarberyZiesler} since the result in \cite{CarberyZiesler}  is local in its nature and  the basic convexity assumption  is not satisfied (precisely,  $\phi(r)=\langle r \rangle$ is convex but $\phi'$ is not convex). 
 \end{remark}
 
 \section{Strichartz estimates for orthonormal families: the sharp admissible case} \label{section:ONSsharp}

In this section, making use of  the weighted oscillatory integral estimates in the previous section,  
we prove the sufficiency parts of Theorems \ref{t:wave}, \ref{t:KGSchro}, \ref{t:KGwave},  and \ref{t:fracSchrosharp} (in this order). 
The necessity part is to be discussed later in Section \ref{section:necesssary}.  

We begin by stating the following result which formalizes the argument we use to establish the orthonormal Strichartz estimates.
 
\begin{proposition} \label{p:abstractdual}
Let $\sigma > 0$ and assume that $(q,r)$ is sharp $\sigma$-admissible such that\,\footnote{Equivalently, $\frac{2(\sigma + 1)}{\sigma} < r < \infty$ if $0 < \sigma \leq \frac{1}{2}$, or $\frac{2(\sigma + 1)}{\sigma} < r < \frac{2(2\sigma+1)}{2\sigma-1}$ if $\sigma > \frac{1}{2}$.  }
\begin{equation} \label{e:rtilderange}
\max\{1 + 2\sigma,2\} < \tr < 2 + 2\sigma.
\end{equation}
Let $S$ be a domain in the complex plane which contains the strip $\{z\in \mathbb C: -\widetilde r/2\le  \Re z \le 0 \}$. 
Suppose that $(\Theta_z)_{z}$ is  an analytic family of  functions for $z\in S$\,\footnote{This means the map $z\mapsto \Theta_z(\xi)$ is an analytic function on $S$ for each $\xi\in\mathbb{R}^d$. }  
which satisfies the following: 
For some constants $A_0, A_1$ and $N > 0$,
\begin{align} \label{e:abstract0}
\sup_{\xi \in \mathbb{R}^d}|&\Theta_{i\kappa} (\xi)| \leq (1 + |\kappa|)^N A_0, 
\\
\label{e:abstractosc}
\bigg| \int e^{i(x \cdot \xi + t\phi(\xi))} &\Theta_{-\frac{\tr}2 + i\kappa}(\xi) \, \mathrm{d}\xi  \bigg| \leq (1 + |\kappa|)^N A_1 |t|^{-\sigma}.
\end{align}
Additionally, if $\Theta_{-1}$ is nonnegative, then
\begin{equation*}
\bigg\|   \sum_j  \nu_j |U_\phi \sqrt{\Theta_{-1}(D)} f_j| ^2 \bigg\|_{L^{\frac q2,\beta}_tL^{\frac r2}_x} \lesssim A_0^{\frac{2}{r}} A_1^{1 - \frac{2}{r}} \| \nu \|_{\ell^\beta}
\end{equation*}
holds for all families of orthonormal functions $(f_j)_j$ in $L^2$ and $\beta = \frac{2r}{r+2}$.
\end{proposition}

The assumption that  $\Theta_{-1}$ is nonnegative is not essential. As long as $\sqrt{\Theta_{-1}(D)}$ can be properly defined, Proposition \ref{p:abstractdual} is valid.

\begin{proof}
Consider the analytic family  operators $(T_z)_z$ which are given by 
\[ 
\mathcal F(T_z G)(\xi,\tau) =  \frac{ (\tau-\phi(\xi))_+^z}{\Gamma(1+z)} A_0^{- 1 - \frac{2}{\tr}z}A_1^{\frac{2}{\tr}z} \Theta_z(\xi) \widehat G(\xi,\tau).
\]
Note that
\[ 
T_{-1}= \frac{1}{2\pi}A_0^{-1 + \frac{2}{\tr}}A_1^{-\frac{2}{\tr}} U_\phi \Theta_{-1}(D) U_\phi^*.
\] 
Therefore, by Proposition \ref{p:duality} it suffices to prove
\begin{equation*}
\|W_1T_{-1} W_2 \|_{\mathcal{C}^{\tr}} \lesssim \|W_1\|_{(\tq,2\tr),\tr}\|W_2\|_{(\tq,2\tr),\tr}\,.
\end{equation*}
Since $(q,r)$ is sharp $\sigma$-admissible,  we note that $\frac{\widetilde r}{2\widetilde q}=\frac{\widetilde r-2\sigma}{4}$. 
By analytic interpolation, this follows from  the estimates
\begin{equation} \label{e:abstractinfty}
\| W_1 T_{i\kappa} W_2\|_{\mathcal{C}^\infty} \leq C_\kappa \|W_1\|_{\infty, \infty}\|W_2\|_{\infty, \infty}
\end{equation}
and
\begin{equation} \label{e:abstract2}
\|W_1 T_{-\frac{\tr}2 + i\kappa} W_2\|_{\mathcal{C}^2} \leq C_\kappa \| W_1\|_{(\frac{4}{\tr-2\sigma},  4),  2}  \| W_2\|_{(\frac{4}{\tr-2\sigma},4), 2}
\end{equation}
for constants $C_\kappa$ which grow at most exponentially with $\kappa$. Since \eqref{e:abstractinfty} is equivalent to the $L^2$ boundedness of $T_{i\kappa}$, 
the estimate in \eqref{e:abstractinfty} is an easy consequence of Plancherel's theorem and the assumption \eqref{e:abstract0}. 

For \eqref{e:abstract2}, let $z = -\frac{\tr}2 + i\kappa$, in which case the kernel of $T_z$ is given by 
\[ 
\frac{C}{\Gamma(1- \frac{\tr}{2} + i\kappa)} \int  \tau_+^z e^{i(t-s)\tau} \left( \int  e^{i((x-y) \cdot \xi+(t-s)\phi(\xi))} \Theta_{-\frac{\tr}2 + i\kappa}(\xi) \,\mathrm{d}\xi \right) \mathrm{d}\tau.
\] 
Thus, by the assumption \eqref{e:abstractosc} we see\footnote{Here we are using the fact that $\mathcal F(\frac{t^z_+}{\Gamma(z+1)})(\tau) =   i e^{iz\pi/2} (\tau+i0)^{-z-1}=  i (    e^{iz\pi/2} \tau_+^{-z-1}-e^{-iz\pi/2} \tau_-^{-z-1})$, where the latter equality is true if $z\notin \mathbb Z$ (see \cite[p. 172]{GS}).}  that the kernel $K_{-\frac{\tr}2 + i\kappa}$ of the operator $W_1 T_{-\frac{\tr}2 + i\kappa} W_2$ satisfies the bound
\[ 
| K_{-\frac{\tr}2 + i\kappa}(x,t, y,s) | \leq C_\kappa |W_1(x,t)||t-s|^{-\sigma -1+\frac{\tr}2} |W_2(y,s)|
\]
for some constant $C_\kappa$ which grows at most exponentially with $\kappa$.
Using \eqref{e:HLS} and the assumption \eqref{e:rtilderange}, we obtain
\begin{align*} 
\|W_1 T_{-\frac{\tr}2 + i\kappa} W_2\|_{\mathcal{C}^2}^2 & \leq C_\kappa \iiiint   |W_1(x,t)|^2|t-s|^{-2\sigma-2+{\tr}} |W_2(y,s)|^2 \, \mathrm{d}x\mathrm{d}y\mathrm{d}s\mathrm{d}t 
\\
&  \leq C_\kappa \| W_1^2\|_{(u,2), 1}  \| W_2^2\|_{(u,2), 1}, 
\end{align*}
where $\frac2u =\tr - 2\sigma$. The estimate in \eqref{e:abstract2} follows.
\end{proof}

\subsection{The wave equation}
Since we now have Proposition \ref{p:abstractdual}, to show sufficiency part of Theorem \ref{t:wave}  we need to choose an appropriate analytic family $\Theta_z(\xi)$. However 
it is already more or less  clear from the perspective  of  Proposition  \ref{p:osc-wave}. 

\begin{proof}[Proof of Theorem \ref{t:wave} (sufficiency part)]
It suffices to prove
\begin{equation} \label{e:wavefinalgoal}
\bigg\|   \sum_j  \nu_j |e^{it\sqrt{-\Delta}} f_j| ^2 \bigg\|_{(\frac q2, \beta), \frac r2 }\lesssim \| \nu\|_{\ell^\beta}
\end{equation}
for $(q,r)$ which are sharp $\frac{d-1}{2}$-admissible, $\beta = \frac{2r}{r+2}$,  and
\begin{equation} \label{e:waverestricted}
\frac{2(d+1)}{d-1} < r<\frac{2d}{d-2},  
\end{equation}
 where $(f_j)_j$ is an orthonormal family in $\dot{H}^s$ and $s = \frac{d+1}{2}(\frac{1}{2} - \frac{1}{r})$. Indeed, for such $(q,r)$ we have $\frac{q}{2} > \beta$ and thus   
 the strong type estimate \eqref{e:ONSwaveFS} follows by embedding between Lorentz spaces. Since this estimate is trivially valid when $(q,r,\beta)=(\infty,2,1)$, by interpolation we may extend the range to the full range 
$
2\leq r<\frac{2d}{d-2}.
$

Let us set 
\[
g_j= |D|^{\frac{d+1}2(\frac{1}{2} -\frac{1}{r})} f_j. 
\]
Clearly, $(g_j)_j$ is an orthonormal family in $L^2$. 
To prove \eqref{e:wavefinalgoal} under the condition \eqref{e:waverestricted}, we use Proposition \ref{p:abstractdual} with
\[
\Theta_z(\xi) = \frac{{\tr}+2z}{\tr-2}  |\xi|^{\frac{d+1}{\tr}z} \chi^2(|\xi|).
\]
Then,  \eqref{e:abstract0} holds and, from Proposition \ref{p:osc-wave}, we have  \eqref{e:abstractosc} with $\sigma = \frac{d-1}{2}$. 
Also, note  that $d < \tr < d+1$ from \eqref{e:waverestricted}. 
Since 
$
\Theta_{-1}(\xi) = |\xi|^{- (d+1)(\frac{1}{2} - \frac{1}{r})}\chi^2(|\xi|),
$
 it follows from Proposition \ref{p:abstractdual} that we have the estimate
\[
\bigg\|   \sum_j  \nu_j \Big|e^{it\sqrt{-\Delta}} \chi(|D|)  |D|^{-s} g_j\Big| ^2 \bigg\|_{(\frac q2, \beta), \frac r2 }\lesssim \| \nu \|_{\ell^\beta}
\]
with $s=\frac{d+1}{2}(\frac{1}{2}-\frac{1}{r})$ and $\beta = \frac{2r}{r+2}$.  This immediately yields 
\[
\bigg\|   \sum_j  \nu_j |e^{it\sqrt{-\Delta}} \chi(|D|) f_j| ^2 \bigg\|_{(\frac q2, \beta), \frac r2 }\lesssim \| \nu \|_{\ell^\beta}
\]
for all families of orthonormal functions $(f_j)_j$ in $\dot{H}^s$, with $s=\frac{d+1}{2}(\frac{1}{2}-\frac{1}{r})$ and $\beta = \frac{2r}{r+2}$. Note that the  sharp $\frac{d-1}2$-admissible estimate \eqref{e:wavefinalgoal} with $s=\frac{d+1}{2}(\frac{1}{2}-\frac{1}{r})$  is invariant under rescaling.    Thus, it is easy to see that  rescaling the above  estimate gives 
\[
\bigg\|   \sum_j  \nu_j |e^{it\sqrt{-\Delta}} \chi(\varepsilon |D|)^{1/2}f_j| ^2 \bigg\|_{(\frac q2, \beta), \frac r2 }\lesssim \|  \nu \|_{\ell^\beta}
\]
uniformly in $\varepsilon > 0$, and letting $\varepsilon$ tend to zero, we obtain \eqref{e:wavefinalgoal}.
\end{proof}

\subsection{The Klein--Gordon equation}   
As mentioned in the introduction, we separately handle the Schr\"odinger-like  case   and the wave-like case. As before, the proof is rather straightforward once we have the right analytic family, so we shall be brief.

\begin{proof}[Proof of Theorem \ref{t:KGwave} (sufficiency part)] We only hand the critical case $s = \frac{d+1}{2}(\frac{1}{2} - \frac{1}{r})$ since the other case $s > \frac{d+1}{2}(\frac{1}{2} - \frac{1}{r})$ can be shown in exactly same manner by making use of the estimate \eqref{e:KGwave-more-damping}
   instead of \eqref{KGwavedecay}. 
   
In a similar manner to the proof of Theorem \ref{t:wave}, it suffices to prove the estimate \begin{equation}
\label{KGest-epsilon}
\bigg\|   \sum_j  \nu_j |e^{it\sqrt{1-\Delta}} \chi(\varepsilon |D|)f_j| ^2 \bigg\|_{(\frac q2, \beta), \frac r2 }\lesssim \|  \nu \|_{\ell^\beta}
\end{equation}
uniformly in $\varepsilon > 0$. Here, $(q,r)$ is sharp $\frac{d-1}{2}$-admissible with $r$ satisfying \eqref{e:waverestricted}, 
$\beta = \frac{2r}{r+2}$, and an orthonormal family $(f_j)_j$  in $H^s$, $s = \frac{d+1}{2}(\frac{1}{2} - \frac{1}{r})$. 
To this end, we use Proposition \ref{p:abstractdual} with
\[
\Theta_z(\xi) = \frac{{\tr}+2z}{\tr-2}  \langle\xi\rangle^{\frac{d+1}{\tr}z} \chi^2(\varepsilon |\xi|), 
\]
and  we see  \eqref{e:abstract0} holds with $A_0$ independent of $\varepsilon$ and, from the estimate \eqref{KGwavedecay} in Proposition \ref{p:KGoscillatory_wave}, we obtain \eqref{e:abstractosc} with $\sigma = \frac{d-1}{2}$ and $A_1$ independent of $\varepsilon$. Since
$
\Theta_{-1}(\xi) = \langle\xi\rangle^{-2s}\chi^2(\varepsilon |\xi|)
$
and   $( \langle  D \rangle^{-s} f_j )_j $ forms an orthonormal family in $L^2$, 
 the desired uniform estimate \eqref{KGest-epsilon} follows from Proposition \ref{p:abstractdual}.
\end{proof}

\begin{proof}[Proof of Theorem \ref{t:KGSchro} (sufficiency part)] As before we only hand the critical case $s = \frac{d+2}{2}(\frac{1}{2} - \frac{1}{r})$  since the other case $s > \frac{d+2}{2}(\frac{1}{2} - \frac{1}{r})$ can be shown by following the argument below  if we use the estimate \eqref{e:KGwave-more-damping}   instead of \eqref{KGwavedecay}.

It is enough to show that  \eqref{KGest-epsilon} holds uniformly in $\varepsilon > 0$ for  sharp $\frac{d}{2}$-admissible  $(q,r)$ with $2\le  r<\frac{2(d+1)}{d-1}, $  $\beta = \frac{2r}{r+2}$, and an orthonormal family $(f_j)_j$  in $H^s$, $s = \frac{d+2}{2}(\frac{1}{2} - \frac{1}{r})$. 
We use Proposition \ref{p:abstractdual} with
\[
\Theta_z(\xi) =  \langle\xi\rangle^{\frac{d+2}{\tr}z} \chi^2(\varepsilon |\xi|).
\]
Then the assumption \eqref{e:abstract0} trivially holds  and, from Proposition \ref{p:KGoscillatory_Schro}, we have \eqref{e:abstractosc} with $\sigma = \frac{d}{2}$ which is uniform in $\varepsilon$.  Since
$
\Theta_{-1}(\xi) = \langle\xi\rangle^{-2s} \chi^2(\varepsilon |\xi|)
$
and $( \langle  D \rangle^{-s} f_j )_j $ is  an orthonormal family in $L^2$, 
we deduce from Proposition \ref{p:abstractdual} that
\eqref{KGest-epsilon}
holds for all families of orthonormal functions $(f_j)_j$ in $H^s$, $\beta = \frac{2r}{r+2}$ and where $(q,r)$ is sharp $\frac{d}{2}$-admissible with
\begin{equation} \label{e:schrorestricted}
\frac{2(d+2)}{d} < r<\frac{2(d+1)}{d-1}.
\end{equation}
Since this estimate holds trivially when $(q,r,\beta)=(\infty,2,1)$, by interpolation we may extend the range to 
$
2\leq r<\frac{2(d+1)}{d-1},
$
and this completes the proof.
\end{proof}

\begin{remark} \label{trivial-extension} In the proofs of  Theorem \ref{t:KGwave} and  Theorem \ref{t:KGSchro} the  noncritical cases $s > \frac{d+1}{2}(\frac{1}{2} - \frac{1}{r})$ and  $s > \frac{d+2}{2}(\frac{1}{2} - \frac{1}{r})$  can also be deduced from the critical cases by making use of the  inequality 
\[ \Big\|  \Big( \sum_j \big|\langle D \rangle^{-\beta} g_j\big | ^2\Big)^\frac12 \Big \|_{p}\lesssim \Big\|  \Big( \sum_j \big| g_j\big | ^2\Big)^\frac12 \Big \|_{p}\] 
which is valid for $\beta\ge 0$ and $1\le p< \infty$ with any  $(g_j)_j$ not necessarily orthogonal. This inequality may  be shown from the trivial inequality 
$\|\langle D \rangle^{-\beta} f\|_p\lesssim \|f\|_p$, $1\le p\le \infty$, and randomization (Kintchin's inequality).    
\end{remark}

\subsection{The fractional Schr\"{o}dinger equation}
Clearly, our phase function $\phi_\alpha(\xi) = |\xi|^\alpha$, $\alpha\in\mathbb{R}\setminus\{0,1\}$ is almost homogeneous of order $\alpha$ and  $|\!\det H \phi_\alpha(\xi)|=C_{d,\alpha} |\xi|^{d(\alpha-2)}$ for some constant $C_{d,\alpha}$.  Hence, Theorem \ref{t:fracSchrosharp} may be deduced from the following theorem. 
\begin{theorem} \label{t:radial}
Let $d\geq 1$, $\alpha\in \mathbb R\setminus\{0,1\}$ and $\phi$ be almost homogeneous of order $\alpha$. 
Suppose $(q,r)$ is sharp $\frac{d}{2}$-admissible. Then, the following hold for all families of orthonormal functions $(f_j)_j$ in $L^2$.
\vspace{-7pt}
\begin{enumerate}
[leftmargin=0.8cm, labelsep= 0.4 cm, topsep=0pt]
\item[$(i)\,$] If $2\leq r<\frac{2(d+1)}{d-1}$ and $\beta = \frac{2r}{r+2}$, 
then we have the estimate
\begin{equation}
\label{h-ortho}
\bigg\|   \sum_j  \nu_j \Big|U_\phi (|\det H\phi(D)|^{\frac{r-2}{4r}}  f_j)\Big| ^2 \bigg\|_{L^{\frac q2}_tL^{\frac r2}_x} \lesssim \| \nu \|_{\ell^\beta} .
\end{equation}
\item[$(ii)$] If $d = 2$ and $6 \leq r < \infty,$
or if
$d \geq 3$ and $\frac{2(d+1)}{d-1} \leq r \le \frac{2d}{d-2},$ then \eqref{h-ortho}
holds with $\beta < \frac{q}{2}$.  
\end{enumerate}
\end{theorem}
\begin{proof}[Proof of Theorem \ref{t:radial}]
It is enough to prove $(i)$. Indeed, the standard argument combined with the Littlewood--Paley inequality \footnote{After a  Littlewood--Paley decomposition, one may use a  rescaling argument with Lemma \ref{l:phaseScale} to get uniform bounds for each dyadic piece.}  gives the estimate 
\[\Big\|U_\phi (|\det H\phi(D)|^{\frac{r-2}{4r}}  f)\Big\|_{L^{ q}_tL^{ r}_x} \lesssim \|f\|_2,\] 
for all $\frac d2$-admissible $(q,r)$. For example, see \cite[p. 978]{KeelTao}. 
The estimate is trivially equivalent to \eqref{h-ortho} with $\beta=1$.  In particular, for $d \geq 3$, interpolation between the estimate in $(i)$ and \eqref{h-ortho} with $(q,r,\beta) = (2,\frac{2d}{d-2},1)$ proves $(ii)$. When $d=2$, a similar argument works except that we must interpolate between $(i)$ and \eqref{h-ortho} with $\beta = 1$ and $(q,r)$ sharp $1$-admissible with $(\frac{1}{r},\frac{1}{q})$ arbitrarily close to $(0,\frac{1}{2})$.

In order to show $(i)$, first let us consider $\alpha > 0$, in which case it suffices to prove the estimate
\begin{equation} \label{e:radialfinalgoal}
\bigg\| \sum_j  \nu_j \Big|U_\phi (|\det H\phi(D)|^{\frac{r-2}{4r}} \chi(\varepsilon |D|) f_j)\Big| ^2 \bigg\|_{L^{\frac q2}_tL^{\frac r2}_x} \lesssim \| \nu\|_{\ell^\beta} 
\end{equation}
uniformly in $\varepsilon > 0$. Here, $(f_j)_j$ is an orthonormal family in $L^2$, $\beta = \frac{2r}{r+2}$, and $(q,r)$ is sharp $\frac{d}{2}$-admissible satisfying 
 \eqref{e:schrorestricted}. 
Indeed, once this is established, we take the limit $\varepsilon \to 0$ and then interpolate the resulting bound with the case $(q,r,\beta) = (\infty,2,1)$ to obtain the desired estimates for the range  $2\le r \le \frac{2d}{d-2}$.

To prove \eqref{e:radialfinalgoal}, we consider 
\[
\Theta_z(\xi) = |\det H\phi(\xi)|^{-z/\tr} \chi^2(\varepsilon|\xi|).
\]
Obviously \eqref{e:abstract0} holds and we use Proposition \ref{p:OsciAlmostRad} to verify \eqref{e:abstractosc} with $\sigma = \frac{d}{2}$. Since $
\Theta_{-1}(\xi) = |\det H\phi(\xi)|^{\frac{r-2}{2r}} \chi^2(\varepsilon|\xi|)$, 
we obtain \eqref{e:radialfinalgoal} from Proposition \ref{p:abstractdual}.

The case $\alpha < 0$ can be proved in a very similar manner. It suffices to prove the uniform bound \eqref{e:radialfinalgoal} with $\chi$ replaced by $\chi_\infty$, upon which we take the limit $\varepsilon \to \infty$. To this end, we apply Proposition \ref{p:abstractdual} to 
\[
\Theta_z(\xi) = |\det H\phi(\xi)|^{-z/\tr} \chi_\infty^2(\varepsilon |\xi|)
\]
and again use Proposition \ref{p:OsciAlmostRad} to verify \eqref{e:abstractosc}.  The remainder is identical to the previous case, so we omit the details. 
\end{proof}

\begin{remark}
\label{lorentz-improvement}
Since the proofs of Theorems \ref{t:fracSchrosharp}--\ref{t:KGwave} all rely on Proposition \ref{p:abstractdual}, it is clear that the estimates in part $(i)$ of Theorems \ref{t:wave} and \ref{t:KGwave} are true with the Lorentz norm $\| \cdot \|_{(\frac{q}{2}, \frac{2r}{r+2} ),\frac{r}{2}}$ on the left-hand side for sharp $\frac{d-1}2$--admissible $(q,r)$ satisfying \eqref{e:waverestricted}, and  similarly the same is also true for  the estimates  in part $(i)$ of Theorems \ref{t:fracSchrosharp}  and  \ref{t:KGSchro} if $(q,r)$ is $\frac{d}2$--admissible  and  satisfies \eqref{e:schrorestricted}.
\end{remark}

\section{Strichartz estimates for orthonormal families: the non-sharp admissible case} \label{section:nonsharp}

In this section, we present extensions of Theorems \ref{t:fracSchrosharp}--\ref{t:KGwave} to the corresponding non-sharp admissible cases.  In this case,  emphasis lies on  proving the estimates with initial data of the sharp regularity. If the dispersion relation is homogeneous, the optimal regularity is naturally determined by the homogeneous degree of the dispersion relation. In case of the classical Strichartz estimates the sharp regularity estimates are well known for the non-sharp admissible  case. However, in  contrast to the classical Strichartz estimates, generalizing to the estimates for orthonormal families with the optimal summability exponent $\beta$ is no longer trivial.  For a discussion on why  more elementary arguments, such as those based on Littlewood--Paley type considerations do not seem to yield the desired estimates, we refer the reader to \cite{BHLNS}. However, for the Schr\"odinger equation this issue was  resolved in \cite{BHLNS} with an argument which made use of improved estimates in the scale of Lorentz spaces (see Remark \ref{lorentz-improvement}). The basic strategy devised in \cite{BHLNS} also works for the propagators under consideration here to recover the non-sharp admissible bounds. However, here we provide a somewhat more straightforward proof based on Lieb's version of the Sobolev inequality for families of orthonormal functions \cite{Lieb_Sobolev}.

In what follows the estimates we have already obtained in the sharp case play an important role in establishing the non-sharp case and, to a certain degree, this ``deduction" can be captured in an abstract framework. Thus, prior to the statements for each particular equation, we present some results which hold in a level of generality. 

In order to facilitate our presentation we introduce some notations. 
For $\sigma \geq \frac{1}{2}$, we introduce the points $A_\sigma$ and $B_\sigma$ in $[0,\frac{1}{2}] \times [0,\frac{1}{2}]$ given by
\begin{align*}
A_\sigma = \bigg(\frac{2\sigma - 1}{2(2\sigma + 1)}, \frac{\sigma}{2\sigma + 1}\bigg), \quad 
B_\sigma  = \bigg(\frac{\sigma}{2(\sigma+1)},\frac{\sigma}{2(\sigma + 1)}\bigg).
\end{align*}
For $q,r\ge 2$, define $\beta_\sigma(q,r) \in [1,\infty]$  by 
\[
\frac{\sigma}{\beta_\sigma(q,r)} = \frac{1}{q} + \frac{2\sigma}{r}. 
\]
We also introduce the points
\begin{align*}
C = \bigg(\frac{1}{2},0\bigg), \qquad D = \bigg(0,\frac{1}{2}\bigg)
\end{align*}
and, for $\sigma \geq 1$, the point $E_\sigma$ by
\begin{align*}
E_\sigma = \bigg(\frac{\sigma-1}{2\sigma},\frac{1}{2}\bigg).
\end{align*}
Additionally,  if  $P_1, P_2, \dots, P_n$ are points in $[0,\frac{1}{2}]\times [0,\frac{1}{2}]$, by $\mathrm{int} (P_1P_2 \cdots P_n)$ we denote the interior of the  convex hull of the set $\{P_1, P_2, \dots, P_n\}$.

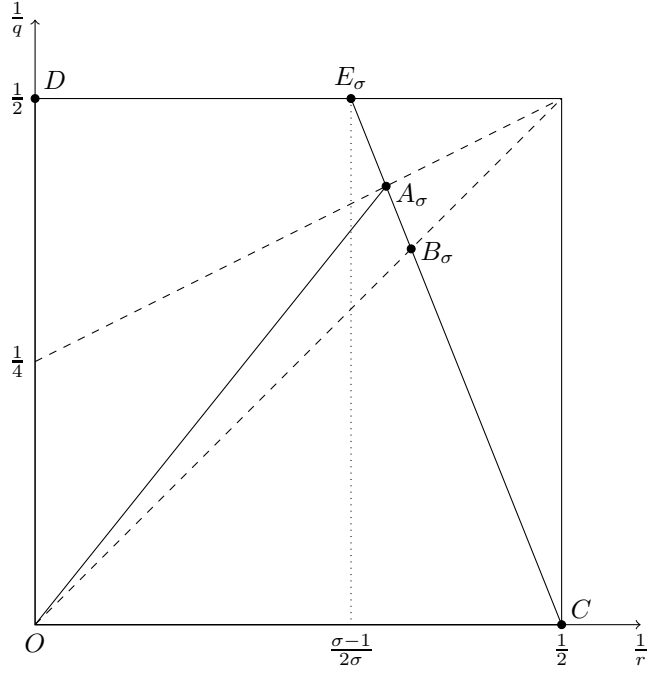
\begin{figure} 
\label{fig:OtoG}
\begin{center}
\begin{tikzpicture}[scale=3.5]
\draw [<->] (0,2.3) node (yaxis) [left] {$\tfrac{1}{q}$}
|- (2.3,0) node (xaxis) [below] {$\tfrac{1}{r}$};
\draw (0, 0) rectangle (2, 2);
\draw (0,0)--(4/3,5/3)--(2,0)--(0,0);

\node [above] at (1.2,2) {$E_\sigma$};
\node [above] at (0.075,2) {$D$}; 
\node [right]at (4/3,49/30) {$A_\sigma$};
\node [right] at (2, 0.06) {$C$};
\node [right] at (10/7,99/70) {$B_\sigma$};
\node [below] at (1.2,0) {$\frac{\sigma-1}{2\sigma}$};

\draw [dashed] (0,1)--(2,2);
\draw [dashed] (0,0)--(2,2);
\draw [dotted] (1.2,2)--(1.2,0);
\node [left]at (0,2) {$\frac{1}{2}$};
\node [left]at (0,1) {$\frac14$};
\node[below] at (0,0) {$O$};
\node[below] at (2,0) {$\frac{1}{2}$};
\draw (4/3,5/3)--(1.2,2);

\filldraw[fill=black] (4/3,5/3) circle[radius=0.15mm];
\filldraw[fill=black] (10/7,10/7)  circle[radius=0.15mm];
\filldraw[fill=black] (0,2)  circle[radius=0.15mm];
\filldraw[fill=black] (2,0)  circle[radius=0.15mm];
\filldraw[fill=black] (1.2,2)  circle[radius=0.15mm];
\end{tikzpicture}
\end{center}
\caption{The points $A_\sigma, B_\sigma, C, D, E_\sigma$ in the case $\sigma > 1$.}
\end{figure}

\begin{proposition} \label{p:nonsharp}
Let $\sigma \geq \frac{1}{2}$, $s(q,r) = \frac{a}{q} + d(\frac{1}{2} - \frac{1}{r})$ for some $a \in \mathbb{R}$, and $\Psi(\xi)$ be  $|\xi|$ or $\langle \xi \rangle$. Assume that 
for  sharp $\sigma$-admissible $(q,r)$  satisfying  \eqref{e:rtilderange} and,  for all families of orthonormal functions $(f_j)_j$ in $L^2$,
 the estimate
\begin{equation} \label{e:ONSrefinedsharp}
\bigg\| \sum_j \nu_j |U_\phi \Psi(D)^{-s(q,r) } f_j|^2 \bigg\|_{\frac{q}{2},\frac{r}{2}} \lesssim \|\nu\|_{\ell^{\beta}}
\end{equation}
holds with $\beta = \beta_\sigma(q,r)$. 
Then, for all $(q,r)$ which are non-sharp $\sigma$-admissible with
$
q > \frac{2\sigma - 1}{2\sigma}r,
$
the estimate
\begin{equation} \label{e:ONSintgoal}
\bigg\| \sum_j \nu_j |U_\phi \Psi(D)^{-s(q,r) } f_j|^2 \bigg\|_{\frac{q}{2},\frac{r}{2}} \lesssim \|\nu\|_{\ell^{\beta}}
\end{equation}
holds for all families of orthonormal functions $(f_j)_j$ in $L^2$ with $\beta = \beta_\sigma(q,r)$.
\end{proposition}

\begin{remark} \label{remark:afterP5.1}
(1) Suppose $\sigma > 1$ and, in addition to the assumptions in Proposition \ref{p:nonsharp}, we assume that the (single-function) Strichartz estimate
\[
\| U_\phi \Psi(D)^{-s(2,r)} f \|_{2,r} \lesssim \|f\|_2
\]
holds for $\frac{\sigma - 1}{\sigma} < r < \infty$; that is, $(\frac{1}{r},\frac{1}{2})$ lies on the line segment $(D,E_\sigma)$. Then, for such $r$, clearly \eqref{e:ONSintgoal} holds with $(q,\beta) = (2,1)$. We also note that $\beta_\sigma(q,r) = \frac{q}{2}$ whenever $(\frac{1}{r},\frac{1}{q})$ belongs to the line segment $(O,A_\sigma)$. Thus, by complex interpolation between estimates \eqref{e:ONSintgoal} with $(\frac1r, \frac1q)$ on $(D,E_\sigma)$  and points in the region $\mathrm{int} (OA_\sigma C)$ arbitrarily close to the line segment $(O,A_\sigma)$, we deduce that \eqref{e:ONSintgoal} holds if $(\frac1r, \frac1q)$ belongs to $\mathrm{int} (ODE_\sigma A_\sigma)$ and $\beta < \frac{q}{2}$.

(2) If $(q,r)$ is sharp $\sigma$-admissible then
$
\beta_\sigma(q,r) = \frac{2r}{r+2}.
$
As we have seen, $\beta = \frac{2r}{r+2}$ is the sharp summability exponent appearing on the right-hand side of the estimates in Theorems \ref{t:fracSchrosharp} and \ref{t:KGSchro} for $(\frac{1}{r},\frac{1}{q})$ lying on the line segment $(A_\frac{d}{2}, C]$. 
\end{remark}

\begin{proof}[Proof of Proposition \ref{p:nonsharp}]
The key estimate is the following version of the Sobolev inequality for orthonormal functions, due to Lieb \cite{Lieb_Sobolev}. For all $1<p<\infty$, 
\begin{equation}\label{e:ONSobolev}
\bigg\| \sum_j \nu_j |\Psi(D)^{- \frac{d}{2p'}}f_j|^2 \bigg\|_{L^p(\mathbb{R}^d)} \lesssim \|\nu\|_{\ell^{p,1}} 
\end{equation}
holds for all families of orthonormal functions $(f_j)_j$ in $L^2(\mathbb{R}^d)$ and sequences $\nu$ in $\ell^{p,1}$. 
Notice that for each fixed $t\in\mathbb{R}$, $(U_\phi f_j (t,\cdot))_j$ remains to be an orthonormal family in $L^2(\mathbb{R}^d)$. Also, note that $s(\infty,r) = \frac{d}{2(r/2)'}$. So, \eqref{e:ONSobolev} implies 
\begin{equation}\label{e:0822-1}
\bigg\| \sum_j \nu_j |U_\phi \Psi(D)^{- s(\infty,r)}f_j|^2 \bigg\|_{\infty, \frac r2} \lesssim \|\nu\|_{\ell^{\frac r2,1}}.  
\end{equation}
In view of $\beta_\sigma(\infty, r) = \frac r2$, complex interpolation between \eqref{e:0822-1} and \eqref{e:ONSrefinedsharp} implies 
\begin{equation} \label{e:inftygone}
\bigg\| \sum_j \nu_j |U_\phi \Psi(D)^{-s(q,r) } f_j|^2 \bigg\|_{\frac{q}{2},\frac{r}{2}} \lesssim \|\nu\|_{\ell^{\beta,1}}
\end{equation}
for $(\frac{1}{r},\frac{1}{q})$ belonging to $\mathrm{int} (OA_\sigma C)$, with $\beta = \beta_\sigma(q,r)$; we shall, in fact, only need to make use of this estimate for   $(\frac{1}{r},\frac{1}{q})$ in $\mathrm{int} (OA_\sigma B_\sigma)$.

The estimate \eqref{e:inftygone} already provides the sharp estimates for various equations if we accept 
the weaker Lorentz space norm on the right-hand side; however, we need to strengthen \eqref{e:inftygone} to the desired strong-type estimate \eqref{e:ONSintgoal} and this can be done by using real interpolation.
 We fix $(\frac{1}{r},\frac{1}{q})$ belonging to $\mathrm{int} (OA_\sigma B_\sigma)$ and choose two distinct points $(q_0,r_0),$ $(q_1,r_0)$ such that $(\frac{1}{r_j},\frac{1}{q_j})$ belongs to $\mathrm{int} (OA_\sigma B_\sigma)$ and $s(q,r) = s(q_j,r_j)$ for $j=0,1$. By real interpolation, this\footnote{Here, we are using the fact that $
(L^{q_0}(L^{r_0}),L^{q_1}(L^{r_1}))_{\theta,q}=L^q(L^{r,q})
$ holds whenever $q_0,q_1,r_0,r_1 \in [1,\infty)$, $\frac{1}{q}=\frac{1-\theta}{q_0}+\frac{\theta}{q_1}$, $\frac{1}{r}=\frac{1-\theta}{r_0}+\frac{\theta}{r_1}$, $\theta\in(0,1)$ (see \cite{Lions-Peetre} and \cite{Cwikel74}).} yields
\begin{equation*} 
\bigg\| \sum_j \nu_j |U_\phi \Psi(D)^{-s(q,r) } f_j|^2 \bigg\|_{\frac{q}{2},(\frac{r}{2},\frac{q}{2})} \lesssim \|\nu\|_{\ell^{\beta,\frac{q}{2}}},
\end{equation*}
where $\beta = \beta_\sigma(q,r)$.  Now, notice that 
\begin{equation} \label{e:lucky}
\frac{q}{2} > \beta_\sigma(q,r)
\end{equation}
holds whenever $(\frac{1}{r},\frac{1}{q})$ belongs to $\mathrm{int} (OA_\sigma B_\sigma)$. Since we have both \eqref{e:lucky} and $q \leq r$ when $(\frac{1}{r},\frac{1}{q})$ belongs to $\mathrm{int} (OA_\sigma B_\sigma)$, we obtain \eqref{e:ONSintgoal} whenever $(\frac{1}{r},\frac{1}{q})$ lies in $\mathrm{int} (OA_\sigma B_\sigma)$, from the embeddings $L^{\frac{q}{2}}_tL^{\frac{r}{2},\frac{q}{2}}_x \subseteq L^{\frac{q}{2}}_tL^{\frac{r}{2}}_x$ and $\ell^{\beta} \subseteq \ell^{\beta,\frac{q}{2}}$. 

Finally, we observe that \eqref{e:ONSintgoal} trivially holds when $(q,r,\beta) = (\infty,2,1)$. Indeed, from the trivial estimate
\[
\| U_\phi f \|_{\infty,2} \lesssim \|f\|_2
\]
and since $s(\infty,2) = 0$, we see that \eqref{e:ONSintgoal} follows by the triangle inequality. By interpolation, we obtain the desired estimate \eqref{e:ONSintgoal} whenever $(\frac{1}{r},\frac{1}{q})$ belongs to $\mathrm{int} (OA_\sigma C)$.
\end{proof}

Once we have Proposition \ref{p:nonsharp}, obtaining the desired estimates for non-sharp admissible $(q,r)$ is rather straightforward. 
We prove sufficiency parts of the theorems stated below and the necessity parts will  be shown in Section \ref{section:necesssary}.

\begin{theorem}
[The wave equation and the non-sharp $\frac{d-1}2$-admissible] 
\label{t:wavenonsharp} Let  $d\geq 2$ and suppose $(q,r)$ is non-sharp $\frac{d-1}{2}$-admissible. 
\vspace{-7pt}
\begin{enumerate}
[leftmargin=0.8cm, labelsep= 0.3cm, topsep=0pt]
\item[$(i)\,$] If $(\frac{1}{r},\frac{1}{q})$ belongs to $\mathrm{int}(OA_{\frac{d-1}{2}}C)$, then \eqref{e:ONSwaveFS}
holds for all families of orthonormal functions $(f_j)_j$ in $\dot{H}^s$, with $s = \frac{d}{2} - \frac{d}{r} - \frac{1}{q}$, and $\beta = \beta_{\frac{d-1}{2}}(q,r)$. 
\item[$(ii)$] If $d \geq 4$ and $(\frac{1}{r},\frac{1}{q})$ belongs to $\mathrm{int}(ODA_{\frac{d-1}{2}}E_{\frac{d-1}{2}})$, then
\eqref{e:ONSwaveFS}
holds for all families of orthonormal functions $(f_j)_j$ in $\dot{H}^s$, with $s = \frac{d}{2} - \frac{d}{r} - \frac{1}{q}$, and $\beta  < \frac{q}{2}$. This is sharp in the sense that the estimate fails if $\beta > \frac{q}{2}$.
\end{enumerate}
\end{theorem}
\begin{proof}[Proof of Theorem \ref{t:wavenonsharp} (sufficiency part)]
From the classical Strichartz estimates \eqref{e:StrWave} and (1) in Remark \ref{remark:afterP5.1},  
the claimed estimates in $(ii)$ follow from those estimates in $(i)$ by interpolating them with trivial estimates. 
To prove the claimed estimates in $(i)$, we apply Proposition \ref{p:nonsharp} 
with $\Psi(\xi) = |\xi|$ and $s(q,r) = \frac{d}{2} - \frac{d}{r} - \frac{1}{q}$. Thus, it sufficient to verify \eqref{e:ONSrefinedsharp}. This  follows from Theorem \ref{t:wave}. 
\end{proof}

Following a similar argument used to prove Theorem \ref{t:wavenonsharp}, we prove the following theorems concerning the Klein--Gordon equation in the non-sharp $\sigma$-admissible cases, where $\sigma = \frac{d}{2}, \frac{d-1}{2}$. Also, we later discuss further estimates which are available for the Klein--Gordon equation.
\begin{theorem}[The Klein--Gordon equation and the non-sharp $\frac{d-1}{2}$-admissible] 
\label{t:KGnonsharp_wave} 
 Let $d\geq 2$ and suppose $(q,r)$ is non-sharp $\frac{d-1}{2}$-admissible. 
\vspace{-7pt}
\begin{enumerate}
[leftmargin=0.8cm, labelsep= 0.3cm, topsep=0pt]
\item[$(i)\,$] If $(\frac{1}{r},\frac{1}{q})$ belongs to $\mathrm{int}(OA_{\frac{d-1}{2}}C)$, then
\eqref{KG-ortho}
holds for all families of orthonormal functions $(f_j)_j$ in $H^s$, with $s \ge \frac{d}{2} - \frac{d}{r} - \frac{1}{q}$, and $\beta = \beta_{\frac{d-1}{2}}(q,r)$. 
\item[$(ii)$] If $d \geq 4$ and $(\frac{1}{r},\frac{1}{q})$ belongs to $\mathrm{int}(ODA_{\frac{d-1}{2}}E_{\frac{d-1}{2}})$, then
 \eqref{KG-ortho}
holds for all families of orthonormal functions $(f_j)_j$ in $H^s$, with $s \ge \frac{d}{2} - \frac{d}{r} - \frac{1}{q}$, and $\beta  < \frac{q}{2}$. This is sharp in the sense that the estimate fails if $\beta > \frac{q}{2}$.
\end{enumerate}
\end{theorem}

\begin{proof}[Proof of Theorem \ref{t:KGnonsharp_wave} (sufficiency part)]  As we have noticed before, it suffices to prove the estimates in $(i)$ since $(ii)$ follows 
by interpolation between the estimates in $(i)$ and the trivial estimates. We only consider the critical case $s(q,r) = \frac{d}{2} - \frac{d}{r} - \frac{1}{q}$ since the other case can be shown by the same argument with a little modification, or using the inequality in  Remark \ref{trivial-extension}.
Similarly as in  the proof of Theorem \ref{t:wavenonsharp}, we apply Proposition \ref{p:nonsharp} with $\Psi(\xi) = \langle \xi \rangle$ and $s(q,r) = \frac{d}{2} - \frac{d}{r} - \frac{1}{q}$. For the estimate \eqref{e:ONSrefinedsharp}, we employ Theorem \ref{t:KGwave}. 
\end{proof}

\begin{theorem}[The Klein--Gordon equation and the non-sharp $\frac{d}{2}$-admissible] \label{t:KGnonsharp_Schro} Let $d\geq 1$ and suppose $(q,r)$ is non-sharp $\frac{d}{2}$-admissible. 
\vspace{-7pt}
\begin{enumerate}
[leftmargin=0.8cm, labelsep= 0.3cm, topsep=0pt]
\item[$(i)\,$] If $(\frac{1}{r},\frac{1}{q})$ belongs to $\mathrm{int}(OA_{\frac{d}{2}}C)$, then
 \eqref{KG-ortho}
holds for all families of orthonormal functions $(f_j)_j$ in $H^s$, with $s \ge \frac{d}{2} - \frac{d}{r} - \frac{d-2}{dq}$, and $\beta = \beta_{\frac{d}{2}}(q,r)$. This is sharp in the sense that the estimate fails if $\beta > \beta_{\frac{d}{2}}(q,r)$.
\item[$(ii)$] If $d \geq 3$ and $(\frac{1}{r},\frac{1}{q})$ belongs to $\mathrm{int}(ODA_{\frac{d}{2}}E_{\frac{d}{2}})$, then 
 \eqref{KG-ortho}
holds for all families of orthonormal functions $(f_j)_j$ in $H^s$, with $s \ge \frac{d}{2} - \frac{d}{r} - \frac{d-2}{dq}$, and $\beta  < \frac{q}{2}$. This is sharp in the sense that the estimate fails if $\beta > \frac{q}{2}$.
\end{enumerate}
\end{theorem}

\begin{proof}[Proof of Theorem \ref{t:KGnonsharp_Schro} (sufficiency part)] Similarly as in the proof of  Theorem \ref{t:KGnonsharp_wave} (sufficiency part) we only consider the critical case $s(q,r) = \frac{d}{2} - \frac{d}{r} - \frac{d-2}{dq}$. 
It suffices also to prove the estimates in $(i)$ and to this end we apply Proposition \ref{p:nonsharp} with $\Psi(\xi) = \langle \xi \rangle$ and $s(q,r) = \frac{d}{2} - \frac{d}{r} - \frac{d-2}{dq}$. For the estimate \eqref{e:ONSrefinedsharp}, we use Theorem \ref{t:KGSchro}. 
\end{proof}

For simplicity of the exposition, we have presented our main results for the Klein--Gordon equation (Theorems \ref{t:KGSchro}, \ref{t:KGwave}, \ref{t:KGnonsharp_wave} and \ref{t:KGnonsharp_Schro}) in the case where $(q,r)$ is $\sigma$-admissible for $\sigma = \frac{d}{2},  \frac{d-1}{2}$. 
Complex interpolation between these estimates gives the orthonormal Strichartz estimates for the Klein--Gordon equation corresponding
to the sharp $\sigma$-admissible cases with $\sigma\in (\frac{d-1}{2}, \frac{d}{2})$. However, we can obtain these estimates in a unified way, with $\sigma = \frac{d-1}{2} + \rho$ for $\rho \in [0,\frac{1}{2}]$ ($d \geq 2$), and $\rho \in (0,\frac{1}{2}]$ ($d = 1$). For example, since the dispersive estimate with $O(|t|^{-\sigma})$ follows if we interpolate the estimates 
 \eqref{KGwavedecay} (\eqref{e:KGwave-more-damping}) and   \eqref{e:KGdispersivemain} (\eqref{e:KGschro-more-damping}, repectively), then simple modification to the proof of Theorems \ref{t:KGSchro} and \ref{t:KGwave} yields the following.

\begin{theorem}[The Klein--Gordon equation and the sharp $\sigma$-admissible,  $d=1$] \label{t:KGsharpunified_d=1} 
Let $d = 1$, $\sigma \in (0,\frac{1}{2}]$ and suppose $(q,r)$ is sharp $\sigma$-admissible. Then  \eqref{KG-ortho}
holds for all families of orthonormal functions $(f_j)_j$ in $H^s$, with $s \ge (\sigma + 1)(\frac{1}{2}-\frac{1}{r})$, and $\beta = \frac{2r}{r+2}$. 
\end{theorem}

\begin{theorem}[The Klein--Gordon equation and the sharp $\sigma$-admissible, $d \geq 2$] 
\label{t:KGsharpunified} 
Let $d\geq 2$, $\rho \in [0,\frac{1}{2}]$ and suppose $(q,r)$ is sharp $(\frac{d-1}{2} + \rho)$-admissible. 
\vspace{-7pt}
\begin{enumerate}
[leftmargin=0.8cm, labelsep= 0.3cm, topsep=0pt]
\item[$(i)\,$] If $ 2\leq r<\frac{2(d+2\rho)}{d-2+2\rho},$   
  then  \eqref{KG-ortho}
holds for all families of orthonormal functions $(f_j)_j$ in $H^s$, with $s \ge  (\frac{d+1}{2} + \rho)(\frac{1}{2}-\frac{1}{r})$, and $\beta = \frac{2r}{r+2}$. 
\item[$(ii)$] When $d=2$, suppose $\frac{2(1+\rho)}{\rho} \leq r < \infty$ for $\rho \in (0,\frac{1}{2}]$. When $d=3$, suppose $6 \leq r < \infty$ for $\rho = 0$, and 
$\frac{2(d+2\rho)}{d-2+2\rho} \leq r \leq \frac{2(d-1+2\rho)}{d-3+2\rho}$ for $\rho \in (0,\frac{1}{2}]$. When $d \geq 4$, suppose
$\frac{2(d+2\rho)}{d-2+2\rho} \leq r \leq \frac{2(d-1+2\rho)}{d-3+2\rho}$ for $\rho \in [0,\frac{1}{2}]$.
Then,
\eqref{KG-ortho} holds for all families of orthonormal functions $(f_j)_j$ in $H^s$, with $s \ge (\frac{d+1}{2} + \rho)(\frac{1}{2}-\frac{1}{r})$ and $\beta < \frac{q}{2}$. This estimate is sharp in the sense that the estimate fails for $\beta > \frac{q}{2}$.
\end{enumerate}
\end{theorem}
By following the approach taken to prove our main results in the non-sharp admissible case, it is possible to extend Theorems \ref{t:KGsharpunified_d=1} and \ref{t:KGsharpunified} in a similar manner.  In \emph{Remark} \ref{re:klein_gordon} we offer some further remarks regarding the sharpness of the exponent $\beta$ in the case of the Klein--Gordon equation.

\begin{theorem}[The fractional Schr\"odinger equation and the non-sharp $\frac{d}{2}$-admissible] 
\label{t:fracSchrononsharp} 
Suppose $\alpha \in \mathbb{R} \setminus \{0,1\}$. Let $d\geq 1$ and suppose $(q,r)$ is non-sharp $\frac{d}{2}$-admissible with $\frac{d}{r} + \frac{\alpha}{q} < d$. 
\vspace{-7pt}
\begin{enumerate}
[leftmargin=0.8cm, labelsep= 0.3cm, topsep=0pt]
\item[$(i)\,$] If $(\frac{1}{r},\frac{1}{q})$ belongs to $\mathrm{int}(OA_{\frac{d}{2}}C)$, then
\eqref{FS-ortho} holds for all families of orthonormal functions $(f_j)_j$ in $\dot{H}^s$, with $s = \frac{d}{2} - \frac{d}{r} - \frac{\alpha}{q}$, and $\beta = \beta_{\frac{d}{2}}(q,r)$. This is sharp in the sense that the estimate fails if $\beta > \beta_{\frac{d}{2}}(q,r)$.
\item[$(ii)$] If $d \geq 3$ and $(\frac{1}{r},\frac{1}{q})$ belongs to $\mathrm{int}(ODA_{\frac{d}{2}}E_{\frac{d}{2}})$, then
\eqref{FS-ortho} holds for all families of orthonormal functions $(f_j)_j$ in $\dot{H}^s$, with $s = \frac{d}{2} - \frac{d}{r} - \frac{\alpha}{q}$, and $\beta  < \frac{q}{2}$. This is sharp in the sense that the estimate fails if $\beta > \frac{q}{2}$.
\end{enumerate}
\end{theorem}

\begin{proof}[Proof of Theorem \ref{t:fracSchrononsharp} (sufficiency part)]
As in the proof of Theorems \ref{t:wavenonsharp}--\ref{t:KGnonsharp_Schro}, the claimed estimates in $(ii)$ follow from those in $(i)$. To prove the claimed estimates in $(i)$, we apply Proposition \ref{p:nonsharp} with $\Psi(\xi) = |\xi|$ and $s(q,r) = \frac{d}{2} - \frac{d}{r} - \frac{\alpha}{q}$. For the estimate \eqref{e:ONSrefinedsharp}, we use Theorem \ref{t:fracSchrosharp}. 
\end{proof}

\section{Necessary conditions}
\label{section:necesssary}
Let $\widetilde \chi_0\in C_c^\infty(2^{-2}, 2^2)$ and  $\widetilde \chi_0\gtrsim 1$ on $(2^{-1}, 2)$. 
For a given $(q,r) \in [2,\infty) \times [2,\infty)$, consider the frequency localized estimate
\begin{equation}\label{e:cutoffONS}
\bigg\| \sum_j\nu_j |U_\phi \widetilde \chi_0(|D|)f_j|^2 \bigg\|_{\frac q2,\frac r2}\lesssim \|\nu \|_{\ell^\beta}.
\end{equation} 
We show this estimate implies necessary conditions under fairly mild conditions on the dispersion relation $\phi$.
The following are based on a slight generalization of the construction given in \cite{BHLNS}.

\begin{proposition} 
\label{p:necessary}
Suppose \eqref{e:cutoffONS} holds for all orthonormal family $(f_j)_j$ and $\phi$ is continuously differentiable away from the origin.  Then, 
we have   \begin{equation} \label{e:necessaryL}
\beta \le \beta_{\frac{d}{2}}(q,r). 
\end{equation}
Additionally if  $\phi$ is nonnegative and radial with $\phi(\xi) = \phi_0(|\xi|)$ and $\phi_0$ is continuously differentiable away from the origin and increasing, then
\begin{equation} 
\label{e:necessaryU}
\beta \le \frac{q}{2}.
\end{equation}
\end{proposition}

From this, the necessity claims in Theorems \ref{t:fracSchrosharp}--\ref{t:KGwave}, and Theorems  \ref{t:wavenonsharp}--\ref{t:fracSchrononsharp} all follow.
Indeed, the orthonormal Strichartz estimates in these theorems clearly implies the corresponding  frequency localized estimate of the form \eqref{e:cutoffONS}
with suitable choices of $\widetilde \chi_0$. The exponents $q,r,$ and $\beta$ should satisfy \eqref{e:necessaryL} and \eqref{e:necessaryU}.

It seems reasonable to expect that our estimates in $(i)$ of Theorems \ref{t:wave}, \ref{t:KGwave}, \ref{t:wavenonsharp} and \ref{t:KGnonsharp_wave} (corresponding to the line segment $(A_{\frac{d-1}{2}},C)$ in the sharp $\frac{d-1}{2}$-admissible case) are also sharp with respect to the range of allowable $\beta$. 
This would follow if we could have $\beta \le \beta_{\frac{d-1}{2}}(q,r)$. Unfortunately, we are not able to reach this conclusion yet. 
% from \eqref{e:necessaryL} and \eqref{e:necessaryU}. 

\begin{proof}[Proof of Proposition \ref{p:necessary} ]
To show \eqref{e:necessaryL} we use the family of vectors 
\[
(v_j)_{j \in J} := \{\xi \in R^{-1}\mathbb{Z}^d: 2^{-1}\le  |\xi|\le 2 \}
\] 
where $\# J \sim R^{d}$. If we let 
$$
\widehat{f}_j(\xi) = c R^{d/2} \chi (2R|\xi-v_j|)
$$
for each $j\in J$, then $(f_j)_j$ becomes orthonormal system in $L^2(\mathbb{R}^d)$ for an appropriate choice of the constant $c$. Note that 
$$
|U_\phi \widetilde \chi_0(|D|) f_j(x)| \gtrsim R^{-d/2} \1_{\{|t|\lesssim R^{2}, |x+t\nabla\phi(v_j)|\lesssim R\}}.
$$ 
Also, we have $|\nabla\phi(v_j)| \lesssim 1$ since $|v_j| \sim 1$ and this gives the uniform lower bound 
$\1_{\{|t|\lesssim R^{2}, |x+t\nabla\phi(v_j)|\lesssim R\}} \gtrsim \1_{\{|t|,|x| \lesssim R\}}$. Therefore, if we assume \eqref{e:cutoffONS} holds, then 
$$
R^{\frac2q + \frac{2d}r} \lesssim R^{\frac d\beta}. 
$$
By taking $R\to\infty$, we see that $\beta \le \beta_{\frac{d}{2}}(q,r)$.

To get the second condition  \eqref{e:necessaryU}, let us consider 
\[
 f_j = c e^{-ij\phi(D)} g,
\]
where 
$$
\widehat g(\xi) = \1_{[ \phi_0^{-1}(\ell \pi), \phi_0^{-1}((\ell + 2)\pi)]} (|\xi|) |\xi|^{-\frac{d-1}{2}} \phi_0'(|\xi|)^{1/2},
$$
$c$ is a constant to be chosen momentarily, and we here choose $\ell \in \mathbb{Z}$ so that the set $\{ \phi_0^{-1}(\ell \pi)\le |\xi|\le \phi_0^{-1}((\ell + 2)\pi)\}\cap  
\{2^{-1}\le  |\xi|\le 2 \}$ has nonzero measure. Then, by changing to spherical coordinates, it is easy to check that $(f_j)_j$ is an orthonormal family by choosing an appropriate constant $c$. 

Notice that  
$$
\bigg\| \sum_j \nu_j |U_\phi \widetilde \chi_0(|D|) f_j|^2 \bigg\|_{\frac q2, \frac r2}^{\frac q2} \geq \sum_{n\in\mathbb{Z}} \int_{n}^{n+\varepsilon} \nu_n^{\frac q2} \| e^{i(t-n)\phi(D)} g \|_{r}^{q}\, \mathrm{d}t \gtrsim \|\nu\|_{\ell^{\frac q2}}^{\frac q2},
$$
by choosing $\varepsilon>0$ small enough so that $ \| e^{is\phi(D)} g \|_{r} \sim \| g \|_r $ uniformly in $s\in (0,\varepsilon)$. Therefore, if we assume \eqref{e:cutoffONS} holds, then we see $\| \nu \|_{\frac q2} \lesssim \| \nu \|_\beta$ which shows \eqref{e:necessaryU}. 
\end{proof}

\begin{remark} 
\label{re:klein_gordon}
In cases where the Sobolev exponent $s$ is determined by scaling considerations in terms of $q$ and $r$, the summability exponent $\beta$ is considered as a function of these exponents. In other cases such as the Klein-Gordon equation
$$
\bigg\| \sum_j \nu_j |e^{it\sqrt{1 - \Delta}} f_j|^2 \bigg\|_{\frac q2, \frac r2} \lesssim \| \nu \|_{\ell^\beta}
$$ 
for families of orthonormal functions $(f_j)_j$ in $H^s$, it is reasonable to expect that the sharp value of $\beta$ depends on $q,r$ and $s$. To highlight this, for simplicity, we consider the case $q = r = \frac{2(d+1)}{d-1}$. Regarding $(q,r)$ as a non-sharp $\frac{d}{2}$-admissible pair, it follows from Theorem \ref{t:KGnonsharp_Schro} that we have the estimate
\begin{equation}\label{e:KG_ST_Schro}
\bigg\| \sum_j \nu_j |e^{it\sqrt{1-\Delta}} f_j|^2  \bigg\|_{\frac{2(d+1)}{d-1}} \lesssim \| \nu \|_{\ell^{\beta_1}}
\end{equation} 
for any family of orthonormal functions $(f_j)_j$ in $H^{s_1}$, where 
$$
(\beta_1,s_1) = \bigg(\frac{d}{d-1},\frac{d^2+3d-2}{2d(d+1)}\bigg).
$$
On the other hand, regarding $(q,r)$ as a sharp $\frac{d-1}{2}$-admissible pair, from Theorem \ref{t:KGFS} we have 
\begin{equation}\label{e:KG_ST_wave}
\bigg\| \sum_j \nu_j |e^{it\sqrt{1-\Delta}} f_j|^2  \bigg\|_{\frac{2(d+1)}{d-1}} \lesssim \| \nu \|_{\ell^{\beta_2}}
\end{equation} 
for any family of orthonormal functions $(f_j)_j$ in $H^{s_2}$, where 
$$
(\beta_2,s_2) = \bigg(\frac{d+1}{d},\frac12\bigg). 
$$
Easy computations show $\beta_1 > \beta_2$ and $s_1 > s_2$. So, from the viewpoint of the regularity, \eqref{e:KG_ST_wave} is better than \eqref{e:KG_ST_Schro}, but if we want further gain in the summability exponent $\beta$, then \eqref{e:KG_ST_Schro} is better. 
It seems to be an interesting problem  
to identify the sharp value of $\beta$ for allowable $(q,r,s)$ but we do not pursue this here.
\end{remark}

\section{Applications} \label{section:applications}
In this section, we present several applications of our results towards local well-posedness of infinite systems of Hartree type, weighted velocity averaging estimates for kinetic transport equations, and refined versions of classical Strichartz estimates. As we touched on in the Introduction, that Strichartz estimates for orthonormal families of initial data enjoy such a variety of connections has already been observed by other authors and we clarify this at appropriate points in the remainder of this section. Our purpose here is to obtain improvements in various respects over the existing literature in this direction based on our progress on Strichartz estimates for orthonormal families obtained in this paper. Although broadly speaking the methods overlap with existing work, we believe there is value in taking care to include some detail, for clarity of the exposition, and since it allows us to draw certain connections between the applications themselves. 

To begin with, it will helpful to introduce an operator-theoretic version of the Strichartz estimates for orthonormal families which is useful for the forthcoming applications. For this, we introduce the notion of the density function $\rho_\gamma$ (also written $\rho[\gamma]$ where appropriate) of certain classes of compact and self-adjoint operators $\gamma$ on $L^2(\mathbb{R}^d)$. Formally, $\rho_\gamma(x) = \gamma(x,x)$, where (with the usual abuse of notation) $\gamma(x,y)$ denotes the integral kernel of the operator $\gamma$. In terms of the spectral decomposition $\gamma = \sum_j \nu_j \Pi_{f_j}$, where $(f_j)_j$ is an orthonormal family in $L^2(\mathbb{R}^d)$ and $\Pi_{f_j} := \langle \cdot, f_j \rangle f_j$ is the orthogonal projection onto the span of $f_j$, we formally obtain
$$
\rho_\gamma (x) = \sum_j \nu_j |f_j(x)|^2.
$$ 
It turns out that the estimate 
\begin{equation} \label{e:ONS_abstract}
\bigg\| \sum_j \nu_j |U_\phi \Psi(D)^{-s} f_j|^2 \bigg\|_{\frac q2,\frac r2} \leq C_* \| \nu \|_{\ell^\beta}, \qquad \textrm{$(f_j)_j$ orthonormal in $L^2(\mathbb{R}^d)$}
\end{equation}
ensures that the density function $\rho[\Psi(D)^{-s}\gamma(t) \Psi(D)^{-s}]$ is well defined in $L^{\frac{q}{2}}_t L^{\frac{r}{2}}_x$, where $\gamma(t) = e^{it\phi(D)} \gamma_0 e^{-it\phi(D)}$, and satisfies
\begin{equation}\label{e:ONS_abstract_op}
\left\|\rho [\Psi(D)^{-s} \gamma(t) \Psi(D)^{-s} ]\right\|_{\frac q2,\frac r2}
\leq C_* \|\gamma_0\|_{\mathcal{C}^\beta(L^2)} 
\end{equation}
whenever $\gamma_0 \in \mathcal{C}^\beta(L^2)$. Here $q,r \geq 2$, $\beta \geq 1$, $s \in \mathbb{R}$, and $\Psi(\xi) = |\xi|$ or $\Psi(\xi) = \langle \xi \rangle$. We refer the reader to \cite{FLLS} and \cite{FS_AJM} for further details.

As a further word on notation, in the remainder of this section we usually abbreviate $\mathcal{C}^\beta(L^2)$ to $\mathcal{C}^\beta$. Furthermore, $\mathcal{C}^{\beta,s}$ denotes the Sobolev-type Schatten space with norm
$$
\| \gamma \|_{\mathcal{C}^{\beta,s}} = \| \Psi(D)^s \gamma \Psi(D)^s \|_{\mathcal{C}^\beta(L^2)}. 
$$

Also, the homogeneous and inhomogeneous Besov spaces $\dot{B}^s_{q,r}$ and $B^s_{q,r}$ will arise in this section; their norms are given by
\begin{align*}
\|f\|_{\dot{B}^s_{q,r}}& = \Big(\sum_{j \in \mathbb{Z}} (2^{js} \|P_j f\|_q)^r \Big)^\frac{1}{r}\,, \\
\|f\|_{B^s_{q,r}}&  = \|P_{\leq 0} f\|_q + \Big(\sum_{j \geq 1} (2^{js} \|P_j f\|_q)^r \Big)^\frac{1}{r} 
\end{align*}
for finite $r$, with the obvious modification in the case $r = \infty$. Here, $(P_j)_{j \in \mathbb{Z}}$ is a smooth Littlewood--Paley family of operators associated to dyadic annuli given by
\begin{equation} \label{e:LP}
\widehat{P_j f}(\xi) = \chi_0(2^{-j} |\xi|) \widehat{f}(\xi),
\end{equation}
where $\chi_0$ is a non-negative and smooth function supported on $[\frac{1}{2},2]$ such that $\sum_{j \in \mathbb{Z}} \chi_0(2^{-j} |\xi|) = 1$ for $\xi \neq 0$. Also, $\widehat{P_{\leq 0} f}(\xi) = \chi(\xi) \widehat{f}(\xi)$, where $\chi$ is a non-negative and smooth function supported on $[-\frac{1}{2},\frac{1}{2}]$.

\subsection{Local well-posedness for Hartree-type infinite systems}
Estimates of the form \eqref{e:ONS_abstract} are a powerful tool in the analysis of infinite systems. For concreteness, we begin by discussing the infinite system of fractional Schr\"odinger equations on $\mathbb{R}^{1+d}$ with Hartree-type nonlinearity
 \begin{equation} \label{e:fractionalsystem}
 \begin{cases}
 i\partial_t u_j = (-\Delta)^{\alpha/2}u_j + (w * \rho)u_j ,\quad \quad j \in \mathbb{N} \\
 u_j(0,\cdot)=f_j,\\
 \end{cases}
 \end{equation}
where $\rho = \sum_{k=1}^\infty |u_k|^2$ is the density function and $w$ is the interaction potential. We focus on the case where the initial data $(f_j)_j$ forms an orthonormal family in $L^2(\mathbb{R}^d)$, in which case we necessarily have $\alpha \geq 2$ in order for \eqref{e:ONS_abstract} to hold when $s=0$. Recent developments on the rigorous mathematical theory of the above system in the case $\alpha = 2$ have made significant use of estimate \eqref{e:ONS_abstract} (or variants which incorporate additional smoothing effects). We refer the reader to \cite{CHP-1}, \cite{CHP-2}, \cite{FS_AJM}, \cite{LewinSabin-1}, \cite{LewinSabin-2}, and \cite{Sabin_survey} for further background and details on these applications, and to other recent developments in the theory of such systems in \cite{HKY} (on finite-time blowup criteria) and \cite{GLN} (on ground states). We also remark that in the case of a single particle, the system \eqref{e:fractionalsystem} reduces to 
 \begin{equation*} 
  i\partial_t u = (-\Delta)^{\alpha/2}u + (w * |u|^2)u  
  \end{equation*}
and this has been the subject of numerous papers; see, for example, \cite{CHHO}, \cite{CHKL}, \cite{GuoZhu} and the references therein.

In the standard way, in order to treat the infinite system \eqref{e:fractionalsystem} we consider the operator formalism
\begin{equation} \label{e:fractionalsystem_op}
  \left\{
  \begin{array}{ll}
	i\partial_t\gamma=[(-\Delta)^{\alpha/2}+w\ast\rho_\gamma,\gamma],  \quad (t,x)\in\mathbb{R}^{1+d}, \ d\ge1, \\
	\gamma(t,\cdot) = \gamma_0.
  \end{array} %\right.
  \right.
\end{equation}
Here $\gamma=\gamma(t)$ is a bounded and self-adjoint operator on $L^2(\mathbb{R}^d)$, $\rho_\gamma$ is the associated density function and $[\cdot,\cdot]$ denotes the commutator. Regarding this equation, we have the following local-in-time existence theorem.
\begin{theorem} \label{t:LWP_frac}
Let $d \geq 1$, $\alpha > 2$, $r \in [2,r_\alpha)$ and $\frac{\alpha}{q} = d(\frac{1}{2} - \frac{1}{r})$, where $r_\alpha = \frac{2d}{d-\alpha}$ for $\alpha < d$, and $r_\alpha = \infty$ for $\alpha \geq d$. For $d=1$ we assume $\beta \in [1,\beta_\frac{1}{2}(q,r)]$, and for $d\geq 2$ we assume $\beta$ satisfies
\begin{equation*}
\beta \in 
\left\{
\begin{array}{lllllllll}
[1,\beta_\frac{d}{2}(q,r)] & \quad \textrm{if $r < \frac{2(d+\alpha-1)}{d-1}$} \vspace{2mm}\,,  \\
\textrm{$[1,\frac{q}{2}) $} & \quad \textrm{if $r \geq \frac{2(d+\alpha-1)}{d-1}$.}
\end{array}
\right.
\end{equation*}
Then, for any $w \in L^{(r/2)'}$ and $\gamma_0\in\mathcal{C}^{\beta}$, there exist $T=T(\|\gamma_0\|_{\mathcal{C}^{\beta}}, \|w\|_{L^{(r/2)'}}) > 0$ and a unique solution $\gamma\in C^0_t([0,T];\mathcal{C}^{\beta})$ to 
\eqref{e:fractionalsystem_op}
on $[0,T]\times\mathbb{R}^d$ whose density $\rho_\gamma$ belongs to $L^{q/2}_tL^{r/2}_x([0,T]\times\mathbb{R}^d)$.
\end{theorem}
Our proof of Theorem \ref{t:LWP_frac} follows the contraction mapping arguments in \cite{FS_AJM} and \cite{LewinSabin-1} for the Duhamel formulation of \eqref{e:fractionalsystem_op}, into which we incorporate our new estimates of the form \eqref{e:ONS_abstract} from Theorem \ref{t:fracSchrononsharp}. The argument is robust and similar applications are possible for wider classes of equations of the type
\begin{equation}\label{e:FracHa}
  \left\{
  \begin{array}{ll}
	i\partial_t\gamma=[\phi(D)+w\ast\rho_\gamma,\gamma],   \quad (t,x)\in\mathbb{R}^{1+d},\, \, d\ge1, \\
	\gamma(t,\cdot) = \gamma_0,
  \end{array} %\right.
  \right.
\end{equation}
and thus we present an abstract statement in the forthcoming Proposition \ref{p:wellposed_abstract} based on the assumption that the associated propagator $U_\phi$ satisfies the estimate \eqref{e:ONS_abstract}. In addition, for control of the nonlinearity we shall also assume the operator norm estimates
\begin{equation} \label{e:nonlinearitycontrol}
\| \Psi(D)^{\pm s} (w * \rho) \Psi(D)^{\mp s} \|_{\mathcal{C}^\infty} \leq C_{**} \|w\|_{\mathcal{X}_{r,s}} \| \rho \|_\frac{r}{2},
\end{equation}
where $\mathcal{X}_{r,s}$ is an appropriate function space. It is of course immediate from H\"older's inequality that in the case $s = 0$ (needed for Theorem \ref{t:LWP_frac})  we may take the Lebesgue space $\mathcal{X}_{r,0} = L^{(r/2)'}$. Some remarks on the case of general $s$ will follow later.

\begin{proposition} \label{p:wellposed_abstract}
Suppose \eqref{e:ONS_abstract} and \eqref{e:nonlinearitycontrol} hold. Then, for any $w \in \mathcal{X}_{r,s}$ and $\gamma_0\in\mathcal{C}^{\beta,s}$, there exist $T=T(\|\gamma_0\|_{\mathcal{C}^{\beta,s}}, \|w\|_{\mathcal{X}_{r,s}}) > 0$ and a unique solution $\gamma\in C^0_t([0,T];\mathcal{C}^{\beta,s})$ to \eqref{e:FracHa} on $[0,T]\times\mathbb{R}^d$ whose density $\rho_\gamma$ satisfies $\rho_\gamma\in L^{q/2}_tL^{r/2}_x([0,T]\times\mathbb{R}^d)$.
\end{proposition}
By the Duhamel principle  the solution of the inhomogeneous equation  
\begin{equation*}
  \left\{
  \begin{array}{ll}
	i\partial_t\gamma=[\phi(D),\gamma] + R(t),  \quad (t,x)\in\mathbb{R}^{1+d},  \ d\ge1 \\
	\gamma(t,\cdot) = \gamma_0,
  \end{array} %\right.
  \right.
\end{equation*}
is formally given by 
\[  
\gamma(t) = e^{-it\phi(D)}\gamma_0e^{it\phi(D)}
-i\int_0^t e^{-i(t-t')\phi(D)}R(t')e^{i(t-t')\phi(D)}\, \mathrm{d}t'
\]
and thus the following proposition provides control on the inhomogeneous component of the Duhamel formula for \eqref{e:FracHa}.
\begin{proposition} \label{p:inhomONS}
%\vspace{-7pt}
\
\begin{enumerate} 
[leftmargin=.8cm, labelsep=0.3 cm, topsep=0pt]
\item[$(i)\,$] If \eqref{e:ONS_abstract} holds, then for any $V\in L^{(q/2)'}_tL^{(r/2)'}_x$ we have 
\begin{equation} \label{e:dual_ONS}
\bigg\| \int_{\mathbb{R}}e^{it\phi(D)} \Psi(D)^{-s} V(t,\cdot) \Psi(D)^{-s} e^{-it\phi(D)}\, \mathrm{d}t \bigg\|_{\mathcal{C}^{\beta'}} \leq C_{*} \|V\|_{(\frac{q}{2})',(\frac{r}{2})'}.
\end{equation}
\vspace{2mm}
\item[$(ii)\,$] (Inhomogeneous estimate) If \eqref{e:ONS_abstract} holds, then
\begin{equation}\label{e:inhom}
\|\rho[\Psi(D)^{-s} \gamma_{R}(t) \Psi(D)^{-s}] \|_{\frac{q}{2},\frac{r}{2}}
\leq 4C_{*}
\bigg\|\int_{\mathbb{R}}e^{it\phi(D)}|R(t)|e^{-it\phi(D)}\, \mathrm{d}t\bigg\|_{\mathcal{C}^\beta}.
\end{equation}
 Here, for each $t \in \mathbb{R}$, $R(t):L^2\to L^2$ is a self-adjoint operator and
\[%\begin{equation}\label{e:26Feb}
\gamma_{R}(t)=\int_0^te^{-i(t-t')\phi(D)}R(t')e^{i(t-t')\phi(D)}\, \mathrm{d}t'.
\]%\end{equation}
\end{enumerate}
\end{proposition}

\begin{proof}
For both $(i)$ and $(ii)$, we follow the argument in \cite[Section 2.3]{FLLS} and so we will be brief with details. 

For $(i)$, by duality it suffices to show
\begin{equation} \label{e:dual_Xray}
\bigg|{\rm Tr}\bigg( \gamma_0 \int_{\mathbb{R}}e^{it\phi(D)} \Psi(D)^{-s} V(t,\cdot)  \Psi(D)^{-s} e^{-it\phi(D)}\, \mathrm{d}t \bigg)\bigg| \leq  C_{*} \|V\|_{(\frac{q}{2})',(\frac{r}{2})'}
\end{equation}
whenever $\|\gamma_0 \|_{\mathcal{C}^\beta}=1$. Using the cyclic property of trace, it follows that the left-hand side of \eqref{e:dual_Xray} coincides with
\[
\bigg| \int_{\mathbb{R}} \int_{\mathbb{R}^d} \rho[\Psi(D)^{-s}\gamma(t) \Psi(D)^{-s}](x) V(t,x) \, \mathrm{d}x \mathrm{d}t\bigg|,
\]
where $\gamma(t) = e^{-it\phi(D)} \gamma_0 e^{it\phi(D)}$, and thus \eqref{e:dual_Xray} follows from H\"older's inequality and \eqref{e:ONS_abstract_op}.

For $(ii)$, we again proceed using duality. First, for non-negative functions $V \in L^{(q/2)'}_tL^{(r/2)'}_x$ we write
\begin{align*}
\int_{\mathbb{R}} \int_{\mathbb{R}^d} \rho[\Psi(D)^{-s}\gamma_R(t) \Psi(D)^{-s}](x) V(t,x) \, \mathrm{d}x \mathrm{d}t = \int_\mathbb{R} \int_0^t \mathrm{Tr}(A(t)  B(t')) \, \mathrm{d}t'\mathrm{d}t.
\end{align*}
Here the operators $A(t)$ and $B(t)$ are given by
\begin{align*}
A(t) & = e^{it\phi(D)} \Psi(D)^{-s} V(t,\cdot) \Psi(D)^{-s} e^{-it\phi(D)},  \quad
B(t) = e^{it\phi(D)}R(t)e^{-it\phi(D)},
\end{align*}
and since they are self-adjoint, it follows that $|{\rm Tr}(A(t)B(t'))| \leq {\rm Tr} (|A(t)| |B(t')|)$. Therefore, by H\"older's inequality for the Schatten spaces along with Part $(i)$ of this proposition, we obtain
\begin{align*}
& \Big|\int_{\mathbb{R}} \int_{\mathbb{R}^d} \rho[\Psi(D)^{-s}\gamma_R(t) \Psi(D)^{-s}](x) V(t,x) \, \mathrm{d}x \mathrm{d}t\Big| \\
& \qquad \qquad \leq  C_* \|V\|_{(\frac{q}{2})',(\frac{r}{2})'}\Big\| \int_{\mathbb{R}} e^{it' \phi(D)} |R(t')| e^{-it'\phi(D)}\, \mathrm{d}t' \Big\|_{\mathcal{C}^{\beta}}.
\end{align*}
The same inequality holds for general $V \in L^{(q/2)'}_tL^{(r/2)'}_x$ with constant $4C_*$ by appropriately decomposing $V$, and hence \eqref{e:inhom} follows from duality.
\end{proof}

\begin{proof}[Proof of Proposition \ref{p:wellposed_abstract}]
The argument below follows the proof of \cite[Theorem 14]{FS_AJM} (see also \cite{LewinSabin-1}). To begin, we set $M := \|\gamma_0\|_{\mathcal{C}^{\beta,s}}$ and define the map $\Lambda$ by
\[
\Lambda(\gamma,\rho)
=
(\Lambda_1(\gamma,\rho),\rho[\Lambda_1(\gamma,\rho)]),
\]
where
\[
\Lambda_1(\gamma,\rho)(t)
=
e^{-it\phi(D)}\gamma_0e^{it\phi(D)}
-i\int_0^t e^{-i(t-t')\phi(D)}[w \ast \rho(t'),\gamma(t')]e^{i(t-t')\phi(D)}\, \mathrm{d}t'.
\] 
We consider these mappings on the space
\begin{align*}
X_T=
\{
(\gamma,\rho)\in C^0_t([0,T];\mathcal{C}^{\beta,s})\times L^{q/2}_tL^{r/2}_x([0,T]\times\mathbb{R}^d): \| (\gamma, \rho) \|_{X_T} \leq C^*M\}
\end{align*}
equipped with the norm
\[
\| (\gamma,\rho) \|_{X_T} := \|\gamma\|_{C^0_t([0,T];\mathcal{C}^{\beta,s})}+\|\rho\|_{L^{q/2}_tL^{r/2}_x([0,T]\times\mathbb{R}^d)}.
\]
Here, $T = T(M, \|w\|_{\mathcal{X}_{r,s}}) \leq1$ is to be chosen momentarily and $C^*$ is a constant such that $C^*>\max{(10,10C_{*})}$. With this setup, it suffices to show that there is some $(\gamma,\rho) \in X_T$ which is a fixed point of $\Lambda$, and thus our goal is to show that $\Lambda$ is a contraction on the space $X_T$ if $T$ is small enough.

The proof that $\Lambda$ maps $X_T$ to itself rests on the estimates
\begin{equation}\label{e:contract-1}
\|\Lambda_1(\gamma,\rho)\|_{C^0_t([0,T];\mathcal{C}^{\beta,s})}
\leq M + 2C_{**} \|w\|_{\mathcal{X}_{r,s}} T^{1/(q/2)'}(C^*M)^2
\end{equation}
and
\begin{equation}\label{e:contract-2}
\|\rho[\Lambda_1(\gamma,\rho)]\|_{L^{q/2}_tL^{r/2}_x([0,T]\times\mathbb{R}^d)}
\leq
C_*\big(M + 8C_{**} \|w\|_{\mathcal{X}_{r,s}} T^{1/(q/2)'}(C^*M)^2\big).
\end{equation}
for any $(\gamma,\rho) \in X_T$. Once these claims are proved, then it is clear that a sufficiently small choice of $T = T(M, \|w\|_{\mathcal{X}_{r,s}})$ shows $\Lambda : X_T \to X_T$. A similar argument shows that $\Lambda$ is also a contraction, and we omit the details.

Since Schatten norms are unitarily invariant it follows that 
$
\|e^{-it\phi(D)}\gamma_0e^{it\phi(D)}\|_{\mathcal{C}^{\beta,s}}=
\|\gamma_0\|_{\mathcal{C}^{\beta,s}}= M
$
and thus for \eqref{e:contract-1} it suffices to establish
\begin{equation} \label{e:contract-1'}
\int_0^T\mathcal N \,\mathrm{d}t' \leq 2C_{**} \|w\|_{\mathcal{X}_{r,s}} T^{1/(q/2)'}(C^*M)^2,
\end{equation}
where 
\[\mathcal N=\| e^{-i(t-t')\phi(D)}[w \ast \rho(t'),\gamma(t')]e^{i(t-t')\phi(D)} \|_{\mathcal{C}^{\beta,s}}. \] 
For the integrand, we again use the unitary invariance of the Schatten norms along with H\"older's inequality to obtain
\begin{align*}
\mathcal N \leq
\big( \| \Psi(D)^s (w\ast \rho(t') ) \Psi(D)^{-s} \|_{\mathcal{C}^\infty} 
+
\| \Psi(D)^{-s} (w\ast \rho(t')) \Psi(D)^s \|_{\mathcal{C}^\infty} \big) \| \gamma(t') \|_{\mathcal{C}^{\beta,s}} 
\end{align*}
and thus, using \eqref{e:nonlinearitycontrol}, we get
\begin{align*}
 \int_0^T\mathcal N \, \mathrm{d}t'  \leq 2C_{**}\|\gamma\|_{C^0_t([0,T];\mathcal{C}^{\beta,s})} \|w\|_{\mathcal{X}_{r,s}} \int_0^T  \|\rho(t')\|_\frac{r}{2}   \, \mathrm{d}t'.
\end{align*}
Hence \eqref{e:contract-1'} follows from H\"older's inequality.

For \eqref{e:contract-2}, we first use \eqref{e:ONS_abstract_op} to see that
\[
\| \rho[e^{-it\phi(D)}\gamma_0e^{it\phi(D)}] \|_{L^{q/2}_tL^{r/2}_x([0,T]\times\mathbb{R}^d)} \leq C_*M.
\]
From this, along with \eqref{e:inhom} and unitary invariance of Schatten norms, we obtain
\begin{align*}
& \|\rho[\Lambda_1(\gamma,\rho)]\|_{L^{q/2}_tL^{r/2}_x([0,T]\times\mathbb{R}^d)} \\
& \qquad \qquad \leq  C_*M + 4C_*\int_0^T \| \Psi(D)^s [w * \rho(t'),\gamma(t')] \Psi(D)^s \|_{\mathcal{C}^\beta} \, \mathrm{d}t
\end{align*}
and estimating as in the proof of \eqref{e:contract-1} yields \eqref{e:contract-2}.
\end{proof}

We conclude with some remarks about the pseudo-relativistic case $\phi(\xi) = \langle \xi \rangle$, in which case the single-equation system of Hartree type has been studied in numerous papers; see, for example, \cite{CHHO}, \cite{FL}, \cite{Lenzmann} and further references in these papers. In this case, we have estimates of the form \eqref{e:ONS_abstract} with $\Psi(\xi) = \langle \xi \rangle$ and certain $s > 0$, and \eqref{e:nonlinearitycontrol} holds with the Besov space $\mathcal{X}_{r,s} = B^{s + \delta}_{(r/2)',\infty}$ for arbitrary $\delta > 0$. This latter assertion follows from the estimate
\begin{equation} \label{e:Triebel}
\| f g\|_{H^{s}} \leq C_{s,\delta} \|f\|_{{B}^{|s|+\delta}_{\infty,\infty}}\|g\|_{H^{s}}
\end{equation}
which is valid for all $s \in\mathbb{R}$ and $\delta>0$ (see, for example,  \cite[p. 205]{Triebel}). We also remark that the analogous estimate to \eqref{e:Triebel} on the torus was used in \cite{Nakamura} to establish local well-posedness of infinite systems of Hartree equations in the periodic setting.

\subsection{Velocity averages for kinetic transport equations}
Associated with the dispersion relation $\phi$, consider the kinetic transport equation
\begin{equation}\label{e:Kinetic}
\left\{ \begin{array}{ll}
    \partial_t F(t,x,\xi) + \nabla\phi(\xi)\cdot \nabla_xF(t,x,\xi) = 0,& (t,x,\xi) \in \mathbb{R} \times \mathbb{R}^d\times \mathbb{R}^d, \\
    F(0,x,\xi) = f(x,\xi). 
  \end{array} \right.
\end{equation} 
It is easy to see that the solution is given by the explicit formula 
\[
F(t,x,\xi) = f(x-t\nabla\phi(\xi),\xi).
\] 
The standard kinetic transport equation $\partial_t F+ \xi \cdot \nabla_xF = 0$ corresponds to the dispersion relation $\phi(\xi) = \frac{1}{2}|\xi|^2$, and the Strichartz estimates in this context
\begin{equation*}
\|F\|_{L^a_tL^b_xL^c_\xi(\mathbb{R} \times \mathbb{R}^d \times \mathbb{R}^d)} \lesssim \|f\|_{L^\beta_{x,\xi}(\mathbb{R}^d \times \mathbb{R}^d)}
\end{equation*}
were first studied by Castella and Perthame \cite{CastellaPerthame}. From the explicit form of the solution, it is obvious that it suffices to consider the case $c=1$, or equivalently, the velocity averaging estimate\footnote{As will become clear, the relabelling $(a,b) \to (\frac{q}{2},\frac{r}{2})$ is convenient for facilitating the connection to \eqref{e:ONS_abstract}}
\begin{equation} \label{e:noweightaverage}
\| \varrho f \|_{L^\frac{q}{2}_tL^\frac{r}{2}_x(\mathbb{R} \times \mathbb{R}^d)}\lesssim \|f\|_{L^{\beta}_{x,\xi}(\mathbb{R}^d \times \mathbb{R}^d)},
\end{equation}
where $\varrho$ is the velocity averaging operator given by
\[
\varrho f(t,x) = \int_{\mathbb{R}^d} f(x-t\xi,\xi)\, \mathrm{d}\xi.
\]
Combining the contribution in \cite{CastellaPerthame}, along with an observation by Keel and Tao in \cite{KeelTao}, all non-endpoint estimates \eqref{e:noweightaverage} were settled; specifically, the estimates were shown to be true for $r \in [2,\frac{2(d+1)}{d-1})$ and false for $r \in (\frac{2(d+1)}{d-1},\infty]$, where $(q,r)$ is $\frac{d}{2}$-sharp admissible and $\beta = \frac{2r}{r+1}$. The failure of the endpoint case $r = \frac{2(d+1)}{d-1}$ is due to Guo and Peng \cite{GuoPeng} (see also the work of Ovcharov \cite{Ovcharov}) when $d=1$, and \cite{BBGL_CPDE} for all $d \geq 2$. Some recent advances in the theory of nonlinear kinetic systems relied on Strichartz estimates for the standard kinetic transport equation; we refer the reader to work of Bourneveas, Calvez, Guti\'errez and Perthame \cite{BCGP} on global weak solutions for the Othmer--Dunbar--Alt kinetic model of chemotaxis, and also Ars\'enio \cite{Arsenio} and He--Jiang \cite{HJ} for contributions to the theory of the Boltzmann equation. In \cite{HJ} it is suggested that weighted Strichartz estimates may be beneficial for the study of certain kinds of potentials, and in this section we establish the following weighted extensions of the estimates in \cite{CastellaPerthame} and \cite{KeelTao} as applications of our main results in this paper.
\begin{theorem}\label{t:kinetic}
Suppose $d\ge1$. Let $(\frac{1}{r},\frac{1}{q})$ belong to $\mathrm{int}(OA_{\frac{d}{2}}C) \cup [C,A_\frac{d}{2})$ and define $\beta = \beta_{\frac{d}{2}}(q,r)$.
\vspace{-7pt}
\begin{enumerate}
[leftmargin=0.8cm, labelsep=0.3 cm, topsep=0pt]
\item[$(i)\,$]
Suppose $\alpha\in\mathbb{R}\setminus\{0,1\}$ and let $s= \frac d2 - \frac dr - \frac{\alpha}{q}$. If $\frac{d}{r} + \frac{\alpha}{q} < d$, then 
\begin{equation*}
\bigg\| \int_{\mathbb{R}^d} f(x-t|\xi|^{\alpha - 2}\xi,\xi) |\xi|^{-2s}\, \mathrm{d}\xi \bigg\|_{L^{\frac q2}_tL^{\frac r2}_x} \lesssim \| f \|_{L^{\beta}_{x,\xi}}
\end{equation*}
for all $f \in L^{\beta}_{x,\xi}$.
\item[$(ii)$]  Let  $s = \frac{d}{2} - \frac{d}{r} - \frac{d-2}{dq}$. Then
\begin{equation*}
\bigg\| \int_{\mathbb{R}^d} f(x-t\tfrac{\xi}{\langle \xi \rangle},\xi) \langle \xi\rangle^{-2s} \, \mathrm{d}\xi \bigg\|_{L^{\frac q2}_tL^{\frac r2}_x} \lesssim \| f \| _{L^{\beta}_{x,\xi}}
\end{equation*}
for all $f \in L^{\beta}_{x,\xi}$.

\end{enumerate}
\end{theorem}
\begin{remark} $(i)\,$ The case $\alpha = 2$ in $(ii)$ was established in \cite{BHLNS}. Also, in the case $\alpha > 1$, the estimates were obtained in \cite{BLNS} in the restricted range corresponding to $(\frac{1}{r},\frac{1}{q})$ in $\mathrm{int} (OFC)$, where $F = (\frac{d-1}{4d},\frac{1}{4})$; the approach in \cite{BLNS} was based on a Fourier-analytic argument and is completely different to the approach taken here.

$(ii)\,$
By duality, the velocity averaging estimate
\begin{equation} \label{e:abstract_velocity}
\bigg\| \int_{\mathbb{R}^d} f(x-t\nabla \phi(\xi),\xi) \Psi(\xi)^{-2s} \, \mathrm{d}\xi \bigg\|_{L^{\frac q2}_tL^{\frac r2}_x} \lesssim \| f \| _{L^{\beta}_{x,\xi}}
\end{equation}
is equivalent to the $X$-ray transform type estimate
\begin{equation} \label{e:average*}
\bigg\| \int_{\mathbb{R}} V(t,x+t\nabla \phi(\xi)) \Psi(\xi)^{-2s}\, \mathrm{d}t \bigg\|_{L^{\beta'}_{x,\xi}} \lesssim \|V\|_{L^{(q/2)'}_tL^{(r/2)'}_x}.
\end{equation}
Since $e^{it\phi(D)} x e^{-it\phi(D)} = x + t\nabla \phi(D)$, formally at least we have
\[
e^{it\phi(D)} V(t,x) e^{-it\phi(D)} = V(t,x + t\nabla \phi(D))
\]
and thus \eqref{e:average*} has a striking resemblance to \eqref{e:dual_ONS}. In fact, the latter implies the former via a semi-classical limiting argument and this is the mechanism through which we establish Theorem \ref{t:kinetic} as an application of our main results in this paper. The following proposition formalizes this link, albeit at the level of the density functions $\rho_\gamma$ and $\varrho$, rather than their dual counterparts in \eqref{e:dual_ONS} and \eqref{e:average*}.
\end{remark}

\begin{proposition}\label{p:density}
Suppose \eqref{e:ONS_abstract} holds for some $q,r \ge 2$, $s\in\mathbb{R}$, $\beta = \beta_{\frac{d}{2}}(q,r)$ and $\phi$ is even. Then \eqref{e:abstract_velocity} holds.
\end{proposition}
In the case of the Schr\"odinger equation $\phi(\xi) = |\xi|^2$ when $(q,r)$ is sharp $\frac{d}{2}$-admissible, Proposition \ref{p:density} can be found in \cite[Lemma 9]{Sabin_survey}. A more general statement for the Schr\"odinger equation can be found in \cite[Proposition 5.1]{BHLNS}. Our proof of Proposition \ref{p:density} follows the same lines of argument and thus we simply provide a sketch; for further details, we refer the reader to \cite{Sabin_survey} and \cite{BHLNS}.
\begin{proof}[Proof of Proposition \ref{p:density}]
Fix any $f \in \mathcal{S}(\mathbb{R}^d)$ and, for $h > 0$, define the quantization $\gamma_{h}$ by
\[
\gamma_{h}\psi(x)
=
\int_{\mathbb{R}^{2d}}
f(\tfrac h2(x+x'),\xi)e^{i(x-x')\cdot\xi}\psi(x')\, \mathrm{d}\xi\mathrm{d}x',\;\;\; \psi\in\mathcal{S}(\mathbb{R}^d).
\]
Let us use $K_h(x,x')$ to denote the integral kernel of $\gamma_h$. Then we have 
$$
\widehat{K_h}(\xi,\xi') = \big( \tfrac{2\pi}{h} \big)^{d} \mathcal{F}_x [ f(\cdot, \tfrac{\xi-\xi'}2)](\tfrac{\xi+\xi'}{h})
$$
and from the definition of the density function and changes of variable, we have 
\begin{align*}
&\qquad\rho[\Psi(D)^{-s}\gamma_h(t/h)\Psi(D)^{-s}](x/h) \\
& =\int_{\mathbb{R}^{2d}} e^{-i\frac{t}{2h} (\phi(h\xi-\xi') - \phi(\xi')) } \Psi(h\xi-\xi')^{-s} \Psi(\xi')^{-s} \mathcal{F}_x [ f(\cdot, \tfrac{h\xi-2\xi'}2)](\xi) e^{ix\cdot \xi}\, \mathrm{d}\xi\mathrm{d}\xi' .
\end{align*}
Hence,  it is clear that 
\begin{align*}
&\lim_{h\to0}\rho[\Psi(D)^{-s}\gamma_h(t/h)\Psi(D)^{-s}](x/h) \\
=&  \int_{\mathbb{R}^{2d}} e^{-i\frac{t}{2} \nabla \phi(\xi') \cdot \xi }  \Psi(\xi')^{-2s} \mathcal{F}_x [ f(\cdot, -\xi')](\xi) e^{ix\cdot \xi}\, \mathrm{d}\xi\mathrm{d}\xi' \\
=& \int_{\mathbb{R}^d} f(x-\tfrac t2 \nabla\phi(\xi'),-\xi') \Psi(\xi')^{-2s}\, \mathrm{d}\xi'.
\end{align*}
This and \eqref{e:ONS_abstract_op} imply 
\begin{align*}
\bigg\| \int_{\mathbb{R}^d} f(x-t\nabla\phi(\xi),\xi)\Psi(\xi)^{-2s}\, \mathrm{d}\xi \bigg\|_{\frac q2, \frac r2} 
&\lesssim \liminf_{h\to0} \big\| \rho[\Psi(D)^{-s}\gamma_h(t/h)\Psi(D)^{-s}](\cdot/h) \big\|_{\frac q2,\frac r2} \\
&\lesssim \liminf_{h\to0} h^{\frac2q + \frac{2d}{r}} \| \gamma_h \|_{\mathcal{C}^\beta} \nonumber. 
\end{align*}
For the right-hand side, we have $
\| \gamma_h \|_{\mathcal{C}^\beta} \lesssim h^{-\frac d\beta} \| f \|_{L^\beta_{x,v}}
$
for $h > 0$ sufficiently small, and the result follows since $\beta = \beta_\frac{d}{2}(q,r)$.
\end{proof}
\begin{proof}[Proof of Theorem \ref{t:kinetic}]
Combine Theorems \ref{t:fracSchrosharp}, \ref{t:KGSchro},  \ref{t:KGnonsharp_Schro} and \ref{t:fracSchrononsharp} with Proposition \ref{p:density}.
\end{proof}

\subsection{Refined Strichartz estimates}
Dispersive PDE theory has been significantly enriched by the exploitation of refined versions of the classical single-function Strichartz estimates, including their decisive role in the construction of profile decompositions and understanding of concentration phenomena; see, for example,  \cite{BV}, \cite{BourgainIMRN}, \cite{CK}, \cite{CHKLY}, \cite{MVV_IMRN}, \cite{MVV_Duke}, \cite{Ramos}. In particular, we emphasize the key role of the papers by Bourgain \cite{BourgainIMRN} and Moyua--Vargas--Vega \cite{MVV_IMRN}, \cite{MVV_Duke} in many of these developments.  

An example of a refined Strichartz estimate takes the form
\begin{equation} \label{e:refined_goal}
\| U_\phi f \|_{L^q_tL^r_x} \lesssim \|f\|_\mathcal{X}^{1-\theta} \|f\|_\mathcal{H}^\theta
\end{equation}
for some $\theta \in (0,1)$ and some norm $\|\cdot\|_\mathcal{X}$ which satisfies $\|f\|_\mathcal{X} \lesssim \|f\|_\mathcal{H}$. The role of the norm $\|\cdot\|_\mathcal{X}$ is to identify the extent to which the frequency support of the initial data concentrates on certain subsets of $\mathbb{R}^d$. Frank and Sabin \cite{FS_AJM} observed that estimates of the form 
\begin{equation} \label{e:ONSgeneral-S7}
\Big\| \sum_j  \nu_j |U_\phi f_j| ^2 \Big\|_{L^{\frac q2}_tL^{\frac r2}_x} \lesssim \| \nu \|_{\ell^\beta}
\end{equation}
for families of orthonormal functions $(f_j)_j$ in $\mathcal{H}$, lead to 
\begin{equation} \label{e:abstract_refinedStrichartz}
\|U_\phi f \|_{L^q_tL^r_x} \lesssim \Big( \sum_{j \in \mathbb{Z}} \| P_jf \|_\mathcal{H}^{2\beta} \Big)^{\frac{1}{2\beta}},
\end{equation}
as long as $r < \infty$. In the case $\mathcal{H} = \dot{H}^s(\mathbb{R}^d)$, the right-hand side of \eqref{e:abstract_refinedStrichartz} is the homogeneous Besov norm $\|f\|_{\dot{B}^s_{2,2\beta}}$ and if $\beta > 1$ we immediately obtain \eqref{e:refined_goal} with $\theta = \frac{1}{\beta}$ and
\[
\|f\|_\mathcal{X} = \sup_{k \in \mathbb{Z}} \|P_k f\|_{\dot{H}^s}.
\]
In this regard, Besov space refinements of the classical Strichartz estimates are highly desirable. Prior to the recent developments on estimates of the form \eqref{e:ONSgeneral-S7}, typically refined Strichartz estimates were established using results or ideas from bilinear Fourier restriction theory, in many cases relying on sophisticated  bilinear estimates (see, for example, \cite{BV} and \cite{Ramos} which relied upon bilinear estimates in \cite{Tao_GAFA} and \cite{Tao_MathZ}, respectively). The observation of Frank--Sabin that Besov-space refinements are derivable from \eqref{e:ONSgeneral-S7} opens a fresh line of attack which is likely to be fruitful in future developments.

Although in certain contexts it is necessary to obtain yet further refinements of \eqref{e:abstract_refinedStrichartz} which detect concentration on smaller sets (such as dyadic cubes), as discussed in more detail below in the context of the fractional Schr\"odinger equation, the Besov space refinements \eqref{e:abstract_refinedStrichartz} themselves suffice for certain applications to mass concentration phenomena. Beforehand, we briefly recall the argument given by Frank and Sabin for generating \eqref{e:abstract_refinedStrichartz} from \eqref{e:ONSgeneral-S7}. In doing so, we observe that whenever $(P_j)_{j \in \mathbb{Z}}$ is a family of operators acting on a suitable subspace of $\mathcal{H}$ satisfying the following properties (1)--(4), we obtain \eqref{e:abstract_refinedStrichartz} from \eqref{e:ONSgeneral-S7}.

\begin{enumerate}
[leftmargin=0.8cm, labelsep=0.3 cm, topsep=0pt]
\item \emph{Decomposition of the identity:} $f = \sum_{j \in \mathbb{Z}} P_j f.$

\item   \emph{Commutativity with the solution operator:} $U_\phi P_jf = P_j U_\phi f$ for all $j \in \mathbb{Z}$. \footnote{Strictly speaking, the right-hand side should be interpreted as $P_j$ acting on the function $x \mapsto U_\phi f(x,t)$ for fixed $t \in \mathbb{R}$.}

\item  \emph{Littlewood--Paley inequality:} 
$
\| f \|_{L^r(\mathbb{R}^d)} \lesssim \| ( \sum_{j \in \mathbb{Z}} |P_jf|^2 )^\frac{1}{2} \|_{L^r(\mathbb{R}^d)}.
$

\item \emph{Almost orthogonality:} $N \in \mathbb{N}$ exists such that $P^*_k P_j = 0$ whenever $|j-k| \geq N$.
\end{enumerate}

Indeed, as a result of the first three properties we have
\begin{align*}
\|U_\phi f \|_{L^q_tL^r_x} \lesssim \Big\| \sum_{j \in \mathbb{Z}}  |U_\phi P_jf| ^2 \Big\|_{L^{\frac q2}_tL^{\frac r2}_x}^\frac{1}{2}.
\end{align*}
From the almost orthogonality property, we may decompose $\mathbb{Z} = \cup_{k = 1}^{N} J_k$ into disjoint index sets such that, for each $k \in \{1,\ldots,N\}$, the family $(P_jf)_{j \in J_k}$ is orthogonal in $\mathcal{H}$. Hence, by the triangle inequality in $L^{\frac q2}_tL^{\frac r2}_x$ followed by \eqref{e:ONSgeneral-S7}, we obtain
\begin{equation*} 
\|U_\phi f \|_{L^q_tL^r_x} \lesssim \Big( \sum_{k=1}^N \Big( \sum_{j \in J_k} \|P_jf\|_\mathcal{H}^{2\beta} \Big)^\frac{1}{\beta} \Big)^\frac{1}{2}
\end{equation*}
and thus, by the equivalence of norms on finite-dimensional spaces, \eqref{e:abstract_refinedStrichartz} follows. 

We also remark that if, in addition, $\sum_{j \in \mathbb{Z}} \|P_jf\|_\mathcal{H}^2 \lesssim \|f\|_\mathcal{H}^2$ (which, for example, trivially follows if we add the assumption that $P_k^*P_j$ is positive semi-definite for each $j,k \in \mathbb{Z}$) then \eqref{e:abstract_refinedStrichartz} implies the classical Strichartz estimate
\[
\| U_\phi f\|_{L^q_tL^r_x} \lesssim \|f\|_\mathcal{H}
\] 
via the trivial embedding $\ell^{2\beta} \subseteq \ell^2$. In this regard, \eqref{e:abstract_refinedStrichartz} represents a refined version of the classical Strichartz estimate for $\beta > 1$.

Clearly if $(P_j)_{j \in \mathbb{Z}}$ is the Littlewood--Paley family given by \eqref{e:LP}, then the properties (1)--(4) hold (with $N = 2$), and $P^*_kP_j$ is positive semi-definite for each $j,k$.

As our first application in this direction, for appropriate $(q,r,\beta,s)$ described in Theorem \ref{t:refined_fractional} below, we may use Theorems \ref{t:fracSchrosharp} and \ref{t:fracSchrononsharp} to deduce \eqref{e:abstract_refinedStrichartz} and consequently
\begin{equation} \label{e:refined_fractional}
\| e^{it(-\Delta)^{\alpha/2}} f \|_{L^q_tL^r_x} \lesssim \|f\|_{\dot{B}^s_{2,2\beta}}
\end{equation}
for all $f$ in the homogeneous Besov space $\dot{B}^s_{2,2\beta}$.
\begin{theorem} \label{t:refined_fractional}
Let $\alpha \in \mathbb{R} \setminus \{0,1\}$. Suppose $q,r \geq 2$ satisfy $\frac{d}{r} + \frac{\alpha}{q} < d$ and define $s = \frac{d}{2} - \frac{d}{r} - \frac{\alpha}{q}$.
\vspace{-7pt}
\begin{enumerate}
[leftmargin=.8cm, 
labelsep= 0.3 cm, topsep=0pt]
\item[$(i)\,$] Let $d \geq 2$. If $(\frac{1}{r},\frac{1}{q})$ belongs to $\mathrm{int}(OA_\frac{d}{2}C) \cup [C,A_\frac{d}{2})$, then \eqref{e:refined_fractional} holds with $\beta \leq \beta_\frac{d}{2}(q,r)$. If $(\frac{1}{r},\frac{1}{q})$ belongs to $\mathrm{int}(ODE_\frac{d}{2}A_\frac{d}{2}) \cup (O,A_\frac{d}{2}] \cup [A_\frac{d}{2},E_\frac{d}{2})$, then \eqref{e:refined_fractional} holds with $\beta < \frac{q}{2}$.
\item[$(ii)$] Let $d = 1$. If $(\frac{1}{r},\frac{1}{q})$ belongs to $\mathrm{int}(OA_\frac{1}{2}C) \cup [C,A_\frac{1}{2})$, then \eqref{e:refined_fractional} holds with $\beta \leq \beta_\frac{1}{2}(q,r)$.
\end{enumerate}
\end{theorem}
Estimates of the form \eqref{e:refined_fractional} have appeared several times in the literature for the fractional Schr\"odinger equation. As an example, in  work on the mass concentration phenomena for $L^2$-critical nonlinear fractional Schr\"odinger equations in \cite{CHL}, a key component of the proof (see equation (2.8) in \cite{CHL}) is establishing \eqref{e:refined_fractional} in the case $s=0$ for some $\beta > 1$. The argument in \cite{CHL} proceeds by establishing \eqref{e:refined_fractional} with $(q,r,\beta,s) = (4,4,2,\frac{d-\alpha}{4})$ (using bilinear arguments) and interpolation with standard Strichartz estimates (i.e. $\beta  = 1$). As a result, \eqref{e:refined_fractional} is obtained in the case of non-sharp $\frac{d}{2}$-admissible $(q,r)$ with $s=0$ and certain $\beta \in (1,2)$. Observe that the technical condition $\frac{d}{r} + \frac{\alpha}{q} < d$ in Theorem \ref{t:refined_fractional} is completely redundant for such $(q,r)$ and applying our result may lead to improvement in the range of $\beta$ for \eqref{e:refined_fractional}. For example, for $(\frac{1}{r},\frac{1}{q})$ on $(O,A_\frac{d}{2})$ with $s=0$, Theorem \ref{t:refined_fractional} gives \eqref{e:refined_fractional} with $\beta$ arbitrarily close to $\frac{d-1+\alpha}{d}$ and this exponent clearly grows unboundedly with $\alpha$. 

The estimate \eqref{e:refined_fractional} also appeared in \cite{BOQ} in their study on existence of extremizers for Fourier extension estimates associated to plane curves of the form $(\xi,|\xi|^\alpha)$, where $\alpha > 1$. In particular, it was shown that \eqref{e:refined_fractional} holds when $(d,q,r,\beta,s) = (1,6,6,\frac{3}{2},\frac{2-\alpha}{6})$; since $\beta_\frac{1}{2}(6,6) = \frac{3}{2}$, we completely recover this result from Theorem \ref{t:refined_fractional}\,$(ii)$.

For the wave equation, we obtain
\begin{equation} \label{e:refined_wave}
\| e^{it\sqrt{-\Delta}} f \|_{L^q_tL^r_x} \lesssim \|f\|_{\dot{B}^s_{2,2\beta}}
\end{equation}
as a consequence of Theorems \ref{t:wave} and \ref{t:wavenonsharp} combined with \eqref{e:abstract_refinedStrichartz}, where $(q,r,\beta,s)$ are prescribed in the following.
\begin{theorem} \label{t:refined_wave}
Suppose $q,r \geq 2$ and define $s = \frac{d}{2} - \frac{d}{r} - \frac{1}{q}$.
\vspace{-7pt}
\begin{enumerate}
[leftmargin=.8cm, 
labelsep= 0.3 cm, topsep=0pt]
\item[$(i)\,$] Let $d \geq 3$. If $(\frac{1}{r},\frac{1}{q})$ belongs to $\mathrm{int}(OA_\frac{d-1}{2}C) \cup [C,A_\frac{d-1}{2})$, then \eqref{e:refined_wave} holds with $\beta \leq \beta_\frac{d-1}{2}(q,r)$. If $(\frac{1}{r},\frac{1}{q})$ belongs to $\mathrm{int}(ODE_\frac{d-1}{2}A_\frac{d-1}{2}) \cup (O,A_\frac{d-1}{2}] \cup [A_\frac{d-1}{2},E_\frac{d-1}{2})$, then \eqref{e:refined_fractional} holds with $\beta < \frac{q}{2}$.
\item[$(ii)$] Let $d = 2$. If $(\frac{1}{r},\frac{1}{q})$ belongs to $\mathrm{int}(OA_\frac{1}{2}C) \cup [C,A_\frac{1}{2})$, then \eqref{e:refined_wave} holds with $\beta \leq \beta_\frac{1}{2}(q,r)$.
\end{enumerate}
\end{theorem}
Although the estimates in Theorem \ref{t:refined_wave} for $(\frac{1}{r},\frac{1}{q})$ belonging to $[C,B_\frac{d-1}{2}]$ may be obtained from Theorem \ref{t:waveFS} (due to Frank--Sabin), for such $q$ and $r$, the estimates \eqref{e:refined_wave} appeared earlier with coincidental or superior thresholds for $\beta$, and via different techniques. For example, Quilodr\'an \cite{Q} obtained \eqref{e:refined_wave} in the case $(d,q,r,\beta,s) = (2,6,6,\frac{3}{2},\frac{1}{2})$ in his work on sharp Fourier extension estimates for the two-dimensional cone at the Stein--Tomas exponent; in this instance, the threshold for $\beta$ is the same as in Theorem \ref{t:refined_wave}. In a significant development, Ramos \cite{Ramos} independently obtained the same estimate as in \cite{Q} when $d=2$, and\footnote{In fact, a yet further refinement of \eqref{e:refined_wave} in terms of conical sectors was proved in \cite{Ramos}.} \eqref{e:refined_wave} with $(q,r,\beta,s) = (\frac{2(d+1)}{d-1},\frac{2(d+1)}{d-1},\frac{d+1}{d-1},\frac{1}{2})$ for $d \geq 3$ using deep bilinear Fourier restriction theory; this is superior to the threshold $\beta = \frac{d+1}{d}$ given by Theorem \ref{t:refined_wave} in this special case. We remark that by the aforementioned result of Ramos and a trivial interpolation, one obtains \eqref{e:refined_wave} with $(\frac{1}{r},\frac{1}{q})$ on $[ (\frac{d-1}{2(d+1)},\frac{d-1}{2(d+1)}), E_{\frac{d-1}{2}}]$, $s = \frac{d+1}{4}(\frac{1}{2}-\frac{1}{r})$ and $\beta = \frac{q}{2}$; interestingly, from Theorem \ref{t:refined_wave}, we are essentially able to recover this result
on $[ A_{\frac{d-1}{2}}, E_{\frac{d-1}{2}}]$ since our theorem yields \eqref{e:refined_wave} with $\beta < \frac{q}{2}$ in such a case.

Finally, we note that Theorems \ref{t:KGSchro}, \ref{t:KGwave}, \ref{t:KGnonsharp_wave}--\ref{t:KGsharpunified} all lead to refinements of the Strichartz estimates for the Klein--Gordon equation 
\begin{equation*} 
\| e^{it\sqrt{1-\Delta}} f \|_{L^q_tL^r_x} \lesssim \|f\|_{B^s_{2,2\beta}},
\end{equation*}
where $\|f\|_{\dot{B}^s_{2,2\beta}}$ is the inhomogeneous Besov norm; the details of the statements are left to the interested reader. We also refer the reader to earlier work on refined Strichartz estimates for the Klein--Gordon propagator in \cite{Candy}, \cite{COS}, \cite{COSS}, \cite{KSV}.

\

\begin{acknowledgements}
This work was supported by JSPS Kakenhi grant numbers 18KK0073 and 19H01796  (Bez), Korean Research Foundation Grant  no.  NRF-2018R1A2B2006298 (Lee), and Grant-in-Aid for JSPS Research Fellow no. 17J01766 (Nakamura).  
The authors would like to thank  Younghun Hong and Yoshihiro Sawano for an earlier collaboration from which they benefited in the course of the current project.  
\end{acknowledgements}

\end{document}